\documentclass{amsart}

\RequirePackage{doi}
\usepackage{hyperref}

\usepackage[all]{xy}
\usepackage{amsmath}
\usepackage{amsfonts,graphics,amsthm,amsfonts,amscd,latexsym}
\usepackage{amssymb}
\usepackage{epsfig}
\usepackage{flafter}
\usepackage{mathtools}
\usepackage{comment}

\usepackage{cite} 

\hypersetup{
    colorlinks=true,    
    linkcolor=blue,          
    citecolor=blue,      
    filecolor=blue,      
    urlcolor=blue           
}

\usepackage{tikz}
\usetikzlibrary{graphs,positioning,arrows,shapes.misc,decorations.pathmorphing}
\tikzstyle{printersafe}=[decoration={snake,amplitude=0pt}]

\tikzset{
    >=stealth,
    every picture/.style={thick},
    graphs/every graph/.style={empty nodes},
}

\tikzstyle{vertex}=[
    draw,
    circle,
    fill=black,
    inner sep=1pt,
    minimum width=5pt,
]

\usepackage[position=top]{subfig}
\usepackage{color}

\calclayout

\newcommand{\rank}{\operatorname{rank}}

\newcommand{\supp}{\operatorname{supp}}

\newcommand{\ddivv}{\operatorname{div}}

\newcommand{\principal}{\mathrm{Prin}}

\newcommand{\cox}{\operatorname{Cox}}

\newcommand{\coxyx}{\cox(Y/X)}
\newcommand{\mcox}{\mathfrak{m}_{\cox}}
\newcommand{\compcox}{\widehat{\coxyx_{\mathfrak{m}_{\rm Cox}}}}
\newcommand{\scoxyx}{S_{\coxyx}}

\newcommand{\clyx}{{\rm Cl}(Y_x)}

\newcommand{\clxx}{{\rm Cl}(X_x)}

\newcommand{\clxxq}{{\rm Cl}_\qq(X_x)}

\newcommand{\relpicxz}{{\rm Cl}_\qq(X/Z)}

\newcommand{\relpicxprimez}{{\rm Cl}_\qq(X'/Z)}

\newcommand{\picnumbxz}{\dim_\qq \relpicxz}

\newcommand{\picnumbxprimez}{\dim_\qq \relpicxprimez}

\newcommand{\compR}{\widehat{R_{\mathfrak{m}_x}}}

\newcommand{\ringcone}{R(E, -E\vert_E)}

\newcommand{\ringconec}{R^{(c)}(E, -E\vert_E)}

\newcommand{\reetilde}{R(\widetilde{E}, \widetilde{E}\vert_{\widetilde{E}})}

\newcommand{\ree}{R(E, E\vert_E)}

\newcommand{\conev}{ \widehat{{\rm Cone}(E,-E|_E)_{v}}}

\newcommand{\conee}{{\rm Cone}(E,-E\vert_E)}

\newcommand{\conevtilde}{ \widehat{{\rm Cone}(\widetilde{E},-\widetilde{E}|_{\widetilde{E}})_{\tilde{v}}}}

\newcommand{\coneetilde}{{\rm Cone}(\widetilde{E},-\widetilde{E}|_{\widetilde{E}})}

\newcommand{\gammaxred}{\lceil \Gamma_{X}\rceil}

\newcommand{\gammaxredp}{\lceil \Gamma_{\widetilde{X}} \rceil}

\newcommand{\orbcomp}{\hat{c}}
\newcommand{\finecomp}{\overline{c}}
\newcommand{\abscompz}{c_z}
\newcommand{\abscompx}{c_x}

\newcommand{\pp}{\mathbb{P}}

\newcommand{\qq}{\mathbb{Q}}
\newcommand{\zz}{\mathbb{Z}}
\newcommand{\nn}{\mathbb{N}}
\newcommand{\rr}{\mathbb{R}}
\newcommand{\cc}{\kk}
\newcommand{\kk}{\mathbb{K}}

\def\O#1.{\mathcal {O}_{#1}}			
\def\pr #1.{\mathbb P^{#1}}				
\def\af #1.{\mathbb A^{#1}}			
\def\ses#1.#2.#3.{0\to #1\to #2\to #3 \to 0}	
\def\xrar#1.{\xrightarrow{#1}}			
\def\K#1.{K_{#1}}						
\def\bA#1.{\mathbf{A}_{#1}}			
\def\bM#1.{\mathbf{M}_{#1}}				
\def\bL#1.{\mathbf{L}_{#1}}				
\def\bB#1.{\mathbf{B}_{#1}}				
\def\bK#1.{\mathbf{K}_{#1}}			
\def\subs#1.{_{#1}}					
\def\sups#1.{^{#1}}						
\DeclareMathOperator{\coeff}{coeff}

\newtheorem{introthm}{Theorem}

\newtheorem{theorem}{Theorem}[section]
  \newtheorem{lemma}[theorem]{Lemma}
  \newtheorem{proposition}[theorem]{Proposition}
  \newtheorem{corollary}[theorem]{Corollary}

  \newtheorem{notation}[theorem]{Notation}
  \newtheorem{assumption}[theorem]{Assumption}
  \newtheorem{definition}[theorem]{Definition}
  \newtheorem{example}[theorem]{Example}

  \newtheorem{construction}[theorem]{Construction}

\newtheorem{remark}[theorem]{Remark}

\theoremstyle{remark}

\numberwithin{equation}{section}

\begin{document}

\title[A geometric characterization of toric singularities]{A geometric characterization of toric singularities}

\begin{abstract}
Given a projective contraction $\pi \colon X\rightarrow Z$ and a log canonical pair $(X, B)$ such that $-(K_X+B)$ is nef over a neighborhood of a closed point $z\in Z$, one can define an invariant, the complexity of $(X, B)$ over $z \in Z$, comparing the dimension of $X$ and the relative Picard number of $X/Z$ with the sum of the coefficients of those components of $B$ intersecting the fibre over $z$.
We prove that the complexity of $(X,B)$ over $z\in Z$ is non-negative and that when it is zero then $(X,\lfloor B \rfloor) \rightarrow Z$ is formally isomorphic to a morphism of toric varieties around $z\in Z$.
In particular, considering the case when $\pi$ is the identity morphism, we get a geometric characterization of singularities that are formally isomorphic to toric singularities.
This gives a positive answer to a conjecture due to Shokurov.
\end{abstract}

\author[J.~Moraga]{Joaqu\'in Moraga}
\address{Department of Mathematics, Princeton University, Princeton, NJ 08540, USA} 
\email{moraga@math.princeton.edu}

\author[R.~Svaldi]{Roberto Svaldi}
\address{EPFL, SB MATH-GE, MA B1 497 (B\^{a}timent MA), Station 8, CH-1015 Lausanne, Switzerland.}
\email{roberto.svaldi@epfl.ch}


\thanks{Part of this work was carried out during a visit of JM to the University of Cambridge, as well as a visit of RS to Princeton University.
The authors wish to thank University of Cambridge and Princeton University for the nice work environments, as well as Professors Caucher Birkar and Gabriele Di Cerbo for their hospitality and support.
RS gratefully acknowledges support from the European Union's Horizon 2020 research and innovation programme under the Marie Sk\l{}odowska-Curie grant agreement No. 842071.
}

\maketitle

\setcounter{tocdepth}{1} 
\tableofcontents

\section{Introduction}

Throughout this paper, we work over an algebraically closed field $\kk$ of characteristic zero.

A normal algebraic variety $X$ is a {\em toric variety} if it contains a dense open subset isomorphic to an algebraic torus $(\kk^\ast)^n$ such the natural self-action of the torus extends to the whole of $X$.
A morphism $X \to Y$ of toric varieties is a {\em toric morphism} if it is equivariant with respect to the torus actions on both $X$ and $Y$.

Toric morphisms are some of the simplest morphisms we can encounter in birational geometry:
indeed, they can be described just in terms of simple combinatorial data;
more precisely, a toric morphism $X \to Y$ is encoded by a linear transformation
between $\qq$-vector spaces that maps the fan structure inducing $X$ onto the one inducing $Y$ while preserving the cones composing both structures.
Toric morphisms are also easy to describe from an algebraic standpoint as they naturally correspond to monomial maps.
From a geometric standpoint, though, 
it is a difficult task to prove that a given morphism $X \to Z$ is toric:
indeed, this task boils down to either finding a dense of open set in $X$ isomorphic to a torus with respect to which the morphism is equivariant, or to finding local coordinates in which the morphism can be expressed in monomial form.

In~\cite{Sho00}, Shokurov unified these two viewpoints by proposing a characterization of formally toric morphisms.
A morphism $\pi \colon X \to Z$ is formally toric at a closed point $z \in Z$ if, when taking the base change of $X$ to the formal completion of $Z$ at $z$, the base change of $\pi$ becomes isomorphic to the completion of a morphism of toric varieties, see Definition~\ref{def:formally-toric} for more details. 
To formalize this characterization, Shokurov defines a new numerical invariant, {\it the
complexity}, in the relative setting, cf. Definition~\ref{complexity.def}:
given a projective contraction $\pi \colon X\rightarrow Z$ of quasi-projective normal varieties, a closed point $z \in Z$, and a pair $(X,B)$, $B= \sum b_i B_i$, the complexity of $(X,B)$ at $z\in Z$ is defined as
\[
c_z(X/Z,B):= 
\dim X + 
\picnumbxz - 
\sum_{B_i \cap \pi^{-1}(z) \neq \emptyset} b_i.
\]
Shokurov conjectured that $c_z(X/Z, B) \geq 0$, whenever $-(K_X+B)$ is nef in a neighborhood of $z \in Z$, and moreover that equality holds exactly when $X \to Z$ is formally toric at $z$.

The main goal of this article is to provide the following full solution to Shokurov's conjecture.

\begin{introthm}
\label{weak thm}
Let $(X/Z,B)$ be a log canonical pair over a normal variety $Z$ and let $z \in Z$ be a closed point. 
Assume that $-(K_X+B)$ is nef over a neighborhood of $z \in Z$.
Let $\Sigma$ be a decomposition of $B$. Then, 
\[
c_z(X/Z,B) \geq 
0.
\]
Furthermore, if the equality 
\[
c_z(X/Z,B)=0
\]
holds, then the following conditions are satisfied: 
\begin{enumerate}
    \item 
$K_X+B\sim_{\qq,Z} 0$;
    \item
$X\rightarrow Z$ is formally toric at $z$; and,     \item 
under the formal isomorphism of (2), the components of $\lfloor B \rfloor$ are mapped to the completion of toric invariant divisors.
\end{enumerate}
\end{introthm}

An important step in proving the above result is the following characterization of germs of singularities supporting log canonical pairs that are formally isomorphic to toric singularities, cf.~Definition~\ref{def:form.tor.sing}.
This result generalizes~\cite[Theorem~18.22]{Kol92} to singularities that may fail to be $\qq$-factorial.
For the germ of a normal singularity $x \in X$, we denote by ${\rm Cl}(X_x)$ the local class group at $x$.

\begin{introthm}
\label{thm:intro.local.case}
Let $x\in (X,B)$ be a log canonical singularity. 
Writing $B=\sum_{i=1}^n b_iB_i$ where the $b_i$ are positive and the $B_i$ are prime divisors, then, 
\[
\dim X +
\rank {\rm Cl}(X_x) -
\sum_{i=1}^n b_i 
\geq 0.
\]
If the equality holds, then $(X,\lfloor B\rfloor)$ is a formally toric pair at $x$.
\end{introthm}

Over the course of the years, several authors have worked on Shokurov's conjecture:

\begin{enumerate}
    \item 
in~\cite{Kol92}, Koll\'ar proved the conjecture
for $\qq$-factorial log canonical germs;

    \item 
in dimension two, M\textsuperscript{c}Kernan and Keel proved the conjecture for projective surfaces of Picard rank one,~\cite{KM99}, while, as already mentioned above, Shokurov proved the conjecture for arbitrary morphisms of surfaces;

    \item 
in an unpublished note, Chelstov proved the conjecture for $\qq$-factorial projective varieties of Picard number $1$;

    \item 
in~\cite{Pro01}, Prokhorov proved the conjecture for certain projective $3$-folds. 
The method of his proof relies on the minimal model program in dimension three;
    
    \item 
In~\cite{Yao13}, Yao gives a proof of the above conjecture for log smooth projective pairs $(X,D)$ with $K_X+D\sim_\qq 0$.
Yao's proof is inspired by the mirror-symmetry techniques of~\cite{GHS16};
    
    \item 
finally, in~\cite{BMSZ18}, the authors settle the conjecture in the non-relative case, i.e., for proper varieties mapping to a point. 
\end{enumerate}

Rather than proving Theorem~\ref{weak thm} directly, we shall prove a more general result providing a solution to Shokurov's conjecture in a wider setting.
The wider context we consider is that of generalized pairs, see \S\ref{subsection:mmp-gen-pairs}
.
The benefit of this choice is that we can switch from the category of log canonical pairs with anti-nef canonical divisor to that of generalized log canonical pairs which are $\qq$-trivial over the base, cf.~\eqref{eqn:log.cy.gen}.
While this choice may appear artificial, we want to underline how it makes for a much more advantageous approach:
for instance, the property of being generalized log canonical and $\qq$-trivial over the base is preserved under any birational contraction.

In~\cite{KM99}, M\textsuperscript{c}Kernan and Keel introduced a finer version of the complexity which
in this article we call {\em fine complexity}.
A decomposition $\Sigma$ of an effective Weil $\rr$-divisor $B$ is an effective formal sum $0 \leq \Sigma:=\sum_{i=1}^k b_i B_i \leq B$ where the $B_i$ are effective Weil divisors.
The fine complexity of $(X/Z,B)$ with respect to a decomposition $\Sigma$ of $B$ at a closed point $z \in Z$ is
\[
\finecomp_z(X/Z,B;\Sigma):= 
\dim X+
\dim_\qq \langle \Sigma /Z \rangle - |\Sigma|,
\]
where $\langle \Sigma / Z \rangle \subset \relpicxz$ denotes the span of the Weil divisors $B_i$ used to define $\Sigma$ and $|\Sigma|$ is the sum of the coefficients $b_i$ of those $B_i$ that intersect the fiber over $z$.
It should be clear that any boundary $B$ comes with a natural decomposition given by the decomposition into its prime components.
It is often beneficial, though, to work with other possible decompositions than just this standard one, a principle already exploited in~\cite{BMSZ18}.
Indeed, doing so allows for more flexibility, for example, when performing adjunction along the general fibre of a morphism.

Although the definition of complexity may appear to be more natural at a first glance, one of the advantages of instead using the fine complexity is that the latter is better-behaved under adjunction to a general fiber, as it was already explored in~\cite{BMSZ18}.
Unfortunately, it is hard to control the fine complexity of a pair $(X, B)$, in turn, when doing adjunction along a divisorial log canonical center $E$ of $B$.
The issue comes from codimension one points of $E$ contained in the singular locus of $X$:
at these points, the restriction to $E$ of a decomposition $\Sigma$ of $B$ will typically fail to yield a decomposition for the different of $B$ along $E$.
To overcome this issue, we define a new complexity-like invariant, the {\it orbifold complexity}, see Definition~\ref{def:orb.compl}, that is specifically designed to deal with the presence of orbifold structures at codimension one points of a divisorial lc center of a pair.
Theorem~\ref{formally.toric} provides an analogous version of Theorem~\ref{formally toric} in the more general framework of the orbifold complexity of $(X/Z, B, \mathbf M)$.

\begin{introthm}[cf.~Theorem~\ref{formally.toric}]
\label{formally toric}
Let $(X/Z,B, \mathbf M)$ be a generalized log canonical pair over $Z$.
Assume that
\begin{align}
\label{eqn:log.cy.gen}
K_X+B+\mathbf M_X\sim_{\qq, Z} 0.
\end{align}
Let $z\in Z$ be a closed point and let $\Sigma=\sum_{i=1}^k b_i B_i$ be a decomposition of $B$.
Then, 
\[
c_z(X/Z,B) \geq 
\finecomp_z(X/Z,B;\Sigma) \geq 
\orbcomp_z(X/Z. B; \Sigma)
\geq 0.
\]
Moreover,
if the equality 
\[ 
c_z(X/Z,B)=0
\]
holds, then the following conditions are satisfied: 
\begin{enumerate}
    \item 
$X\rightarrow Z$ is formally toric over $z\in Z$;
    
    \item 
the components of $\lfloor B\rfloor$ corresponds to toric invariant divisors via the formal isomorphism claimed in (1);
    
    \item 
all the prime divisors on $X$ corresponding to toric invariant divisors via the formal isomorphism in (1) appear in the support of $\Sigma$; and,

    \item 
the b-divisor $\mathbf M$ is a torsion b-divisor over a neighborhood of $z$.
\end{enumerate}
\end{introthm}

We emphasize that all the theorems of the paper also hold when replacing the relative Class group with the group of $\qq$-divisors modulo algebraic equivalence over the base. 

To conclude the introduction,
we wish to emphasize how the characterization of toric morphisms proven in this article is expected to play an important role in the birational classification of algebraic varieties.
One basic strategy to understand the structure of germs of singularities of the form $x \in (X, B)$, within the framework of the Minimal Model Program, is to construct simple partial resolutions which are not necessarily log smoothings of $(X, B)$, but that still carry enough structure to make the invariants of the germ $x \in X$ simpler to compute on the new model.
In general, it is expected that there exist special partial resolutions where the exceptional locus displays a toroidal structures -- such model is called a toroidalization of the given algebraic singularity.
In this sense, Theorem~\ref{thm:intro.local.case} shows that if the boundary $B$ has many components going through $x$ (weighted by their coefficient) then we do not really need to pass to a toroidalization, as the singularity $x \in X$ already has a toroidal structure.
In relation to these ideas, for example, in~\cite{Mor20b} the first author proves that a Fano type surface with a large cyclic automorphism admits a birational $\kk^\ast$-action. 
In~\cite{Mor20c}, the first author proves that a Fano type variety of dimension $n$ with a large finite automorphism group of rank $n$ is a compactification of $(\kk^\ast)^n$.
The characterization of toric varieties and morphisms using complexity, as developed in~\cite{BMSZ18} and in this article, is one of the main tools used to prove such result.
We expect that the characterization of formally toric morphisms Theorem~\ref{formally toric}, to have applications to the toroidalization of Fano type morphisms.

\subsection*{Acknowledgements} 
The authors would like to thank V.V. Shokurov for many useful comments on his conjecture and C. Spicer and M. Mauri for reading a first draft of the paper.


\section{Strategy of proof}

In this section, we give a brief description of the proof of the main result,Theorem~\ref{formally toric}, which is essentially divided in five steps.

In Section~\ref{sec:prel}, we introduce the concept of orbifold complexity, which refines both complexity and fine complexity.
Orbifold complexity does not increase when performing adjunction along a divisorial (generalized) log canonical center, cf. Lemma~\ref{lem:complexity-vs-adjunction}. 
This feature is our main motivation for introducing this new concept in the first instance.
As the orbifold complexity is always smaller than the fine complexity, Lemma~\ref{lem:basic.ineq}, for our purposes it will suffice to prove that the orbifold complexity is non-negative and a given morphism is formally toric when the orbifold complexity is zero.

In Section~\ref{sect:orb.absolute},
we show that Theorem~\ref{formally toric} holds for the orbifold complexity of projective generalized pairs with $\qq$-trivial generalized log divisor.
Essentially, the proof follows from the proof of~\cite[Theorem 1.2]{BMSZ18} with some technical adjustments.
First, we prove that by running a suitable MMP, we can construct a model on which the orbifold complexity does not increase and that is a Mori dream space.
Then, we show that such model is a projective toric variety in analogy with~\cite[Theorem 3.1]{BMSZ18}.
To conclude, we show that all the transformations we performed to construct such toric model correspond either to extractions of toric divisors, Lemma~\ref{lemm:X-toric}, or to isomorphisms in codimension one, which naturally preserve the toric structure.

The next step is then proved in Section~\ref{section:applications-projective}: we prove that the orbifold complexity is always non-negative. 
The main strategy is to reduce the general case to the projective case.
To this end, we show that it is possible to extract a divisorial glc center $E$
contained in the fiber over $z\in Z$, 
without increasing the orbifold complexity of a generalized pair $(X/Z,B+M)$.
The orbifold complexity of the generalized pair $(E, B_E+M_E)$ induced by adjunction along $E$ will also be at most that of $(X/Z, B+M)$.
If the orbifold complexity of $(X/Z,B+M)$ was negative, 
then the same would hold for the projective generalized log canonical pair $(E, B_E+M_E)$ which is $\qq$-trivial, since $E$ is contained in the fiber over $z$ and $K_X+B+M\sim_{\qq, Z}0$.
This contradicts the previous step. Hence, the orbifold complexity is always non-negative.
In order to extract $E$, we may have to pass to a suitable birational model of $X$, hence we need to keep track of the orbifold complexity when running the machinery of the MMP: this is explained in \S~\ref{subsection:complexity-vs-mmp}.
Moreover,
we also show that if a generalized pair $(X/Z,B+M)$ has orbifold complexity $0$ over $z\in Z$ then the fiber over $z \in Z$ must contain a glc center;
under this assumption,
a decomposition $\Sigma$ of $B$ realizing the $0$ of the orbifold complexity can only be the decomposition into the prime components of $B$.
These reductions are very useful in order to prove the problem in the general case.

We are then ready to prove the local case of the Theorem~\ref{formally.toric}, i.e., the case when $X \to Z$ is simply the identity morphism of a germ, cf. Section~\ref{section:proof-local}.
This case yields Theorem~\ref{thm:intro.local.case}.
Working on a germ of a glc singularity $(x \in X,B+M)$, we proceed to define a special class of birational extractions, formally toric plt blow-ups: these are generalizations of so-called plt blowups, see \S\ref{sect:plt.blow.ups}. 
A formally toric plt blow-up is a proper birational morphism $Y \to X$ that extracts a unique divisor $E$ over $x \in X$ which is toric and such that the torus invariant prime divisors of $E$ can be lifted to prime divisors $E_i$ on the plt blow-up with nice singularities over $x$.
We show that such partial resolution always exists for a germ of a generalized log canonical pair with orbifold complexity zero, cf. Proposition~\ref{prop:toroida-blow-up}.
We then analyze the relative Cox ring $\coxyx$ of a formally toric plt blow-up $Y\to X$: 
we show that the ideal $\mcox$ of $\coxyx$ generated by the canonical elements $x_i$ (resp. $e$) associated to the $E_i$ (resp. $e$) is maximal and it yields a smooth point in ${\rm Spec} \coxyx$. 
Upon passing to the completion of $\coxyx$ at $\mcox$, that is isomorphic to a power series ring in the variable $x_i, e$.
Moreover, these variables carry over the natural $\clyx$-grading defined in $\coxyx$.
We show that the completion of the local ring of $x \in X$ is isomorphic to subring $S$ of $\mathbb K[\![x_1, \dots, x_r, e]\!]$ of power series spanned by monomials in the $x_i, e$ which have degree zero with respect to the grading indicated above.
As the monomials in the $x_i$ of degree zero restricted to $E$ give all torus invariant sections for the rank 1 sheaves $\mathcal O_E(-lE\vert_E)$, $l \in \zz_{\geq 0}$, then we can show that ring $S$ is indeed isomorphic to the completion of local ring of the orbifold cone 
$\conee := 
{\rm Spec} 
\left(
\oplus_{l=0}^\infty 
H^0(E, \mathcal O_E(-lE \vert _E))
\right)$
at its vertex.

Finally, we are left with proving Theorem~\ref{formally.toric} when $X \to Z$ is not the identity.
If $X \to Z$ is birational, cf.~\S\ref{subsection:bir-case}, we will run certain relative MMP over the base in order to prove that the base of the birational morphism is formally toric.
Then, Lemma~\ref{formally toricity-vs-MMP} shows that the above MMP is formally toric, provided that all contracted divisors are formally toric log canonical centers.
This conditions on the contracted divisors will follow almost directly from the definition of complexity.
In the latter (subsection~\ref{subsection:fib-case}), we will consider the cone over the fibration to reduce to the local case. 
A simple computation using complexities proves that the cone over the fibration has a formally toric strucutre.
Finally, we use this structure on the cone to deduce that the starting fibration was formally toric.


\section{Preliminaries} 
\label{sec:prel}

In this section, we collect some preliminary notions and results.
For the basic notions and definitions on singularities and the Minimal Model Program, we refer the reader to~\cite{KM98} and~\cite{K13}.
Let us recall the definition of contraction.

\begin{definition}
{\em 
A {\em contraction} $\phi \colon X\rightarrow Z$ is a projective morphism of normal algebraic varieties such that $\pi_\ast \mathcal{O}_X=\mathcal{O}_Z$.
A {\em fibration} $\phi \colon X\rightarrow Z$ is a contraction with positive dimensional general fiber.
}
\end{definition}

When the context allows, we omit to write $\pi$ and instead write $X/Z$ to denote a quasi-projective normal variety $X$ endowed with a contraction to $Z$.
The datum $X/Z$ will also be referred to as a {\em relative variety over} $Z$.
If $Z$ is just a point, then we simply write $X$.

\subsection{Generalized pairs}
\label{subsection:gen-pairs}

The reader can find the definitions and basic properties of b-divisors in~\cite[\S2.3.2]{Corti}.

\begin{definition}
\label{gen.pair.def}
{\em 
A {\em generalized pair} is a triple $(X/Z,B,\mathbf{M})$, where 
$X$ is a normal variety relative over $Z$, 
$B$ is an effective divisor on $X$, 
$\mathbf{M}$ is a nef b-divisor over $Z$, and $K_X+B+\mathbf{M}_X$ is $\qq$-Cartier on $X$.
We call $B$ (resp. $\mathbf{M}$) the {\em boundary part} ({\em nef part}) of the generalized pair.
}
\end{definition}

If $\mathbf M=0$, i.e., $\mathcal{O}_X(M_X)=\mathcal{O}_X$ for each model $X$, then we will drop the word \textit{generalized} from the notation.
In this case, $(X, B)$ is just a log pair.
When the context is clear, we will utilize the lighter notation $(X/Z,B+M)$ to denote the generalized pair $(X/Z,B,\mathbf{M})$, where it should be understood that $M:=\mathbf{M}_X$.

In analogy with the classical case of log pairs, one can define log resolutions also for generalized pairs which then leads to the definition of generalized klt (in short, gklt) and generalized log canonical (glc) singularities, cf.~\cite[\S~4]{BZ16}.

\begin{definition}\label{def:gen-dlt}
{\em 
We say that a generalized log canonical pair $(X/Z,B+M)$ is {\em generalized divisorially log terminal},
    (in short, {\em gdlt}), if there exists an open set $U\subset X$ satisfying the following conditions
\begin{enumerate}
    \item 
the coefficients of $B$ are $\leq 1$;

    \item 
$U$ is smooth and $B|_U$ has simple normal crossing support; and,

    \item 
all the generalized non-klt centers of $(X/Z,B+M)$ intersect $U$ and are given by strata of $\lfloor B \rfloor$.
\end{enumerate}
}
\end{definition}

We now introduce the class of semi-log canonical generalized pairs.

Given a variety $X$, 
we denote by $\nu \colon X^\nu \rightarrow X$ the normalization morphism of $X$.
A b-divisor $\mathbf M$ is the datum of a b-divisor on the normalization $X^\nu$.
If $X$ is endowed with a morphism to a normal algebraic variety $Z$, the nefness of $\mathbf M$ is defined by regarding $X^\nu$ over $Z$ via the composition with $\nu$.
Moreover, if $X$ is demi-normal and $M_{X^\nu}$ does not contain any component of the support of the ramification locus $E_{X^\nu} \subset X$, then we can define $ M_{X}\coloneqq \nu_\ast M_{X^\nu}$, where $\nu_\ast$ is the divisorial pushforward.

\begin{definition}
\label{def:gen-slc}
{\em 
Let $X$ be a demi-normal irreducible quasi-projective variety over an algebraic variety $Z$ and let $\nu \colon X^\nu \to X$ be its normalization.
Let $B$ be an effective $\mathbb R$-divisor on $X$ such that $B$ does not contain the conductor of $X$ in its support.
Let $\mathbf M$ be a b-divisor on $X$ which is b-nef over $Z$ and assume that $ M_{X^\nu}$ does not contain any component of the support of the ramification locus $E_{X^\nu} \subset X$.
The pair $(X/Z,B+M)$ is a {\em generalized semi-log canonical pair}, if $K_X+B+M_X$ is $\qq$-Cartier and $(X^\nu, B_{X^\nu}+E_{X^\nu}+M_{X^\nu})$ is a generalized log canonical pair with boundary part $B_{X^\nu}+E_{X^\nu}$, and nef part $M_{X^\nu}$ such that
\[
K_{X^\nu}+B_{X^\nu}+E_{X^\nu}+M_{X^\nu} = \phi^\ast (K_X+B+M_X).
\]
}
\end{definition}

\subsection{Minimal Model Program}
\label{subsection:mmp-gen-pairs}
In this subsection, we will recall some classic results of the Minimal Model Program (in short, MMP) for generalized pairs.
We will also prove some preliminary results that will be useful in the proof of the main theorem.
We recall the following statement which is proved in~\cite[Lemma 4.4]{BZ16}. 

\begin{theorem}
\label{thm:mmp-gdlt}
Let $(X/Z,B+M)$ be a $\qq$-factorial gdlt pair.
Then, we may run a minimal model program with scaling of an ample divisor $A$ over $Z$.
Moreover, any generalized pair appearing in this run of the minimal model program is again gdlt.
\end{theorem}

\noindent 
Although we can run the above minimal model program for gdlt pairs, it is not known in general whether it terminates, see~\cite[Lemma 4.4]{BZ16} for conditions ensuring termination.

\begin{definition}
\label{def:gdlt.mod}
{\em 
A $\qq$-factorial {\em gdlt modification} of a generalized log canonical pair $(X/Z,B+M)$ is a projective birational morphism $\pi\colon Y \rightarrow X$ from a normal $\qq$-factorial variety $Y$ such that $\pi$ only extracts prime divisors $E_i$ of generalized log discrepancy $0$ for $(X/Z,B+M)$ and where the generalized pair $(Y/Z,B_Y+M_Y)$ obtained by log pullback is gdlt.
}
\end{definition}

The existence of gdlt modifications is proved in~\cite[Theorem 3.2]{Fil18}, see also~\cite[Theorem 2.9]{FS20}.

Given a generalized pair $(X/Z,B+M)$ and an $\mathbb{Q}$-Cartier divisor $A$ on $X$ ample over $Z$, we will denote by $(X/Z,B+M+A)$ the generalized pair whose nef part is given by the sum of the nef part of $(X, B+M)$ and the b-divisor $\bar{A}$ which is the Cartier closure of $A$.

\begin{lemma}\label{lemm:gen-with-ample-nef-part}
Let $(X/Z,B+M)$ be a generalized $\qq$-factorial dlt pair, and let $A$ be an ample divisor on $X$ over $Z$.
Consider the gdlt pair $(X/Z,B+M+A)$.
Then, there exists an effective $\mathbb R$-divisor $D$ on $X$ such that
\begin{enumerate}
    \item 
$D$ is big over $Z$,
    \item 
$(X,D)$ is a klt pair, and,
    \item
$K_X+B+M+A\sim_{\rr,Z} K_X+D$.
\end{enumerate}
In particular, if $K_X+B+M$ is pseudo-effective over $Z$, then $(X/Z,B+M+A)$ admits both a good minimal model and an ample model over $Z$.
\end{lemma}

\begin{proof}
We first prove the lemma when $(X/Z,B+M)$ is a generalized klt pair.
Let $\pi \colon X'\rightarrow X$ be a log resolution of $(X/Z,B+M+A)$. 
Thus, $M_{X'}+\bar{A}_{X'}$ is nef and big over $Z$. 
For every $k\in \mathbb{Z}_{>0}$, we can write $M_{X'}+\bar{A}_{X'} \sim_{\qq,Z} A'_k +\frac{E'}{k}$, where $A'_k$ is ample over $Z$, and $E'$ is a fixed effective $\qq$-divisor, see~\cite[Example 2.2.19]{Laz04a}.
Hence, choosing $k\gg 0$ and $A'_k$ general in its $\qq$-linear system, setting $B_{X'}$ to be the boundary part in the log pullback of $(X, B+M)$ to $X'$, the log pair
\[
\left(X',B_{X'}+A'_k+\frac{E'}{k} \right)
\]
is a klt sub-pair, and  $K_X'+B_{X'}+A'_k+\frac{E'}{k} \sim_{\mathbb Q, Z} K_{X'}+B_{X'}+M_{X'}+\bar{A}_{X'} \sim_{\mathbb Q, Z} 0$.
Moreover, every prime component of $B_{X'}+A'_k+\frac{E'}{k}$ with negative coefficient is exceptional over $X$.
Setting $\hat A\coloneqq \pi_\ast A'_k$, $E \coloneqq \pi_\ast \frac{E'}{k}$ on $X$, then $(X/Z,B+\hat A+E)$ is klt, and $\hat A+E\sim_{\qq,Z} M+A$. 
Setting $D:=A_1+E$, the last statement follows from~\cite[Theorem 1.1]{BCHM10} applied to $K_X+D$ over $Z$.
\\
We now assume that $(X/Z,B+M)$ is gdlt and $\qq$-factorial.
For $0< \epsilon \ll 1$, the pair
$(X/Z,B-\epsilon \lfloor B \rfloor +M)$ is generalized klt and
$A+\epsilon \lfloor B \rfloor$ is still ample.
Hence, we conclude by applying the generalized klt case to the generalized pair $(X, (B-\epsilon \lfloor B \rfloor)+M+(A+\epsilon \lfloor B \rfloor))$ of nef part $M+ \overline{(A+\epsilon \lfloor B \rfloor)}$.
\end{proof}

\begin{lemma}\label{lem:extraction-gslc-pair}
Let $(Y/Z,B_Y+M_Y)$ be a $\qq$-factorial gdlt pair and let $z\in Z$ be a closed point.
Assume that $K_Y+B_Y+M_Y\sim_{\qq,Z} 0$, and that $(Y/Z,B_Y+M_Y)$ has a divisorial glc center $E\subset \pi^{-1}(z)$.
Then, there exists a birational contraction over $Z$
\[
\xymatrix{
Y\ar@{-->}[rr] \ar[dr]_{\pi}& & Y' \ar[dl]^{\pi'} \\
& Z &
}
\] 
satisfying the following conditions: 
\begin{enumerate}
    \item 
$(Y',B_{Y'}+M_{Y'})$ is a generalized log canonical pair;

    \item 
the strict transform $E'$ of $E$ on $Y'$ is a divisorial glc center; 

    \item
${\pi'}^{-1}(z)=E'$; and,

    \item 
the generalized pair $(E', B_{E'}+M_{E'})$ defined by the adjunction formula along $E$ 
\[
(K_{Y'}+B_{Y'}+M_{Y'})|_{E'} = K_{E'}+B_{E'}+M_{E'}
\] 
is generalized semi-log canonical.
\end{enumerate}
\end{lemma}

\begin{proof}
We first prove properties (1)-(3).
Let $0< \epsilon \ll 1$ be a rational number.
By Theorem~\ref{thm:mmp-gdlt}, we can run a minimal model program for the gdlt pair $(Y,B_Y-\epsilon E+M_Y)$ with scaling of an ample divisor $A$ over $Z$
\begin{align}
\label{eqn:mmp.eqn}
\xymatrix{
Y=:Y_0 \ar@{-->}[r] \ar[drr]_{\pi=:\pi_0}& 
Y_1 \ar@{-->}[r] 
\ar[dr]^{\pi_1}& 
Y_2 \ar@{-->}[r] 
\ar[d]^{\pi_2}&
Y_3 \ar@{-->}[r] 
\ar[dl]^{\pi_3}&
\dots \\
& & Z & &
}
\end{align}
Since this MMP is $(K_Y+B_Y+M_Y)$-trivial over $Z$, 
then $(Y_i,B_{Y_i}+M_{Y_i})$ is $\qq$-factorial glc for all $i$, where $B_{Y_i}$ and $M_{Y_i}$ are the strict transforms of $B_Y$, $M_Y$, respectively.
As, by construction, the MMP in~\eqref{eqn:mmp.eqn} is also a $(-E)$-MMP, $E$ cannot be contracted at any of its steps.
Let $E_i$ be the strict transform of $E$ on $Y_i$.\\

\noindent
{\it Claim}. 
For $i \gg 0$, $E_i=\pi_i^{-1}(z)$.

\begin{proof}[Proof of the Claim]
If the MMP in~\eqref{eqn:mmp.eqn} terminates in finitely many steps, then the conclusion follows at once.
Hence, we can assume that such MMP does not terminate.
\\
Let $\lambda_i \geq 0$ be the positive real number such that $Y_i\dashrightarrow Y_{i+1}$
is a $(K_{Y_i}+B_{Y_i}-\epsilon E_i+M_{Y_i}+\lambda_i A_{Y_i})$-trivial birational map.
If $\lambda_\infty:=\lim_{i\rightarrow \infty}\lambda_i= >0$, then the MMP in~\eqref{eqn:mmp.eqn} is also a run of the MMP for the generalized pair $(Y,B-\epsilon E+M+\lambda_\infty A)$.
By Lemma~\ref{lemm:gen-with-ample-nef-part}, there exists $0\leq D \sim_{\qq,Z} B-\epsilon E+M+\lambda_\infty A$ such that $(Y/Z,D)$ is a klt pair with big boundary over $Z$.
Hence, the above minimal model program will terminate by~\cite{BCHM10} and for some $i$ large enough, $-E_i$ is nef over $Z$.
Thus, the fiber over $z$ equals $E_i$.
\\
Therefore, we may assume that the numbers $\lambda_i$ converge to zero and, moreover, that from a certain index $i_0$ onwards each step of the MMP is a flip.
Since every step of~\eqref{eqn:mmp.eqn} is $E$-positive, if $Y_i\dashrightarrow Y_{i+1}$ is a flip, the flipped locus is always contained in $E_{i+1}$.
Thus, the number $k_i$ of irreducible components of $\pi_i^{-1}(z)$ cannot increase.
If $k_i=1$ for some $i$, then we are done since $E_i$ is always an irreducible component of the fiber.
Moreover, $k_i > k_{i+1}$ decreases if a step 
of the MMP contains an irreducible component of the fiber in the flipping locus since the flipped locus will be contained in the strict transform of $E_{i+1}$.
\\
Hence, assuming $k_i >1$ for some $i$, we now show that $k_j < k_i$ for some $i > j$: this proves that, eventually, $k_i$ must be $1$.
\\
Let $S$ be a component of $\pi_i^{-1}(z)$ other than $E_i$.
Since the fiber is connected, we may assume that $S \cap E_i \neq \emptyset$.
Hence, through a general point $s\in S$ there exists a curve $C_s \subset \pi_i^{-1}(x)$ which intersects $E_i$ non-trivially and it is not contained in $E_i$.
In particular, $C_s \cdot (-E_i) < 0$, thus,
\begin{align*}
& C_s \subset \mathbf B_{-}(K_{Y_i}+B_{Y_i} -\epsilon E_i +M_{Y_i}/Z), \quad \text{and}
\\
& S\subset \mathbf B_{-}(K_{Y_i}+B_{Y_i}-\epsilon E_i +M_{Y_i}/Z).
\end{align*}
Then, for any $0< \lambda \ll 1$ 
\[
S\subset \mathbf B(K_{Y_i}+B_{Y_i}-\epsilon E_i + M_{Y_i}+\lambda A_{Y_i} /Z),
\]
where $\mathbf B(D)$ denotes the stable base locus of a divisor $D$.
Since $\lim_{i \rightarrow \infty}\lambda_i =0$, it follows that for some sufficiently large $j >i$, $Y_j$ is a minimal model for the pair $(Y_j,B_j-\epsilon E_j +M_j+\lambda A_j)$ over $Z$.
In particular, $S$ must have been contracted or was part of the flipping locus at some step $Y_k\dashrightarrow Y_{k+1}$, for $k\in \{i,\dots, j-1\}$.
This shows that $k_j < k_i$, as desired.
\end{proof}
By the above claim, after finitely many steps of above run of the MMP, we reach a model $Y_i$ where property $(3)$ holds. 
Setting $Y'\coloneqq Y_i$ and $E':=E_i$, then, $(Y',B_{Y'}+M_{Y'})$ is $\qq$-factorial and glc.
By construction, $(Y',B_{Y'}-\epsilon E' +M_{Y'})$ is a gdlt pair for $0< \epsilon \ll 1$. 
Hence, the pair $(E', \widetilde B_{E'})$ defined by the adjunction
\[
(K_{Y'}+B_{Y'})|_{E'} \sim_\qq K_{E'}+\widetilde B_{E'}
\]
is a semi-log canonical pair, see, e.g.,~\cite[Example 2.6]{FG14}.
We conclude that the pair obtained by generalized adjunction
\[
(K_{Y'}+B_{Y'}+M_{Y})|_{E'} 
\sim_\qq 
K_{E'}+B'_{E'}+M_{E'},
\]
is a generalized semi-log canonical pair $(E',B'_{E'}+M_{E'})$ since by the previous observation $E'$ is $S_2$.
\end{proof}

\begin{remark}\label{rem:gdlt}{\em 
Let $(Y',B_{Y'}+M_{Y'})$ be the generalized log canonical pair  constructed in Lemma~\ref{lem:extraction-gslc-pair}.
The generalized  pair $(Y',B_{Y'}+M_{Y'}-\epsilon E')$ is gdlt for any $0< \epsilon \ll 1$, where $E'$ is the divisor over $z \in Z$: in fact, $(Y',B_{Y'}+M_{Y'}-\epsilon E')$ is obtained by running the MMP in~\eqref{eqn:mmp.eqn} for the gdlt pair $(Y,B_{Y}+M_{Y}-\epsilon E)$.
}
\end{remark}

\begin{definition}{\em
Let $(X/Z,B+M)$ be a generalized klt pair.
We say that a birational contraction $\pi\colon Y \rightarrow X$ is a {\em small $\qq$-factorialization},
if $Y$ is $\qq$-factorial and $\pi$ does not extract any divisor. 
}
\end{definition}

In analogy with the classical setup of log pairs, one can prove the existence of some special gdlt modifications.
The following result is a generalization of~\cite[Theorem 3.2]{Fil18}.

\begin{lemma}\label{lem:existence-gen-dlt-mod}
Let $(X/Z, B+M)$ be a generalized log canonical pair over $Z$ and let $z \in Z$ be a closed point. 
Then there exists a $\qq$-factorial gdlt modification $\pi \colon Y\rightarrow X$ over $Z$.
Moreover, if $(X/Z,B+M)$ has a glc center contained in the fiber over $z$, 
then the glc modification can be constructed so that $(Y/Z,B_Y+M_Y)$ has a divisorial glc center mapping to $z$.
\end{lemma}

\begin{proof}
As we know that gdlt modifications exist,
we only need to prove the second part of the statement.
Let $\pi' \colon X'\rightarrow X$ be a log resolution of $(X/Z,B+M)$.
By taking $X'$ to be a sufficiently high model, we may assume that $\pi'$ extracts a prime divisor $E$ of log discrepancy $0$ mapping to $z \in Z$.
Setting $B':={\rm Exc}(\pi')+\pi_\ast ^{'-1}(B)$, then the generalized pair $(X',B'+M_{X'})$ is gdlt and 
\[
K_{X'}+B'+M_{X'} \sim_{\qq,X} \sum_F a_F(X/Z,B+M)F,
\]
where the sum runs over all prime divisors $F$ on $X'$ exceptional over $X$ such that the log discrepancy $a_F(X,B+M)$ is non-negative - this follows as in the proof of Lemma~\ref{lem:extraction-gslc-pair}.
By the negativity lemma, every prime divisor $F$ that is exceptional over $X$ and such that $a_F(X/Z,B+M)>0$ is contained in the relative diminished base locus of $K_{X'}+B'+M_{X'}$ over $X$.
By Theorem~\ref{thm:mmp-gdlt}, we may run a $(K_X'+B'+M_{X'})$-MMP over $X$ with scaling of an ample divisor.
After finitely many steps of this MMP, all the divisors with $a_F(X/Z,B+M)>0$ are contracted, since they are contained in the diminished base locus.
Let $Y$ be one such model.
Hence, $(Y/Z,B_Y+M_Y)$ is a $\qq$-factorial gdlt modification of $X$.
In particular, the strict transform of $E$ on $Y$ is a divisor.
\end{proof}

\subsection{Complexity}
In this subsection, we recall the definitions and main properties of complexity and fine complexity that were introduced in~\cite{BMSZ18};
we also introduce the orbifold complexity.

\subsubsection{Complexity in the relative setting}
We extend the notion of complexity for a log pair to the relative setting and we study some of its properties in this context.

\begin{notation}
\label{not:localizations}
{\em 
Given a normal variety $Z$ and a (closed) point $z \in Z$, we denote by $Z_z:={\rm Spec}(\mathcal O_{Z, z})$, where $\mathcal O_{Z, z}$ is the local ring of $Z$ at $z$.
For a morphism $X\rightarrow Z$, we will denote by $X_z\rightarrow Z_z$ the base change of $X\rightarrow Z$ to $Z_z$.
}
\end{notation}

\begin{definition}
\label{def:cl.gr}
{\em 
Let $X, Z$ be a quasi-projective normal variety.
\begin{enumerate}
    \item 
For a given contraction $\phi \colon X\rightarrow Z$ the class group ${\rm Cl}(X/Z)$ of $X$ over $Z$  is the group of Weil $\qq$-divisors on $X$ modulo the subgroup generated by linear equivalence and by $\phi^\ast {\rm Pic}(Z)$.

    \item 
The $\qq$-class group $\relpicxz$ of $X$ over $Z$ is $\relpicxz:= {\rm Cl}(X/Z) \otimes \qq$.

    \item
Let $x \in X$ be a closed point.
The local class group ${\rm Cl}(X_x)$ at $x \in X$ is the group of Weil divisors of $X$ modulo divisors that are principal in a neighborhood of $x$.
    \item 
The $\qq$-local class group at $x \in X$ is ${\rm Cl}_\qq(X_x):={\rm Cl}(X_x) \otimes \qq$.
\end{enumerate}
}
\end{definition}

Given two Weil $\qq$-divisors $D_1, D_2$ on $X$, then $[D_1]=[D_2] \in \relpicxz$ if there exists a Cartier divisor $H$ on $Z$ so that $mD_1\sim mD_2 + \phi^\ast H$, for some positive integer $m$.

We recall the following definitions that originally appeared in~\cite{BMSZ18}.

\begin{definition}\label{def:complexity}
{\em
Let $X/Z$ be a 
normal variety over $Z$. 
Let $B$ be an effective divisor on $X$.
\begin{enumerate}
    \item 
A {\em decomposition} $\Sigma$ of $B$ is a finite formal sum of the form
\begin{equation}
\label{decomp.def.eqn}
0 \leq \Sigma=\sum_{i=1}^k b_i B_i \leq B,    
\end{equation}
where the $B_i$ are effective Weil divisors (possibly reducible) and for all $i$, $b_i \in \mathbb{R}_{>0}$.

    \item
The {\em norm} $|\Sigma|$ of a decomposition $\Sigma$ of $B$ is $|\Sigma|:=\sum_{i=1}^k b_i$.
    \item 
The {\em span} of a decomposition $\Sigma$ of $B$ (relative to $Z$) is defined as
\[
\langle \Sigma / Z \rangle :=
{\rm span}\langle [B_i] \mid i\in \{1,\dots, k\} \rangle \subseteq \relpicxz.
\] 
    \item 
Given a closed point $z\in Z$, we say that a decomposition $\Sigma$ of $B$ is {\em supported at} $z$ if all the divisors $B_i$ in~\eqref{decomp.def.eqn} intersect the fiber over $z$.
\end{enumerate}
}
\end{definition}

We recall the classic concepts of complexity, fine complexity, and local complexity.

\begin{definition}
\label{complexity.def}
{\em
Let $X/Z$ be a 
normal variety over $Z$. 
Let $z \in Z$ be a closed point.
Let $B$ be an effective divisor on $X$ and let $\Sigma$ be a decomposition of $B$ supported at $z$.
\begin{enumerate}
    \item 
The {\em complexity} of $\Sigma$ at $z\in Z$ is
\[
\abscompz(X/Z,B;\Sigma):=
\dim X +
\picnumbxz  -|\Sigma|.
\]

    \item 
The {\em fine complexity} of $\Sigma$ at $z\in Z$ is
\[
\finecomp_z(X/Z,B;\Sigma):=
\dim X +
\dim_\qq\langle \Sigma /Z \rangle - |\Sigma|.
\]
  
    \item 
The fine complexity $\finecomp_z(X/Z,B)$ (resp. the complexity $c_z(X/Z,B)$) of $(X/Z,B)$ at $z\in Z$ 
is the infimum, over the set of all possible decompositions $\Sigma$ of $B$ supported at $z\in Z$ of $\finecomp_z(X/Z, B; \Sigma)$ (resp. of $c_z(X/Z, B; \Sigma)$).
\end{enumerate}
}
\end{definition}

\begin{remark}
{\em 
\begin{enumerate}
	\item 
The complexity (resp. fine complexity) is a local invariant of the morphism $X \to Z$ around $z \in Z$. 
If we replace $Z$ with an open neighborhood of $z \in Z$, then $\picnumbxz$ (resp. $\dim_{\qq}\langle \Sigma/Z \rangle$) can only decrease and it achieves its minimum value on a sufficiently small open neighborhood of $z \in Z$.
Alternatively, one could also substitute $\relpicxz$ (resp. $\langle \Sigma/Z \rangle$) with ${\rm Cl}_\qq(X_z)$ (resp. with its image inside ${\rm Cl}_\qq(X_z)$) in the definition of complexity (resp. fine complexity).
	\item 
Since for any decomposition $\Sigma$ of an effective $\mathbb R$-divisor $\picnumbxz \geq \dim_\qq \langle \Sigma / Z \rangle$, there is an obvious inequality
\begin{align}
\label{eqn:ineq.compl.vs.abs.compl}
c_z(X/Z,B) 
\geq 
\finecomp_z(X/Z,B).
\end{align}
\end{enumerate}
}
\end{remark}

Let us observe that Definition~\ref{complexity.def} makes sense for any pair $(X, B)$ with $B$ effective even when the $\qq$-divisor $K_X+B$ is not $\qq$-Cartier.
Thus, we can then proceed to define complexity and fine complexity also for generalized pairs as follows.

\begin{definition}
\label{def:compl.gen.pair}
{\em
Let $(X/Z,B+M)$ be a generalized pair over $Z$. 
Let $z\in Z$ be a closed point.
\begin{enumerate}
    \item 
The complexity $\abscompz(X/Z,B+M)$ of $(X/Z,B+M)$ at $z\in Z$ is
\[
\abscompz(X/Z,B+M):= \abscompz(X/Z,B).
\]
    \item
The fine complexity $\finecomp_z(X/Z,B+M)$ of $(X/Z,B+M)$ at $z\in Z$ is
\[
\finecomp_z(X/Z,B+M):= \finecomp_z(X/Z,B).
\]
\end{enumerate}
}
\end{definition}

\subsubsection{Orbifold structures and orbifold complexity}

In this subsection, we introduce the notion of orbifold complexity.
Given a quasi-projective variety $X$, we denote by $X^1$ the set of codimension one points of $X$.
Given $P \in X^1$, if $D=aP + D'$, with $a \in \mathbb Q^\ast$ and $P \not \subset {\rm Supp} \ D'$, then the Cartier index of $D$ at $P$ is the minimal $n \in \zz_{>0}$ such that $na \in \zz_{>0}$.

\begin{definition}
\label{def.orb.struct}
{\em 
Let $X$ be a normal variety and let $B$ be a $\mathbb R$-divisor on $X$.
\begin{enumerate}
    \item 
An {\em orbifold structure} on $X$ is a function 
\[ 
n \colon X^1 \rightarrow \zz_{>0}
\]
such that $n(P)>1$ for only finitely many $P \in X^1$.
    \item 
An orbifold structure $n$ is {\it trivial} if $n(P) =1$ for all $P \in X^{(1)}$;
it is {\it non-trivial}, otherwise.
    \item 
The {\em support} of an orbifold structure $n$ is the union of all the prime divisors $P_i$ with $n_{P_i}>1$.
    \item
An orbifold structure $n$ is compatible with $B$ if the support of $n$ is contained in $\supp(B)$.
\end{enumerate}
}
\end{definition}

Let $n \colon X^1 \rightarrow \zz_{>0}$ be an orbifold structure and let $P_1, \dots, P_k$ be its support.
We denote by $n_P$ the value $n$ at $P \in X^{1}$ and we refer to it as the {\em orbifold} index of $n$ at $P$.
The orbifold structure $n$ is completely determined by the pair $((P_1, \dots, P_k), (n_{P_1}, \dots, n_{P_k}))$, up to re-ordering the $P_i$ and the $n_{P_i}$. 
We will write $n=((P_1, \dots, P_k), (n_{P_1}, \dots, n_{P_k}))$.

\begin{definition}
\label{def.orb.W.div}
{\em
Let $X$ be a normal variety and let $n=((P_1, \dots, P_k), (n_{P_1}, \dots, n_{P_k}))$ be  an orbifold structure on $X$.
A $\qq$-divisor $D$ is said to be an {\em orbifold Weil divisor} $D$ for the orbifold structure $n$ on $X$ if the Cartier index of $D$ at any codimension one point $P$ divides $n_P$.
}
\end{definition} 

When the orbifold structure $n$ is clear from the context, we will simply say that a $\mathbb Q$-divisor $D$ satisfying the properties of Definition~\ref{def.orb.W.div} is an orbifold Weil divisor. 

\begin{remark}
\label{rmk:coeff.orb.struct.canonical}
{\em
Let $X$ be a normal variety.
\begin{enumerate}
\item 
If we only consider the trivial orbifold structure on $X$, then an orbifold Weil divisor on $X$ is simply a Weil divisor on $X$.
    \item
    \label{canonical.expr}
Given an orbifold structure $n=((P_1, \dots, P_k), (n_{P_1}, \dots, n_{P_k}))$ and effective orbifold Weil divisor $D$ for $n$ on $X$, then we can write in a unique way
\[
D=\sum_{P \in X^1} \frac{{\rm num}_P(D)}{n_P}P,
\quad
{\rm num}_P(D) \in \mathbb Z.
\]
We refer to this unique expression as the {\em canonical expression} of the effective orbifold Weil divisor $D$.
\end{enumerate}
}
\end{remark}

\begin{remark}{\em 
\label{rmk:orb.struct.restr}
Let $X$ be a normal variety and let $n:=((P_1, \dots, P_k), (n_{P_1}, \dots, n_{P_k}))$ be an orbifold structure on $X$.
\begin{enumerate}
    \item 
    \label{rmk:orb.struct.bir.contr}
Given a birational contraction $\phi \colon X \dashrightarrow Y$ then there is a natural induced orbifold structure $n'$ on $Y$, defined by $n'_{\phi_\ast P}:=n_P$, for $P \not \subset {\rm exc}(\phi)$.
   
    \item
Analogously, for a proper contraction $f \colon X \to Z$ with $\dim X > \dim Z$, there is an induced orbifold structure $n''$ on the general fiber $F$ of $f$ supported at those prime divisors of $F$ contained in $\cup_{i=1}^k \supp (F \cap P_i)$.
For a prime divisor $Q \subset P_i \cap F$, then $n''_Q:= n_{P_i}$.
\end{enumerate}
}
\end{remark}

\begin{definition}
\label{def:orb.decomp}
{\em  
Let $X \to Z$ be a normal variety over $Z$.
Let $B$ be an effective divisor on $X$.
Let $n=((P_1, \dots, P_k), (n_{P_1}, \dots, n_{P_k}))$ be an orbifold structure on $X$ compatible with $B$.
\begin{enumerate}
    \item 
An {\em orbifold decomposition} $\Sigma$ of $B$ with respect to $n$ is a formal finite sum of the form
\[
0\leq \Sigma := 
\sum_{P \in X^1} \left( 1-\frac{1}{n_{P}}\right) P+\sum_{j=1}^s b_j B_j \leq B,
\]
where for all $j=1,2, \dots, s$, $B_j$ is an effective orbifold Weil divisors on $X$ and $b_j \in \mathbb R_{\geq 0}$.
    \item 
The {\em norm} $|\Sigma|$ of the orbifold decomposition $\Sigma$ of $B$ is $|\Sigma|:=\sum_{j=1}^s b_i$.
    \item 
The {\em span} of the orbifold decomposition $\Sigma$ of $B$ (relative to $Z$) is
\[
\langle \Sigma/Z\rangle := \langle
[B_j] \mid j\in \{1,\dots,s\}\rangle \subseteq \relpicxz.
\]
    \item 
Given a point $z\in Z$, we say that $\Sigma$ is {\em supported at $z$} if all the effective divisors $B_j$ intersect the fiber over $z$.
\end{enumerate}
}
\end{definition}

Let us observe that for any given decomposition (resp. orbifold decomposition) $\Sigma$ of $B$, we can restrict to a neighborhood of $z\in Z$ over which the given decomposition (resp. orbifold decomposition) is supported at $z\in Z$.

In this paper, whenever we fix a relative variety $X/Z$ and a closed point $z\in Z$, we will always assume that any decomposition of a divisor $B$ on $X$ is supported at $z$.

\begin{remark}
\label{rem:coeff.orb.decomp}
{\em 
Given a pair $(X,B)$ and a non-trivial orbifold structure $n=((P_1, \dots, P_k), (n_{P_1}, \dots, n_{P_k}))$, if $B$ admits an orbifold decomposition with respect to $n$, then ${\rm coeff}_{P_i}(B)\geq 1-\frac{1}{n_{P_i}}$, $i=1, \dots, k$.
}
\end{remark}

\begin{definition}
\label{def:orb.compl}
{\em
Let $(X/Z,B+M)$ be a generalized log pair over $Z$. 
Let $z \in Z$ be a closed point.
\begin{enumerate}
    \item 
Let $n$ be an orbifold structure on $X$ and let $\Sigma$ be an orbifold decomposition with respect to $n$ of $B$ supported at $z$.
The {\em orbifold complexity} of $\Sigma$ at $z\in Z$ is
\[
\orbcomp_z(X/Z,B+M;\Sigma):=\dim X +\dim_\qq \langle \Sigma/Z\rangle -|\Sigma|.
\]
    \item   
The orbifold complexity $\orbcomp_z(X/Z,B+M)$ of $(X/Z,B+M)$ at $z\in Z$ is
the infimum of all the orbifold complexities $\orbcomp_z(X/Z,B;\Sigma)$ computed among all possible orbifold decompositions $\Sigma$ of $B$ supported at $z\in Z$ (with respect to all possible orbifold structures on $X$ compatible with $B$).
\end{enumerate}
}
\end{definition} 

When $Z$ is just a point, then we simply drop $Z$ and $z$ from the notation and simply write $c(X, B+M)$ and, similarly, for the fine and orbifold complexities.

We have the following simple inequalities among the different types of complexity introduced so far.
\begin{lemma}
\label{lem:basic.ineq}
Let $(X/Z,B+M)$ be a generalized pair over $Z$. 
Let $z\in Z$ be a closed point.
Then,
\[
\abscompz(X/Z,B+M)\geq
\finecomp_z(X/Z,B+M)\geq
\orbcomp_z(X/Z,B+M).
\]
\end{lemma}

\subsubsection{Local complexities}

\begin{definition}
\label{def:orb.decomp.local}
{\em  
Let $x\in X$ be a germ of a normal variety.
Let $B$ be an effective divisor on $x \in X$.
Let $n=((P_1, \dots, P_k), (n_{P_1}, \dots, n_{P_k}))$ be an orbifold structure on $X$ compatible with $B$ such that $x \in P_i$ for all $i$.
Let 
\[
\Sigma := 
\sum_{P \in X^1} \left( 1-\frac{1}{n_{P}}\right) P+\sum_{j=1}^s b_j B_j \leq B,
\]
be an orbifold decomposition of $B$ for $n$.
\begin{enumerate}
    \item 
The {\em span} of the orbifold decomposition $\Sigma$ of $B$ is
\[
\langle \Sigma\rangle := \langle
[B_j] \mid j\in \{1,\dots,s\}\rangle \subseteq \clxxq.
\]
    \item 
We say that $\Sigma$ is {\em supported at $x$} if all the effective divisors $B_j$ contain $x$.
\end{enumerate}
}
\end{definition}

In this paper, whenever we fix a germ $x \in X$, we will always assume that any orbifold structure and any orbifold decomposition of a divisor $B$ we consider are supported at $x$.

\begin{definition}
\label{def:orb.compl.loc}
{\em
Let $(x \in X,B+M)$ be the germ of a generalized log pair. 
\begin{enumerate}
    \item 
Let $n$ be an orbifold structure on $X$ and let $\Sigma$ be an orbifold decomposition with respect to $n$ of $B$ supported at $x$.
The {\em orbifold complexity} of $\Sigma$ at $x \in X$ is
\[
\orbcomp_x(X,B+M;\Sigma):=\dim X +\dim_\qq \langle \Sigma\rangle -|\Sigma|.
\]

    \item
If $n$ is the trivial orbifold structure and $\Sigma$ is a decomposition of $B$ supported at $x$.
The {\em fine complexity} (resp. {\em complexity}) of $\Sigma$ at $x \in X$ is
\begin{align*}
& \finecomp_x(X,B+M;\Sigma):=\dim X +
\dim_\qq \langle \Sigma\rangle 
-|\Sigma|
\\
(\text{resp.} \
&\abscompx(X,B+M;\Sigma):=
\dim X +
\dim_\qq \clxxq 
-|\Sigma|)
.
\end{align*}

    \item   
The orbifold complexity $\orbcomp_x(X,B+M)$ of $(X,B+M)$ at $x \in X$ is
the infimum of $\orbcomp_x(X,B;\Sigma)$ taken over all possible orbifold decompositions $\Sigma$ of $B$ supported at $x$ with respect to all possible orbifold structures on $X$ compatible with $B$ and supported at $x$.

  \item   
The fine complexity (resp. complexity) $\finecomp_x(X,B+M)$ ($\abscompx(X, B+M)$) of $(X,B+M)$ at $x \in X$ is
the infimum of $\finecomp_x(X,B;\Sigma)$
(resp. of $\abscompx(X,B;\Sigma)$)
taken over all decompositions $\Sigma$ of $B$ supported at $x$.
\end{enumerate}
}
\end{definition} 

It is immediate from the above definition that the analog of Lemma~\ref{lem:basic.ineq} holds, in the framework of germs of generalized log pairs, for the notion of complexity defined in~\ref{def:orb.compl.loc}.

We also define a notion of complexity for general germs of normal singularities.
Although this was not defined in earlier literature in this generality, it had been already considered, cf.~\cite[Theorem 18.22]{Kol92}, at least for $\mathbb Q$-factorial germs.

\begin{definition}
\label{def:loc.compl}
{\em 
The {\em local complexity} $c^{\rm loc}(x \in X)$ of a germ of a normal variety $x \in X$ is
\begin{align*}
    c^{\rm loc}(x \in X)
:= 
\inf\Bigg\{
&
\dim X + 
\rank {\rm Cl}(X_x) - 
\sum_{i} a_i 
\Bigg\},
\end{align*}
where the infimum is taken among all log pair structures $(X,B)$ which are lc in a neighborhood of $x$ and 
$B=\sum_i a_i B_i$ is the decomposition into prime components.
}
\end{definition}

\begin{remark}
\label{rmk:local.complex}
{\em 
Let $X$ be a normal variety and let $B$ be an effective divisor on $X$.
\begin{enumerate}
    \item 
\label{rmk:local.complex.f.g.}
In general, the local class group $\clxx$ is infinite exactly at those point $x \in X$ where $X$ is not $\qq$-factorial.
If $(X,B)$ is a klt pair at a point $x\in X$, then $\clxx$ is finitely generated for every closed point $x\in X$, see, e.g.,~\cite[Theorem 3.27]{BM21}.

    \item 
\label{rmk:local.complex.q.fact}
By~\cite[Theorem 18.22]{Kol92} the local complexity of a germ of a log canonical $\mathbb{Q}$-factorial singularity is always non-negative.
Moreover, if the local complexity is $0$ then the singularity is formally a toric one, cf.~Definition~\ref{def:formally-toric} and Lemma~\ref{lem:toric-complexity=0}.
\end{enumerate}
}
\end{remark}

\subsubsection{Basic properties of complexity,
fine complexity, and orbifold complexity}

\begin{proposition}\label{prop:complexity-minimum}
Let $(X/Z, B+M)$ be a generalized pair and let $z \in Z$ be a closed point.
The orbifold complexity $\orbcomp_z(X/Z,B+M)$ is computed by some orbifold $\Sigma$ of $B$.
The same result holds if we substitute the orbifold complexity with either fine complexity or complexity.
\\
The same result also holds for the orbifold complexity, fine complexity, and complexity of a germ of a generalized log pair.
\end{proposition}

\begin{proof}
It is easy to see that $\abscompz(X/Z,B;\Sigma)$ is minimized by taking $\Sigma$ to be the decomposition of $B$ into its prime components as $\abscompz(X/Z,B+M;\Sigma)$ only depends on $|\Sigma|$.
Indeed, let $\Sigma_i=\sum_{j=1}^{k_i} b_{j,i} B_{j,i} \leq B$ be a sequence of decompositions of $B$ such that
$c_i := \abscompz(X/Z,B+M;\Sigma_i)$
is a strictly decreasing sequence converging $\abscompz(X/Z,B)$. 
For all $i$, the Weil divisors $B_{j,i}$ are pairwise different and $b_{j,i}>0$, $j=1, \dots, k_i$.
Since $B_{j,i}\leq \lceil B \rceil$, there are only finitely many possible Weil divisors $B_{j,i}$.
Hence, up to passing to a subsequence of the $\Sigma_i$, we may assume that $k_{i_1}=k_{i_2}=k$ and $B_{j,i_1}=B_{j,i_2}$ for any two positive integers $i_1$, $i_2$, and $j=1, \dots, k$.
Thus, for each $i_2\geq i_1$, 
\[
|\Sigma_{i_2}|=\sum_{j=1}^{k_{i_2}} b_{j,i_2}>\sum_{j=1}^{k_{i_1}} b_{j,i_1}=|\Sigma_{i_1}|.
\]
Passing to a subsequence again, we may assume that for each $j$ the limit $\lim_{i\rightarrow \infty} b_{j,i}=b_{j,\infty}$ is well-defined.
Thus, setting $\Sigma_\infty = \sum_{j=1}^{k_1} b_{j,\infty} B_{j,i} \leq B$ it immediately follows that
$c_\infty := \abscompz(X/Z,B;\Sigma_\infty)\leq c_i$,
for each $i$ so that $\abscompz(X/Z,B+M;\Sigma_\infty)=\abscompz(X/Z,B+M)$.
\\
The proof in the case of the fine complexity is the same once we observe that
$\dim_\qq \langle \Sigma/Z\rangle$ can only assume finitely many values.
\\
In the case of orbifold complexity, 
let $\Sigma_i$ be a sequence of orbifold decompositions $B$ whose orbifold complexities converge monotonically to $\orbcomp_z(X/Z,B+M)$.
It suffices to prove that there are only finitely many possible orbifold structures in this sequence, as then the proof exactly proceeds as above.
By Remark~\ref{rem:coeff.orb.decomp}, the support of a orbifold structure $n_i$ supporting $\Sigma_i$ is contained in the support of $B$.
Moreover, if $P$ is a prime divisor such that $0< {\rm coeff}_P(B)< 1$, then $n_i(P)$ can only assume finitely many values.
If $\coeff_P(B)=1$ and for a given orbifold structure $n_i$, we can define a new orbifold structure $n_{i}'$ by
\[
n_{i}'(Q)=
\begin{cases}
1 & \text{for }Q=P
\\
n_{i}(Q) & \text{for } Q \neq P.
\end{cases}
\]
We also define a new orbifold decomposition $\Sigma'_i$
\[
\Sigma'_i:= 
\sum_{Q \in X^1}
\left( 1-\frac{1}{n'_i(Q)}\right) Q
+
P+
\sum_{j=1}^{k_i} b_i
(B_{i, j} - \coeff_P(B_{
i, j})P)
\leq B
\]
where the $B_{i, j}$ are the orbifold Weil divisors in the orbifold decomposition $\Sigma_i$.
It is immediate from the definition that 
\[
\orbcomp_z(X/Z, B+M; \Sigma_i')
\leq 
\orbcomp_z(X/Z, B+M; \Sigma_i).
\]
Hence, by repeating this process for any component of $\lfloor B \rfloor$, and at each step substituting $n_i$ (resp. $\Sigma_i$) with the $n'_i$ (resp. $\Sigma'_i$) constructed above, 
we can assume that $n_i$ is supported solely at prime divisors not in $\lfloor B \rfloor$ at which point we are done from the above observations.
\\
Exactly the same proof holds also in the context of germs of generalized log pairs.
\end{proof}

\subsection{Complexity under the Minimal Model Program}\label{subsection:complexity-vs-mmp}

In this subsection, we collect some basic results describing the behavior of the different concepts of complexity when performing the standard operations of the MMP.

\begin{lemma}\label{lem:cut-down-lcc}
Let $(X/Z,B+M)$ be a generalized log canonical pair and let $z\in Z$ be a closed point.
Assume that $K_X+B+M\sim_{\qq,Z}0$. 
Then, there exists an effective $\qq$-Cartier divisor $B'\sim_{\qq,Z} 0$ such that the pair $(X/Z,B+B'+M)$
is generalized log canonical with a glc center contained in the fiber over $z$.
Moreover, 
\begin{align*}
\abscompz(X/Z,B+B'+M)&\leq
\abscompz(X/Z,B+M),
\\  
\finecomp_z(X/Z,B+B'+M)&\leq
\finecomp_z(X/Z,B+M),
\\
\orbcomp_z(X/Z,B+B'+M)&\leq
\orbcomp_z(X/Z,B+M),
\end{align*}
and equality hold in each of the above if and only if $B'=0$.
\end{lemma}

\begin{proof}
If $(X/Z,B+M)$ has a glc center contained in $\pi^{-1}(z)$ we are done.
Hence, we may assume that $(X/Z,B+M)$ has no glc center contained in $\pi^{-1}(z)$.
Let $(Z,B_Z+M_Z)$ be the generalized pair obtained by applying the generalized canonical bundle formula,~\cite[Theorem 4.16]{Fil18}.
Then, $z$ is not a glc center of $(Z,B_Z+M_Z)$.
Let $C$ be a minimal glc center of $(Z,B_Z+M_Z)$ among those containing $z$.
Let $A \ni z$ be an effective ample Cartier divisor not containing any glc center of $(Z,B_Z+M_Z)$.
Let $t$ be the glc threshold of $A$  with respect to $(Z,B_Z+M_Z)$.
Then, $(Z,B_Z+tA+M_Z)$ is glc with a glc center strictly contained in $C$. 
Repeating this argument, we can inductively assume that there exists a $\qq$-Cartier divisor $B'_Z$ such that 
$(Z,B_Z+B'_Z+M_Z)$ is glc and $z$ is a glc center of such pair.
Let $B':=\pi^\ast (B'_Z)$.
Thus, $(X/Z,B+B'+M)$ is glc, and has a glc center in $\pi^{-1}(z)$.
\\
Hence, $\langle \Sigma'/Z\rangle := \langle \Sigma/Z\rangle$, since $B'_Z$ is a $\qq$-Cartier divisor on $Z$.
Finally, given an orbifold decomposition $\Sigma$ of $B$ and writing $B'_Z=\frac 1k C$, $k \in \zz_{>0}$ for $C$ Cartier on $Z$, we define an orbifold decomposition of $B+B'$ by $\Sigma':=\Sigma+\frac 1k \pi^\ast C$; 
here $\pi^\ast C$ is orbifold Weil as $C$ is Cartier, whereas $|\Sigma'|\geq |\Sigma|$ and the equality holds if and only if $B'_Z=0$.
\\
The proof for the absolute and fine complexities is analogous.
\end{proof}

\begin{lemma}
\label{complexity-dlt}
Let $(X/Z,B+M)$ be a generalized log canonical pair and let $z\in Z$ be a closed point.
Let $\pi \colon Y\rightarrow X$ be a $\qq$-factorial extraction of divisors of generalized log discrepancy $0$ for $(X/Z, B+M)$. 
Let $E_1,\dots,E_r$ be the prime exceptional divisors of $\pi$.
\begin{enumerate}
\item 
\label{complexity-vs-dlt-modification}
Let $\Sigma=\sum_{i=1}^k b_iB_i$ be a decomposition of $B$.
Then, there exists a decomposition $\Sigma_{Y}$ of $B_Y$ for which
\begin{align*}
\finecomp_z(Y/Z,B_Y+M_Y;\Sigma_Y)\leq \finecomp_z(X/Z,B+M;\Sigma)
\text{ and }
\abscompz(Y/Z,B_Y+M_Y;\Sigma_Y)\leq
\abscompz(X/Z,B+M;\Sigma).
\end{align*}
The decomposition $\Sigma_Y$ can be taken so that each $E_i$ appears with coefficient $1$.

	\item 
\label{complexity-vs-dlt-modification-orbifold}
Let $n$ be an orbifold structure on $X$ and let $\Sigma$ be an orbifold decomposition of $B$ for $n$.
There exists an orbifold structure $n'$ on $Y$ and an orbifold decomposition $\Sigma_{Y}$ of $B_Y$ for $n'$ such that
\[
\orbcomp_z(Y/Z,B_Y+M_Y;\Sigma_Y)\leq \orbcomp_z(X/Z,B+M;\Sigma).
\]
The orbifold decomposition $\Sigma_Y$ can be chosen so that each $E_i$ appears with coefficient $1$.

\end{enumerate}
\end{lemma}

\begin{proof}
\begin{enumerate}
	\item 
The proof of~\cite[Lemma 2.4.1]{BMSZ18} applies verbatim.  
	\item 
Let $n=((P_1, \dots, P_k), (n_1, \dots, n_k))$ be the orbifold structure on $X$.
We define the orbifold structure $n':=((P'_1, \dots, P'_k), (n_1, \dots, n_k))$ for $P'_i$ the strict transform of $P_i$ on $Y$;
in particular, $n'_{E_i}=1$, for $i=1, \dots, r$.
Writing
\[
\Sigma =
\sum_{P \in X^1}
\left( 1-\frac{1}{n_P}\right) P
+
\sum_{i=1}^s b_iB_i\leq B,
\]
we define
\[
\Sigma_Y :=
\sum_{Q \in Y^1}
\left(
1-\frac{1}{n'_Q}\right) Q+
\sum_{j=1}^r E_j + \sum_{i=1}^s b_i \pi^{-1}_\ast B_i.
\]
$\Sigma_Y$ provides an orbifold decomposition of $B_Y$ as $\Sigma_Y \leq B_Y$.
As $E_1,\dots,E_r, \pi^{-1}_\ast B_1,\dots, \pi^{-1}_\ast B_k$ generate $\langle \Sigma_Y /Z \rangle$, then
\[
\dim_\qq \langle \Sigma_Y /Z \rangle 
= 
\dim_\qq \langle \Sigma /Z \rangle+ r,
\quad 
\text{and}
\quad
| \Sigma_Y | = |\Sigma|+r.
\]
Thus,
\begin{align}
\nonumber & \orbcomp_z(Y,B_Y+M_Y;\Sigma_Y) = \dim Y+
\dim_\qq \langle \Sigma_Y/Z\rangle - |\Sigma_Y|   \\
\nonumber & = \dim X  + \langle \Sigma/Z \rangle + r - (|\Sigma|+r) = \orbcomp_z(X/Z,B+M;\Sigma).
\end{align}
\end{enumerate}
\end{proof}

\begin{lemma}
\label{complexity.divisorial-contraction}
Let $(X/Z,B+M)$ be a generalized log canonical pair and let $z\in Z$ be a closed point.
Let $\pi \colon X\rightarrow X'$
be a divisorial contraction over $Z$ with exceptional divisor $E \subset X$.
Assume that the pair $(X'/Z,B'+M')$ on $X'$ induced via push-forward from $X$ is a generalized log canonical pair over $z\in Z$.
\begin{enumerate}
	\item 
\label{complexity-vs-divisorial-contraction}
Let $\Sigma=\sum_{i=1}^k b_i B_i$ be a decomposition of $B$.
Then, there exists a decomposition $\Sigma'$ of $B'$ for which
\begin{align}
\label{eq.3.34.2a}
&
\finecomp_z(X'/Z,B'+M';\Sigma') \leq \finecomp_z(X/Z,B+M;\Sigma)
\text{ and } \\
\label{eq.3.34.2b}
&
\abscompz(X'/Z,B'+M';\Sigma') \leq \abscompz(X/Z,B+M;\Sigma).
\end{align}
If equality holds in~\eqref{eq.3.34.2a} and for some $i$, $\supp B_i=E$, then $E$ is a glc place of $(X'/Z,B'+M')$ appearing in $\Sigma$ with coefficient one.
Equality holds in~\eqref{eq.3.34.2b} if and only if $\Sigma \geq E$.
	\item 
\label{complexity-vs-divisorial-contraction-orbifold}
Let $n$ be an orbifold structure on $X$ and let $\Sigma$ be an orbifold decomposition of $B$.
There exists an orbifold structure $n'$ on $X'$ and an orbifold decomposition
$\Sigma'$ of $B'$ for which
\begin{align}
\label{eq.3.34.2c}
\orbcomp_z(X'/Z,B'+M';\Sigma') \leq \orbcomp_z(X/Z,B+M;\Sigma).
\end{align}
If equality holds in~\eqref{eq.3.34.2c} and for some $i$, $\supp B_i=E$, then $E$ is a glc place of $(X'/Z,B'+M')$ appearing in $\Sigma$ with coefficient one.
\end{enumerate}
\end{lemma}

The conditions of Lemma~\ref{complexity.divisorial-contraction}
are fulfilled whenever $\pi$ is a divisorial contraction appearing in running a minimal model program over $Z$.

\begin{proof}
\begin{enumerate}
	\item 
Let $\Sigma' = \sum_{i=1}^k b_i \pi_\ast B_i$.
$\Sigma'$ is a decomposition of $B'=\pi_\ast B$.
Moreover, 
$\dim_\qq \langle \Sigma/Z \rangle 
\geq 
\dim_\qq \langle \Sigma'/Z \rangle $ since 
$\pi_\ast B_1,\dots, \pi_\ast B_k$ span $\langle \Sigma' /Z \rangle$, 
and
$\picnumbxz -1 = \picnumbxprimez$.
If none of the $B_i$ is supported on $E$, then $|\Sigma|=|\Sigma'|$ and
\begin{align*}
\finecomp_z(X'/Z,B'+M';\Sigma) &
\leq 
\finecomp_z(X/Z,B+M;\Sigma')
\\
\abscompz(X'/Z,B'+M';\Sigma') &
= 
\abscompz(X/Z,B+M;\Sigma) -1.
\end{align*}
Thus, we may assume that some $B_i$ is supported on  $E$.
Up to summing such $B_i$, we may assume that $B_1$ is the only one among the $B_i$ with such property.
Up to reordering, we can assume that $\{B_1, \dots, B_j\}$ is a basis of $\langle \Sigma/Z \rangle$.
Thus, $\pi_\ast B_2,\dots,\pi_\ast B_j$ is a basis of $\langle \Sigma' /Z \rangle$:
indeed, these divisors generate $\langle \Sigma /Z \rangle$;
To show linear independence over $Z$, let us assume that there existed a relation of the form
$\sum_{i=2}^j \delta_i \pi_\ast B_i \sim_{\qq, Z} 0$.
Pulling-back this linear equivalence to $X$, then
$\sum_{i=2}^j \delta_i B_i + cB_1 \sim_{\qq, Z} 0$ for some $c \in \qq$, which implies $c=0=\delta_i$ for all $i \geq 2$.
As
$|\Sigma| = |\Sigma'|+b_1$, $0< b_1 \leq 1$ then
\begin{align}
\label{eqn:eq.cases.a}
& \finecomp_z(X'/Z,B'+M';\Sigma') = \dim X' + \dim_\qq \langle \Sigma'/Z \rangle - |\Sigma'|  \\
\nonumber & = \dim X  + \dim_\qq \langle \Sigma/Z\rangle -1 - |\Sigma|+b_1 = \finecomp_z(X/Z,B+M;\Sigma)+(b_1-1),
\end{align}
and 
\begin{align}
\label{eqn:eq.cases.b}
& c_z(X',B'+M';\Sigma') = \dim X' +  \dim_\qq  {\rm Cl}_\qq(X'/Z)  - |\Sigma'|  \\
\nonumber & 
= \dim X  + \picnumbxz -1 - |\Sigma|+a = \abscompz(X/Z,B+M;\Sigma)+(a-1).
\end{align}
Thus, equality holds in both~\eqref{eqn:eq.cases.a},\eqref{eqn:eq.cases.b} if and only if $E$ is a summand of $\Sigma$ with coefficient one.
Hence, $E$ is a glc place of $(X'/Z,B'+M')$. 
	\item 
Let $n=((P_1, \dots, P_l), (n_1, \dots, n_l))$ be the given orbifold structure on $X$ and let
\[
\Sigma =\sum_{P \in X^1}
\left(1-\frac{1}{n_P}\right) P
+
\sum_{i=1}^k b_iB_i.
\]
We take the orbifold structure $n'$ defined in Remark~\ref{rmk:orb.struct.restr}.\ref{rmk:orb.struct.bir.contr} and set
\[
\Sigma' :=
\sum_{Q \in Y^1}
\left(
1-\frac{1}{n'_Q}
\right) Q
+
\sum_{i=1}^k b_i\pi_\ast B_i.
\]
$\Sigma'$ is an orbifold decomposition of $B'$
with respect to $n'$.
As $\pi_\ast B_1,\dots, \pi_\ast B_k$ span $\langle \Sigma' /Z \rangle$, then 
$\dim_\qq \langle \Sigma/Z \rangle 
\geq
\dim_\qq  \langle \Sigma'/Z \rangle$.
If for any $i$, $\supp(B_i) \neq E$,
then $|\Sigma|=|\Sigma'|$;
hence,
\[
\orbcomp_z(X/Z,B+M;\Sigma) \geq \orbcomp_z(X'/Z,B'+M';\Sigma')
\]
as claimed.
Thus, we may assume that $B_i=a_iE$ for some $i$ with $a_i\in \qq_{>0}$
and 
$a_i\leq \frac{1}{n_E}$.
Up to summing, we may assume that for all $i>1$, $\supp(B_i)\neq E$.
Let $\{B_1,\dots, B_j\}$ be a basis for $\langle \Sigma/Z \rangle$;
as in the previous part of the proof, $\{\pi_\ast B_2,\dots, \pi_\ast B_j\}$ is a basis of $\langle \Sigma' /Z \rangle$.
On the other hand, $|\Sigma| =|\Sigma'|+b_1$.
Thus,
\begin{align}
\nonumber & \orbcomp_z(X',B'+M';\Sigma') = 
\dim X' + 
\dim_\qq  \langle \Sigma'/Z \rangle - |\Sigma'|  
\\
\nonumber & 
= \dim X  + 
\dim_\qq \langle \Sigma/Z\rangle -1 - |\Sigma|+b_1 = 
\orbcomp_z(X/Z,B+M;\Sigma)+(b_1-1) \geq \orbcomp_z(X/Z,B+M;\Sigma),
\end{align}
since $b_1 \leq 1$;
equality holds if and only if $E$ is a summand of $\Sigma$ with coefficient one.
\end{enumerate}
\end{proof}

\begin{lemma}
\label{complexity.flops}
Let $(X/Z,B+M)$ be a generalized log canonical pair and let $z\in Z$ be a closed point.
Let $\pi \colon X\dashrightarrow X'$ be an isomorphism in codimension one over $Z$. 
Assume that the pair $(X'/Z,B'+M')$ on $X'$ induced via push-forward from $X$ is a generalized log canonical pair over $z\in Z$.
\begin{enumerate}
\item 
\label{complexity-vs-flops}
Let $\Sigma=\sum_{i=1}^k b_i B_i$ be a decomposition of $B$ supported at $z\in Z$.
Then, there exists a decomposition $\Sigma'$ of $B'$ such that
\[
\finecomp_z(X'/Z,B'+M';\Sigma')= \finecomp_z(X/Z,B+M;\Sigma)
\text{ and }
\abscompz(X'/Z,B'+M';\Sigma')=
\abscompz(X/Z,B+M;\Sigma).
\]
	\item
\label{complexity-vs-flops-orbifold}
Let $n$ be an orbifold structure on $X$.
Then, there exists an orbifold structure $n'$ on $X'$ and an orbifold
decomposition $\Sigma'$ of $B'$ for which
\[
\orbcomp_z(X'/Z,B'+M';\Sigma')= 
\orbcomp_z(X/Z,B+M;\Sigma).
\]
\end{enumerate}
\end{lemma}

The conditions of Lemma~\ref{complexity.flops} are fulfilled whenever $X \dashrightarrow X'$ is a flipping contraction of a minimal model program over $Z$.

\begin{proof}
For $(1)$ it suffices to define $\Sigma':=\pi_\ast \Sigma$.
In $(2)$, it suffices to define $n'$ as in Remark~\ref{rmk:orb.struct.restr}.\ref{rmk:orb.struct.bir.contr} and, again, 
define $\Sigma':=\pi_\ast \Sigma$.
\end{proof}

Lemmata~\ref{complexity.divisorial-contraction}-\ref{complexity.flops} readily imply the following corollary.

\begin{corollary}
\label{complexity.MMP}
Let $(X/Z,B+M)$ be a generalized log canonical pair and $z\in Z$ be a closed point. 
Assume that $K_X+B+M\sim_{\qq,Z}0$.
Let $\pi \colon X\dashrightarrow X'$ be a birational contraction consisting of a finite sequence of divisorial contractions and isomorphisms in codimension one over $Z$.
Assume that the pair $(X'/Z,B'+M')$ on $X'$ induced via push-forward from $X$ is a generalized log canonical pair over $z\in Z$.
\begin{enumerate}
	\item 
\label{complexity-under-MMP}
Let $\Sigma$ be a decomposition of $B$.
There exists a decomposition $\Sigma'$ of $B'$ for which
\begin{align}
\label{eqn:abscomp.mmp}
\abscompz(X'/Z,B'+M';\Sigma')& 
\leq 
\abscompz(X/Z,B+M;\Sigma),\\
\label{eqn:finecomp.mmp}
\finecomp_z(X'/Z,B'+M';\Sigma') &
\leq 
\finecomp_z(X/Z,B+M;\Sigma).
\end{align}
Moreover, equality holds in~\eqref{eqn:abscomp.mmp} if and only if all prime divisors contracted by $h$ appear in $\Sigma$ with coefficient $1$.
	\item 
\label{complexity-under-MMP-orbifold}
Let $n$ be a n orbifold structure on $X$.
Let $\Sigma$ be an orbifold decomposition of $B$.
Then, there exists an orbifold structure on $X'$
and an orbifold decomposition $\Sigma'$ of $B'$ for which
\[
\orbcomp_z(X'/Z,B'+M';\Sigma')\leq \orbcomp_z(X/Z,B+M;\Sigma).
\]
\end{enumerate}
\end{corollary}

\subsection{Complexity and adjunction}
\label{subsect:compl.adj}

In this subsection, we prove that the orbifold complexity can only decrease when performing adjunction along certain divisorial generalized log canonical centers.
For the basics of adjunction along divisorial centers for generalized pairs, we refer the reader to~\cite[pp. 304-306]{BZ16}.

\begin{lemma}
\label{lem:complexity-vs-adjunction}
Let $(X/Z,B+M)$ be a $\qq$-factorial generalized log canonical pair and let $z\in Z$ be a closed point.
Let $n$ be an orbifold structure on $X$ and let 
\[
\Sigma:=
\sum_{P\in X^1}
\left(1-\frac{1}{n_P}\right) P
+
\sum_{i=1}^k b_iB_i,
\]
be an orbifold decomposition of $B$.
Assume that the following conditions hold:
\begin{enumerate}
    \item 
$(X/Z,B+M)$ has a divisorial glc center $E$ contained in the fiber over $z$;

    \item 
$E$ is a (projective) normal variety;

    \item 
$B_1=E$, $b_1=1$; and,

   \item 
$\forall j \geq 2$, $B_j\cap E\neq \emptyset$.
\end{enumerate}
Then, there exists an orbifold structure $m$ on $E$
and an orbifold decomposition
$\Sigma_E$ of $B_E$
such that
\[
\orbcomp(E,B_E+M_E;\Sigma_E) \leq  \orbcomp_z(X/Z,B+M;\Sigma),
\]
where $B_E$ is the boundary part in the generalized pair $(E, B_E+M
_E)$ induced by adjunction of $(X, B+M)$ along $E$.
\end{lemma}

Condition $(3)$, together with the assumption that $(X, B+M)$ is generalized log canonical, implies that for all $j \geq 2$, $B_j$ does not contain $E$ in its support.

\begin{remark}
\label{rmk:decomposition.div.lc.center}
{\em 
Let $(X/Z,B+M)$ be a $\qq$-factorial generalized log canonical pair. 
Let 
\[
\Sigma:=
\sum_{P \in X^1}
\left(1-\frac{1}{n_P}\right) P
+
\sum_{i=1}^k b_iB_i,
\]
be an orbifold decomposition of $B$ with respect to an orbifold structure $n$ on $X$.
Let $E \subset X$ be a prime divisor with $B \geq E$ and assume that $n_E \neq 1$.
We aim to show that it is always possible to modify the orbifold structure $n$ and the orbifold decomposition $\Sigma$ so that $n_E=1$ and property $(3)$ in the statement of Lemma~\ref{lem:complexity-vs-adjunction} is satisfied, while at the same time $\orbcomp_z(X/Z,B+M;\Sigma)$ does not increase.
\\
Let $n'$ be the orbifold structure defined by
\[
n'_G:= 
\begin{cases}
1 & G=E\\
n_G & E \neq G \in X^1
\end{cases}
\]
and let $\Sigma'$ be the following orbifold decomposition of $B$ with respect to $n'$
\[
\Sigma':=
\sum_{P\in X^1}
\left(1-\frac{1}{n'_P}\right) P
+ E +
\sum_{i=1}^k b_i (B_i-\coeff_E(B_i)E),
\]
If the support of any of the $B_i$ coincides with $E$, we can then assume, up to reordering and summing, that $B_1=  b_1 E$, $b_1< 1$ and for all $i>1$, $\supp B_i\neq E$.
Then,
\[
B_1-\coeff_E(B_1)E=0, \quad 
\dim_\qq \langle \Sigma' / Z \rangle = 
\dim_\qq \langle \Sigma / Z \rangle
\quad \text{and}  
\quad |\Sigma'|\geq |\Sigma|.
\]
On the other hand, if for all $i \geq 1$, $\supp B_i\neq E$, then $B_i-\coeff_E(B_i)E\neq 0$,
$\dim_\qq \langle \Sigma' / Z \rangle - 
\dim_\qq \langle \Sigma / Z \rangle
$ is either $0$ or $1$, while
$|\Sigma'|= |\Sigma| +1$.
In both cases,
\[
\orbcomp_z(X/Z,B+M;\Sigma') \leq 
\orbcomp_z(X/Z,B+M;\Sigma).
\]
}
\end{remark}

\begin{proof}[Proof of Lemma~\ref{lem:complexity-vs-adjunction}]
For the reader's convenience, we divide the proof into steps.
\\
In view of Remark~\ref{rmk:decomposition.div.lc.center}, we can assume that $n_E=1$.
Let 
\begin{align}
\label{eqn:orb.decomp.adj}
\Sigma=\sum_{P \in X^1}\left(1-\frac{1}{n_P}\right) P
+\sum_{i=1}^kb_iB_i,
\end{align}
be the given orbifold decomposition of $B$.
\\

\noindent
{\it Step 1.} 
In this step, we show that for all $Q \in E^1$ contained in the support of $n$, exactly one of the following two conditions hold:
\begin{itemize}
    \item[(a)] 
there exists a unique prime divisor $P_{0, Q}$ in the support of $n$ containing $Q$; or,
    \item[(b)]
there exist exactly two prime divisors $P_{1, Q},P_{2, Q}$ in the support of $n$ containing $Q$ and
$n_{P_1, Q}=n_{P_2, Q}=2$.
Moreover, in this case $Q \not \subset B_i$ for $i>1$.
\end{itemize}

\noindent
The claim follows by localizing at the generic point of $Q$ and by adjunction along $E$, cf.~\cite[\S~16]{Kol92}, since $(X,B+M)$ is glc at $Q$ and by Remark~\ref{rem:coeff.orb.decomp}
\[
\coeff_{P_{i, Q}(B)}
\geq 1-\frac{1}{n(P_{i, Q})}, 
\quad 
i=0,1,2.
\]
In particular, in case (b), then $Q$ is an lc center of the pair $(X, B+M)$ and $Q$ is not contained in the support of $B_i\vert_E$,
as otherwise, $Q$ would be a center of singularities strictly worse than generalized log canonical.
\\

\noindent
We define the set $\mathcal{S} \subset E^1$ to be 
\begin{align}
    \label{eqn:def.S}
\mathcal{S}:=\{ Q \in E^1 \ \vert \
\text{$Q$ satisfies condition (b) in the statement of Step 1} \}.
\end{align}

\noindent
{\it Step 2}. 
In this step, we define an orbifold structure $m$ on $E$ together with an orbifold decomposition $\Sigma_E$ of $B_E$.\\

\noindent
Given $Q \in E^1$, we denote by $i_Q$ be the Cartier index of $E$ at the generic point of $Q$.
We define the orbifold structure $m$ on $E$ as follows: for $Q \in E^1$,
\begin{align}
\label{def:orb.struct.m.adj}
m(Q):=
\begin{cases}
i_Q & 
\text{if $Q$ is not contained in the support of $n$}\\
n(P_{0, Q})i_Q &
\text{if $Q$ satisfies condition (a) with respect to $P_{0, Q} \in X^1$ in Step 1}\\
1 &
\text{if $Q\in S$}
\end{cases}
.
\end{align}
%
\noindent
By~\eqref{eqn:orb.decomp.adj} and adjunction, if $Q\in \mathcal{S}$, then ${\rm coeff}_{P_{1, Q}}(B)={\rm coeff}_{P_{2, Q}}(B)=\frac{1}{2}$, $Q$ belongs to the support of $\lfloor B_E \rfloor$ and no $B_i$ contains $Q$ in its support for $i>1$.
We define
\[
\Sigma_E :=
\sum_{Q \in E^1}
\left( 1-\frac{1}{m_Q} \right) Q 
+
\sum_{i=2}^k b_i B_i\vert_E. 
\]

\noindent
{\it Step 3.} 
In this step, we show that for $i\geq 2$, the divisors $B_i\vert_E$ are orbifold Weil divisors for the orbifold structure $m$ on $E$.\\

\noindent
Let $Q \in E^1$.
We need to check that for any $i \geq 2$, 
\[
m_Q {\rm coeff}_Q(B_i \vert_E) \in \zz_{\geq 0}.
\]
By Step 1, we may assume that $Q \not \in \mathcal S$.
For $P'\in X^1$, $i_QP'$ is Cartier at $Q$, cf.~\cite[\S 16]{Kol92}, so that
\[
\coeff_Q(P\vert_E)=
\frac{\gamma_{P, Q}}{i_Q},
\quad 
\gamma_{P, Q} \in \mathbb Z_{>0}.
\]
Hence,
\begin{align}
\label{eq:coeff-calc}
 {\rm coeff}_Q(B_i\vert_E)= &
{\rm coeff}_Q
\left( 
\sum_{\substack{P \in X^1\\ P \supset Q}} 
{\rm coeff}_P(B_i) 
P\vert_E 
\right)
=
\sum_{\substack{P \in X^1\\ P \supset Q}}
\frac{
{\rm coeff}_P(B_i)\gamma_{P, Q}}
{i_Q}.
\end{align} 
As $\Sigma$ is an orbifold decomposition, then
$n_P{\rm coeff}_P(B_i)\in \zz_{>0}$, thus,
\[
m_Q\left(
\frac{{\rm coeff}_P(B_i) \gamma_{P, Q}}{i_Q}\right)
\in \zz_{>0}.
\]
\vspace{0.2cm}

\noindent
{\it Step 4.} 
In this step, we show that $\Sigma_E \leq B_E$, i.e., $\Sigma_E$ is an orbifold decomposition of $B_E$ for $m$.\\

\noindent
It suffices to check that for $Q \in E^1$, 
${\rm coeff}_Q(\Sigma_E) \leq {\rm coeff}_Q(B_E)$.
If $Q\in \mathcal{S}$, then  ${\rm coeff}_Q(\Sigma_E)=0$, while ${\rm coeff}_Q(B_E)=1$; 
hence, we may assume that $Q\notin \mathcal{S}$.
By~\eqref{eq:coeff-calc} and the definition of $\Sigma_E$, 
\begin{equation} 
{\rm coeff}_Q(\Sigma_E)=
1-\frac{1}{m_Q}+\sum_{\substack{P \in X^1 \setminus\{E\}
\\ P \supset Q}}
\sum_{i=2}^k  \frac{ b_i{\rm coeff}_P(B_i)\gamma_{P, Q}}{i_Q}.
\end{equation} 
On the other hand, by \cite[Remark 4.8]{BZ16},
\begin{equation}\label{eq1} {\rm coeff}_Q(B_E)
\geq 
{\rm coeff}_Q(\tilde{B}_E)=
1-\frac{1}{i_Q}+\sum_{\substack{P \in X^1\setminus\{E\}
\\ P \supset Q}}
\frac{ {\rm coeff}_P(B)\gamma_{P, Q}}{i_Q},
\end{equation} 
where $\tilde B_E$ is defined by 
\[
K_E+\tilde{B}_E = (K_X+B)\vert_E.
\]
By Step 1, there exists at most one prime divisor $P_{0, Q}$ in the support of $n$ containing $Q$;
if that is the case, then $n_{P_{0, Q}}i_Q=m_Q$ and
\begin{align}
\label{eq3} 
{\rm coeff}_{P_{0, Q}}(B) \geq &
1-\frac{1}{n_{P_{0, Q}}}+\sum_{i=2}^k b_i{\rm coeff}_{P_{0, Q}}(B_i).
\end{align}
For $P \in X^1\setminus\{E, P_{0, Q}\}$, $P \supset Q$, then 
\begin{align}
\label{eq2} 
{\rm coeff}_P(B)\geq & 
\sum_{i=2}^k b_i{\rm coeff}_P(B_i),
\end{align}
Equations~\eqref{eq1}-\eqref{eq3} imply that
\begin{align} \nonumber 
{\rm coeff}_Q(\tilde B_E) & \geq
1-\frac{1}{i_Q} + 
\frac{\left( \left(1-\frac{1}{n_{P_{0, Q}}}\right) +\sum_{i=2}^k b_i{\rm coeff}_{P_{0, Q}}(B_i)\right) \gamma_{P_{0, Q}, Q} }
{i_Q} +\sum_{\substack{P_0\neq P \in X^1\\ P \supset Q}}
\sum_{i=2}^k  \frac{ b_i{\rm coeff}_P(B_i)\gamma_{P, Q}}{i_Q}\\ 
\label{eqn:decomp.ineq}
&\geq 1-\frac{1}{m_Q}+\sum_{\substack{P \in X^1\\ P \supset Q}}
\sum_{i=2}^k  \frac{ b_i{\rm coeff}_P(B_i)\gamma_{P, Q}}{i_Q} = {\rm coeff}_Q(\Sigma_E),
\end{align} 
where we set $\gamma_{P_{0, Q}}=0$ if $Q$ is not contained in the support of $n$.
Equality holds in the chain of inequalities in~\eqref{eqn:decomp.ineq} if and only if equality holds in~\eqref{eq1}-\eqref{eq2}
and ${\rm mult}_Q(P_{0, Q})=1$.
Thus, $B_E\geq \tilde B_E \geq \Sigma_E$.\\

\noindent
{\it Step 5.}
In this step, we show that 
$\orbcomp(E,B_E+M_E;\Sigma_E)\leq \orbcomp_z(X/Z,B+M;\Sigma)$.\\

\noindent
Let us observe that
$|\Sigma|=1+\sum_{i=2}^k b_i$, 
$|\Sigma_E|=\sum_{i=2}^k b_i$, by $(3)$.
Moreover,
\[
\dim_\qq\langle \Sigma_E\rangle 
=
\dim_\qq 
\langle B_2\vert_E,\dots,  B_k\vert_E \rangle 
\leq
\dim_\qq 
\langle B_2,\dots,B_k\rangle 
\leq
\dim_\qq 
\langle \Sigma/Z\rangle
\]
As $\dim E =\dim X -1$, we conclude that
\begin{align*}
\orbcomp_z(X/Z,B+M;\Sigma)=&
\dim X +
\dim_\qq \langle \Sigma/Z\rangle -
1-
\sum_{i=2}^kb_i\\
\geq &
\dim E +
\dim_\qq \langle \Sigma_E\rangle -
\sum_{i=2}^k b_i =
\orbcomp(E,B_E+M_E;\Sigma_E).
\end{align*}
\end{proof}

We finish this subsection by giving some remarks about the above proof.

\begin{remark}{\em
Lemma~\ref{lem:complexity-vs-adjunction} does not have a corresponding version neither for the complexity nor for the fine complexity.
Indeed, in the former case, $\dim_\qq {\rm Cl}_\qq(E)$ may be larger than
$\dim_\qq {\rm Cl}_\qq(X/Z)$
whereas in the latter, it is not necessarily true that the restriction of a decomposition $\Sigma$ of $B$ would yield a decomposition of $B_E$.
Hence, we need to work with the orbifold complexity when dealing with adjunction along divisorial lc center, if we wish to control it.
}
\end{remark}

\begin{remark}
\label{rem:equality-adjunction}
{\em
We adopt the the notation of the proof of Lemma~\ref{lem:complexity-vs-adjunction}.
\begin{enumerate}
	\item
The equality $\orbcomp(E,B_E+M_E;\Sigma_E)= \orbcomp_z(X/Z,B+M;\Sigma)$ holds if and only if
$\dim_\qq \langle \Sigma_E\rangle 
= 
\dim_\qq \langle \Sigma/Z\rangle $. 
	\item
If, moreover, we assume that $\dim_\qq \langle \Sigma_E\rangle =
\dim_\qq {\rm Cl}_\qq(E)$ and $\orbcomp(E,B_E+M_E;\Sigma_E)
=
\orbcomp(E,B_E+M_E)$, 
then $\Sigma_E=B_E$ and $\mathcal S=\emptyset$.
In fact, if that were not to be the case, taking $\tilde \Sigma$ to be the decomposition of $B_E$ into its prime components, then $\orbcomp(E, B_E+M_E; \tilde \Sigma) < \orbcomp(E,B_E+M_E;\Sigma_E)$, contradicting the assumption $\orbcomp(E,B_E+M_E;\Sigma_E)
=
\orbcomp(E,B_E+M_E)$.
\end{enumerate}
}
\end{remark} 

\begin{remark}
\label{rem:res-MY}
{\em 
In the proof of Lemma~\ref{lem:complexity-vs-adjunction}, the constructed orbifold decomposition $\Sigma_E$ is also an orbifold decomposition of the divisor $\tilde B_E \leq B_E$ on $E$ defined by $(K_X+B)|_E = K_E+\tilde{B}_E$.
This follows at once from~\eqref{eqn:decomp.ineq}.
Hence, if $M\vert_E$ gives contribution to $B_E$, i.e., if
$B_E > \tilde{B}_E$, cf.~\cite[Remark 4.8]{BZ16}, then there exists an orbifold decomposition $\Sigma'_E$ of $(E,B_E+M_E)$ such that $\Sigma'_E > \Sigma_E$ and
$|\Sigma'_E|>|\Sigma_E|$. 
Thus, if $\dim_\qq \langle \Sigma' \rangle = 
\dim_\qq \langle \Sigma \rangle$ then 
$\orbcomp(E,B_E+M_E; \Sigma')< \orbcomp(E,B_E+M_E; \Sigma)$.
}
\end{remark}

\subsection{Toric geometry}
For the basic definitions and results about toric varieties, we refer the reader to~\cite{CLS11}.
We will need the following simple technical result about the structure of toric singularities.
For the notion of quotient-dlt (in short, qdlt) singularities, the reader is referred to~\cite[\S 5]{dFKX}.

\begin{lemma}
\label{lem:toric-complexity=0}
Let $X$ be a normal $\qq$-factorial toric variety of dimension $n$.
For any closed point $x\in X$, there exists an affine neighborhood $X^0 \subset X$ containing $x$,
and a boundary divisor $\Gamma^0$ on $X^0$, with the following properties:
\begin{enumerate}
    \item the pair $(X^0,\Gamma^0)$ is qdlt,
    \item $\Gamma^0$ has exactly $n$ components with coefficient one that pass through $x\in X$, and
    \item $\Gamma^0$ contains all the torus invariant prime divisors which pass through $x\in X$.
\end{enumerate}
In particular $\abscompx(X^0, \Gamma^0)=0=c^{\rm loc}(x \in X)$
\end{lemma}

\begin{proof}
The statement is clear if $x \in \mathbb (\kk^\ast)^n \subset X$. 
Hence, we can assume that $x$ is contained in at least one of the toric invariant divisors.
Let $\Gamma_1,\dots, \Gamma_n$ be the torus invariant prime divisors of $X$.
Assume that $x\in X$ is contained in 
\[
\Gamma_1\cap \dots \cap \Gamma_j\setminus  \left(\Gamma_{j+1}\cup \dots \cup \Gamma_n\right).
\]
The variety $X \setminus \cup_{i=1}^{n-j} \Gamma_{j+i}$ is equivariantly isomorphic to $(\kk^\ast)^{n-j} \times V$, where $V$ is a $j$-dimensional normal affine $\qq$-factorial toric variety. 
We denote by $p_1, \dots, p_{n-j}$ (resp., $p_V$) the projection of $(\kk^\ast)^{n-j}\times V$ onto the $n-j$ torus factors (resp. onto $V$).
The prime invariant divisors $\Gamma_1, \dots, \Gamma_j$ are of the form $\Gamma_i= (\kk^\ast)^{n-j} \times \Gamma_{V, i}$, where $\Gamma_{V, 1}\dots, \Gamma_{V, j}$ are the prime invariant divisors of $V$.
By assumption, $x$ is a point of $(\kk^\ast)^{n-j} \times V$, whose projection on $V$ is the invariant point of $V$.
Setting $x_i:=p_i(x)$, then it suffices to consider the boundary divisor
$
\sum_{i=1}^{n-j}p_i^\ast (x_i)+ p_V^\ast ( \Gamma_{V, 1} + \dots + \Gamma_{V, j})$.
\end{proof}

\subsection{Toric formality}
Toric formality is simply the formal analog of the notion of toric varieties and toric morphisms.
Given an algebraic variety $X$ and a closed point $x\in X$, we denote by $\hat {X}_x$ the spectrum of the completion of the local ring at the maximal ideal of $x$.

\begin{definition}\label{def:formally-toric}
{\em 
Let $f \colon X\rightarrow Z$ be a proper contraction of normal quasi-projective varieties. 
Let $z\in Z$ be a closed point, and let $B$ a boundary on $X$. 
We say that the contraction $f$ is {\em formally toric} at $z$ for the pair $(X,B)$,
if there exist
\begin{itemize}
    \item 
a rational polyhedral cone $\sigma \subset \mathbb{Q}^n$, a fan of cones $\Xi \subset \mathbb{Q}^m$ and their associated normal toric varieties $Z(\sigma), X(\Xi)$;
    
    \item 
a proper toric morphism $f_t \colon X(\Xi) \rightarrow Z(\sigma)$;
    
    \item 
$z_0 \in Z(\sigma)$ a closed point; and

    \item 
a toric boundary $\Gamma$ on $X(\Xi)$,
\end{itemize}
satisfying the following conditions:
\begin{enumerate}
    \item 
there exist isomorphisms $\psi, \phi$ making the following diagram commutative
\begin{align}
\label{eqn:formally.toric.diag}
\xymatrix{
X(\Xi)\times_{Z(\sigma)} \widehat {Z(\sigma)}_{z_0}=:
\widehat{X(\Xi)}_{z_0}
\ar[rr]^\psi \ar[d]^{\tilde{f}_t} & & 
\widehat{X}_z:= 
X \times_Z \widehat{Z}_z \ar[d]^{\tilde{f}}\\
\widehat{Z(\sigma)}_{z_0} \ar[rr]^\phi & &
\widehat{Z}_z
}
\end{align}
where $\widehat{Z(\sigma)}_{z_0}$ denotes the completion of the localization of $Z(\sigma)$ at $z_0$; and,
    \item 
$\psi$ maps the support of $\widehat{B}_z:= B \times_{Z} \widehat{Z}_z$ onto the support of $\widehat{\Gamma}_{z_0}:= \Gamma \times_{Z(\sigma)} \widehat{Z(\sigma)}_{z_0}$.
\end{enumerate}
}
\end{definition}

\begin{remark}{\em
\begin{enumerate}
    \item 
As the conditions in~Definition~\ref{def:formally-toric} are local around $z \in Z$, we may always assume that $Z(\sigma)$ and $Z$ are affine varieties.
Moreover, we can always assume that $z_0$ is a torus invariant point of $Z(\sigma)$.
    
    \item 
In Definition~\ref{def:formally-toric}, we do not necessarily assume that $\Gamma$ is the reduced union of all torus invariant divisors, instead we only assume that $\Gamma$ is supported on some invariant divisor.

    \item 
The notion of formally toric morphism is, in general, very different from the notion of toroidal morphism, cf., for example,~\cite{Amb16}.
While the latter is a local condition about the singularities of the reduced fiber of a morphism $X \to Z$ over a neighborhood of a closed point $z\in Z$, the former is a much stronger notion which, in particular, determines the global (formal) structure of the fibers over a neighborhood $z$.
\end{enumerate}
}
\end{remark} 

\begin{definition}
\label{def:form.tor.sing}
{\em
Let $x \in X$ be a germ of a normal variety.
\begin{enumerate}
	\item
A pair $(X, B)$ is {\it formally toric} (or, {\it toroidal}) pair at $x$ if the identity morphism $\xymatrix{X \ar[r]^{Id} & X}$ is formally toric at $x$ for $(X, B)$. 
	
	\item
We say say that a germ $x \in X$ is a {\em formally toric} (or, {\it toroidal}) {\it singularity} if there exists a boundary $B$ on $X$ such that $(X, B)$ is a formally toric pair at $x$.
\end{enumerate}
}
\end{definition}


Given a formally toric singularity $x \in X$ as above, there exists by~\cite[Corollary 2.6]{Art69} a diagram of the form
\begin{align}
    \label{eqn:hut.etale}
\xymatrix{
   &  
   X' \ni P' \ar[dl]_u  \ar[dr]^v &   \\
   X \ni P &  & 
   Z\ni Q
} 
\end{align}
with $u,v$ \'etale, $u(P')=P$, $v(P')=Q$, and $u^{-1}(U)=U'=v^{-1}(T_N)$.

\begin{example}
\label{ex:snc.toroidal}
{\em 
Let $(X, B)$ be a snc log pair around a closed point $x \in X$.
Then the pair $(X, B)$ is formally toric at $x$.
The same conclusion holds if $x$ is an lc center of a qdlt pair $(X, B)$, see~\cite[\S 5]{dFKX}.
}
\end{example}

As we will be interested in characterizing morphisms that are formally toric with respect to log canonical pairs, we will also need a suitable notion of formal toricness for log canonical places.

\begin{definition}
\label{def:formally-toric-lc-place}
{\em 
Let $(x \in X,B)$ be a germ of a formally toric pair with log canonical singularities.
Let $E$ be a log canonical place of $(X,B)$ whose center contains $x$.
Then, $E$ is a {\em formally toric lc place},  if the corresponding valuation on $(\hat {X}(\sigma)_{x_0},\hat {\Gamma})$ is a toric valuation, i.e., a monomial valuation on the toric germ.
}
\end{definition}

\begin{lemma}
\label{formally toricity-vs-MMP}
Let $X \rightarrow Z$ be a formally toric morphism around $z\in Z$ with respect to the log canonical pair $(X,B)$.
Let $\pi\colon X'\dashrightarrow X$ be a birational contraction over $Z$ only extracting formally toric lc places of $(X,B)$ whose centers intersect the fibre over $z$.
Then, the morphism $X'\rightarrow Z$ is formally toric around $z\in Z$ with respect to the pair $(X',\pi^{-1}_\ast B+{\rm Exc}(\pi))$.
\end{lemma}

\begin{proof}
Let $X(\Xi)\rightarrow Z(\sigma)$, $z_0\in Z(\sigma)$, and $\Gamma$ on $X(\Xi)$,
be the data giving $X\rightarrow Z$ a formally toric structure around $z\in Z$ for the pair $(X,B)$ as in Definition~\ref{def:formally-toric}.
Since $\widehat{Z(\sigma)}_{z_0}\simeq \widehat{Z_z}$, we may identify both completions.
Let $\Xi'$ be a refinement of $\Xi$ such that the associated partial resolution $X(\Xi')\rightarrow X(\Xi)$
extracts all the toric lc places corresponding to formally toric lc places
extracted by $X'\dashrightarrow X$.
Without loss of generality, we may assume that $X(\Xi')$ is $\qq$-factorial.
Hence, there exists a birational morphism
\[
\widehat{X(\Xi')}_
{z_0} := X(\Xi') \times_{Z(\sigma)} \widehat{Z(\sigma)}_{z_0}
\dashrightarrow  \widehat{X'}
\]
over $\widehat{Z_z}$, which is defined at every point of codimension one, 
and extracts no divisors.
Let $\widehat{A'}$ be an ample Cartier divisor on $\widehat{X'}$ over $\widehat{Z_z}$
and let $\widehat{H}$ be its pull-back to $\widehat{X(\Xi')}$.
As $X(\Xi')$ is toric, then the Class group of $\widehat{X(\Xi')}$ is generated by the base change of the torus invariant divisors -- since the class group of the formal punctured disk is trivial.
Hence, we can assume that $\widehat H$ is linearly equivalent to the base change of a toric invariant divisor $H$ on $X(\Xi')$.
We can run a minimal model program for $H$ over $Z(\sigma)$ terminating with a toric ample minimal model $X(\Xi'') \to Z(\sigma)$.
Running such minimal model program does not contract any divisor on $X(\Xi')$ as $H$ is movable on $X(\Xi')$, by construction.
By taking the base-change $\widehat{X(\Xi'')}$ of $X(\Xi'')$ with the completion $\widehat{Z(\sigma)}_{z_0}$ of $Z(\sigma)$, and using the isomorphism $\widehat{Z(\sigma)}_{z_0} \to \widehat{Z_z}$,
we obtain a formally toric morphism $\widehat{X''}\rightarrow \widehat{Z_z}$, 
such that $\widehat{X''} \dashrightarrow \widehat{X'}$ is a birational map defined in codimension one and extracting no divisors.
Moreover, the pull-back of $\widehat{A'}$ to $\widehat{X''}$ is ample.
Hence, $\widehat{X''}\simeq \widehat{X'}$ over $\widehat{Z_{z}}$ and said isomorphism also identifies the support of the completion of the strict transform of $\Gamma$ on $X''$ with the support of $\hat{B}$.
\newline
As every prime component of $\pi^{-1}_\ast B+{\rm Exc}(\pi)$, after base change to the completion, corresponds to a ray of $\Xi'$, then their strict transforms on $\widehat{X''}$ corresponds to the images under the isomorphism $\widehat{X''} \simeq \widehat{X(\Xi'')}$ of the base change to completion of the toric invariant divisors of $\widehat{X(\Xi'')}$.
\end{proof}


\section{Orbifold complexity in the absolute case}
\label{sect:orb.absolute}
In this section we will show that Shokurov's conjecture holds for the orbifold complexity in the absolute case, namely, when $Z$ is just a point, see Theorem~\ref{thm:formally toric-projective}. 

\subsection{Orbifold complexity of Mori dream spaces}\label{sec:orbi-mds}

The aim of this subsection is to prove that a projective Mori dream space with a log Calabi-Yau structure of orbifold complexity zero is toric.

We start by proving that the local orbifold complexity is always non-negative and if it is zero, then the germ is formally toric.
This is an analog in the framework of orbifold complexity of~\cite[Lemma 2.4.3]{BMSZ18}. 
While this result is not crucially used in the rest of the article, it strongly motivates the treatment of this section.
For the details on cyclic covers, we refer the reader to~\cite[\S~2-3]{Amb16}. 

\begin{theorem}
\label{thm:local-orbifold-complexity}
Let $(x \in X,B)$ be the germ of a $d$-dimensional log canonical pair.
Let $n$ be an orbifold structure on $X$ and let $\Sigma$ be an orbifold decomposition of $B$
supported on the orbifold Weil divisors $B_1,\dots,B_k$.
Assume that $K_X$ and $B_1,\dots,B_k$ are $\qq$-Cartier at $x \in X$.
Then, the following hold:
\begin{enumerate}
    \item
$\orbcomp_x(X,B;\Sigma)\geq 0$;

    \item 
if $\orbcomp_x(X,B;\Sigma)<\frac{1}{2}$, then, up to reordering the $B_i$, $(X, \supp(B_1+\dots +B_d))$ is formally toric at $x$ and $\lfloor B\rfloor \subseteq \supp(B_1+\dots+B_n)$; 
and,

    \item 
if moreover $\orbcomp_x(X,B;\Sigma)=0$, then all $B_i$ are prime divisors, $\Sigma=B$, and $n$ is supported on formally toric divisorial lc places.
\end{enumerate}
\end{theorem}

\begin{proof}
Let $\phi\colon Y\rightarrow X$ be the finite morphism obtained by inductively taking the index one cover the $\qq$-Cartier $\qq$-divisors $B_1,\dots,B_k$ and $K_X$.
Let $K_Y+B_Y=\pi^\ast(K_X+B)$ be the log pull-back.
For all $i$, set $B_{Y,i} :=\phi^\ast B_i$.
Thus, for all $i$, $B_{W,i}$ is Cartier on $Y$ and the point $x$ has a unique pre-image $w$ on $W$.
For $P\in X^1$ with $n_P>1$ and for each $B_i$, we will use the canonical expression of $B$, 
${\rm coeff}_P(B_i) := \frac{{\rm num}_P(B_i)}{n_P}$,
cf. Remark~\ref{rmk:coeff.orb.struct.canonical}.\ref{canonical.expr}, and we will denote by $Q_P$ the reduced pull-back on $Y$.
We denote by $r_P$ the ramification index over $P$, so that $r_P\leq n_P$ for each $P\subset X^1$.
Writing
\[
\Sigma=
\sum_{P\in X^1}
\left(
1-\frac{1}{n_P}
\right)
P+
\sum_{i=1}^r b_i B_i
\] 
then,
\begin{align}
\nonumber 
&
\phi^\ast(K_X+\Sigma)=
\\
\nonumber 
&  \phi^\ast \left(K_X+\sum_{P\in X^1}
\left(
\left(
1-\frac{1}{n_P}
\right)+
\sum_{i=1}^r b_i {\rm coeff}_{P}(B_i)
\right)P\right) =
\\ 
\label{equation-index-one-cover}
& K_Y+
\sum_{P \in X^1}
\left( 
\left(
1-\frac{r_P}{n_P}
\right) 
+\sum_{i=1}^r b_i\frac{r_P{\rm num }_P(B_i)}{n_P} 
\right)Q_P \geq \\
\nonumber 
& 
K_Y+ \sum_{P\in X^1 }
\left(
\sum_{i=1}^r b_i {\rm num}_{P}(B_i)
\right)Q_P  = 
& \hfill 
[r_P\leq n_P
\text{ and }
\sum_{i=1}^k b_i{\rm num}_P(B_i)
\leq 1]
\\ 
\nonumber 
& 
K_Y+
\sum_{i=1}^r b_i
\left( 
\sum_{P\in X^1} {\rm num}_P(B_i)Q_P
\right)
\geq 
K_Y+\sum_{i=1}^r b_iB_{Y,i}.
\end{align}
Defining $\Sigma_Y:=\sum_{i=1}^r b_i B_{Y,i}$, then~\eqref{equation-index-one-cover} implies that
$K_Y+\Sigma_Y \leq \phi^\ast (K_X+\Sigma) \leq \phi^\ast (K_X+B) = K_Y+B_Y$.
Thus, $B_Y\geq \sum_{i=1}^r b_i B_{Y,i}\geq 0$ and $(Y,B_Y)$ is a log canonical pair.
Moreover, $K_Y$ is Cartier as well and $\Sigma_Y$ gives a local decomposition of $B_Y$ around $y$.
As all the $B_{Y,i}$ are Cartier,~\cite[2.4.3.(1)]{BMSZ18} implies that $\orbcomp_x(X,B;\Sigma)=\finecomp_x(Y,B_Y;\Sigma_Y)\geq 0$.
Assuming $\orbcomp_x(X,B;\Sigma)<\frac{1}{2}$, then~\cite[2.4.3.(2)]{BMSZ18} implies that  $(Y,\supp(B_{Y,1}+\dots +B_{Y,n}))$ is log smooth and $\lfloor B_Y\rfloor\subseteq
\supp(B_{Y,1}+\dots+B_{Y,d})$ and that $x\in X$ is obtained by a sequence of quotients of a smooth germ;
in particular, $x \in X$ is a $\qq$-factorial singularity and it is formally toric by Example~\ref{ex:snc.toroidal}.
By construction $(X,\supp(B_1+\dots+B_d))$ is log canonical and $\lfloor B \rfloor \subset \supp(B_1+\dots+B_d)$.
Hence, $(X,\supp(B_1+\dots +B_d))$ is a log canonical
$\qq$-factorial pair of local complexity zero at $x$.
Then, by~\cite[18.22]{Kol92}, we conclude it is a toric pair.
\\
We prove the last statement.
If the $B_i$ are not prime, then also the $B_{Y,i}$ are not prime.
Hence, as $x \in X$ is a $\qq$-factorial singularity, we can refine the decomposition $\Sigma_Y$ by taking the prime components of $B_Y$, thus obtaining a decomposition of negative local complexity.
This gives a contradiction.
If $\Sigma<B$, then we can always add prime divisors supported on $\supp(\Sigma-B)$ and decrease the complexity.
Finally, if we have a prime divisor $P$ on $X$ so that $n_P>1$ and $P$ is not contained in $\lfloor B\rfloor$, then the equation~\ref{equation-index-one-cover} is a strict inequality, so we can find a decomposition for $(Y,B_Y)$ with negative complexity, leading to a contradiction.
\end{proof}

The following is the main theorem of this section.

\begin{theorem}
\label{thm:proj-MDS-orbifold-complexity}
Let $(X,B)$ be a log canonical pair.
Assume that $X$ a projective $\qq$-factorial Mori dream space and that $K_X+B\sim_\rr 0$.
Let $n$ be an orbifold structure on $X$ and let $\Sigma$ be an orbifold decomposition of $B$ with respect to $n$ supported on the orbifold Weil divisors $B_1,\dots,B_k$.
Then, $\orbcomp(X, B; \Sigma)\geq 0$.
Moreover, if $\orbcomp(X, B; \Sigma)<1$,
then the following conditions hold:
\begin{enumerate}
\item $X$ is a projective toric variety,
\item $\lfloor B\rfloor$ is torus invariant, and 
\item all but $\lfloor 2c\rfloor$ toric prime divisors appear among $\supp(B_1),\dots,\supp(B_k)$.
\end{enumerate}
\end{theorem}

We will need the following two technical statements adapted from those in~\cite[\S 3]{BMSZ18}.

\begin{lemma}\label{lem:orb-comp-pic=1}
Fix a positive integer $d$.
Assume Theorem~\ref{thm:proj-MDS-orbifold-complexity} holds up to dimension $d-1$.
\\
Let $X$ be a $\qq$-factorial $d$-dimensional klt projective variety and let $(X,B)$ be a log canonical pair with $K_X+B\sim_\rr 0$.
Let $n$ an orbifold structure on $X$ and let $\Sigma$ be an orbifold decomposition of $B$ supported on the orbifold Weil divisors $B_1,\dots,B_k$.
Assume that $B_1,\dots,B_k$ span the same
ray in the cone of divisors.
If $|\Sigma|>d$, then $X$ has Picard rank one.
\end{lemma}

\begin{proof}
The proof of~\cite[Lemma 3.3]{BMSZ18} applies verbatim, using Remark~\ref{rmk:orb.struct.restr}.
\end{proof}

\begin{proposition}
\label{prop:gen-class-group}
Fix a positive integer $d$.
Assume Theorem~\ref{thm:proj-MDS-orbifold-complexity} holds up to dimension $d-1$.
\\
Let $X$ be a projective $\qq$-factorial Mori dream space of dimension $d$ and let $(X,B)$ be a log canonical pair with $K_X+B\sim_\rr 0$.
Let $n$ be an orbifold structure on $X$ compatible with $n$ and let $\Sigma$ be an orbifold decomposition of $B$ supported on the orbifold Weil divisors $B_1,\dots,B_k$.
If $\orbcomp(X,B;\Sigma)<1$, then $B_1,\dots,B_r$ generate ${\rm Cl}_\qq(X)$.
\end{proposition}

We remind the reader that a Mori dream space satisfies $H^1(X, \mathcal O_X)=0$.

\begin{proof}
The proof of~\cite[Theorem 3.2]{BMSZ18} applies verbatim, using Remark~\ref{rmk:orb.struct.restr}.
\end{proof}

\begin{proof}[Proof of Theorem~\ref{thm:proj-MDS-orbifold-complexity}]
We prove the result by induction on the dimension $d$ of $X$.
The case $d=1$ is trivially true.
Thus, we may assume that the result holds for any variety of dimension $d-1$.
By Proposition~\ref{prop:gen-class-group}, and inductive hypothesis, the divisors $B_i$ generate ${\rm Cl}_\mathbb{Q}(X)$, so that
\[
\orbcomp(X,B;\Sigma)= 
 \dim X +\dim_\qq{\rm Cl}(X)_\qq - \sum_{i=1}^k b_i.
\]
Writing, cf. Remark~\ref{rmk:coeff.orb.struct.canonical}.\ref{canonical.expr},
\[
\Sigma =
\sum_{P\in X^1}
\left(
\left(
1-\frac{1}{n_P}
\right)+
\sum_{i=1}^k b_i \frac{{\rm num}_{P}(B_i)}{n_P}
\right)P,
\]
and defining
\begin{align}
    \label{eqn:def.qi}
Q_i :=\sum_{P\in X^1} {\rm num}_P(B_i)P,
\end{align}
then for all $i$, $Q_i$ is a Weil divisor satisfying $\supp(Q_i)=\supp(B_i)$.
As for all $P \in X^1$, $\sum_{i=1}^k b_i {\rm num}_{P}(B_i) \leq 1$, then 
\begin{align}
\label{eqn:special.ineq}
1-\frac{1}{n_P} + \sum_{i=1}^k b_i \frac{{\rm num}_{P}(B_i)}{n_P} 
\geq 
\sum_{i=1}^k b_i {\rm num}_{P}(B_i)
\end{align}
and equality holds if and only $n_P=1$ or $\Sigma \geq P$.
Hence,
\begin{align*}
\Sigma \geq &
\sum_{P \in X^1}
\left( 
\sum_{i=1}^k b_i {\rm num}_{P}(B_i)
\right)P
=
\sum_{i=1}^k b_i 
\left( \sum_{P \in X^1}
{\rm num}_{P}(B_i) P
\right)
=\sum_{i=1}^k b_i Q_i.
\end{align*}
Hence, $\tilde{\Sigma}:=\sum_{i=1}^k b_i Q_i$ is a decomposition of $B$ of fine complexity complexity
\begin{align}
\label{eqn:complexity.vs.orb.complexity}
\finecomp(X, B; \tilde{\Sigma})=
\dim X+
\dim_\qq \langle \tilde{\Sigma} \rangle - 
\sum_{i=1}^k b_i
\leq \dim X +
\dim_\qq{\rm Cl}(X)_\qq -
\sum_{i=1}^k b_i
=\orbcomp\left(X, B; \Sigma\right) < 1.
\end{align}
By~\cite[Theorem 1.2]{BMSZ18}, $\finecomp(X, B; \tilde{\Sigma}) \geq 0$ and equality must hold in~\eqref{eqn:complexity.vs.orb.complexity}, since $\dim_\qq{\rm Cl}_\qq(X) - \dim_\qq \langle \tilde{\Sigma} \rangle \in \mathbb Z_{\leq 0}$.
As $\orbcomp\left(X, B; \Sigma\right)<1$,~\cite[Theorem 1.2]{BMSZ18} implies that
\begin{enumerate}  
\item $X$ is a projective toric variety,
\item $\lfloor B\rfloor$ is a torus invariant divisor, and 
\item all but $\lfloor 2c\rfloor$ toric prime divisors appear among the $Q_i$.
\end{enumerate}
By~\eqref{eqn:def.qi}, we conclude that property $(3)$ above must hold also for
$\supp(B_i)$.
\end{proof}

\subsection{Orbifold complexity of generalized projective pairs}
\label{section:projective-case}

The aim of this subsection, is to prove a version of Theorem~\ref{formally toric} for the orbifold complexity in the absolute case, that is, under the assumption that $Z$ is just a point.
In particular, the following theorem generalizes~\cite[Theorem 1.2]{BMSZ18} in a twofold manner:
firstly, we prove the theorem in the more general framework of the orbifold complexity;
moreover, we work in the category of generalized log canonical pairs that are log Calabi--Yau.

\begin{theorem}
\label{thm:formally toric-projective}
Let $(X,B+M)$ be a projective generalized log canonical pair such that $K_X+B+M\sim_{\qq} 0$.
Let $n$ orbifold structure on $X$ and let $\Sigma$ be an orbifold decomposition of $B$ supported on the orbifold Weil divisors $B_1,\dots,B_k$.
Then $\orbcomp(X,B+M;\Sigma)\geq 0$.
Moreover, if $\orbcomp(X,B+M;\Sigma)=c<1$, then the following conditions hold:
\begin{enumerate}
\item $X$ is a projective toric variety,
\item the divisor $\lfloor B\rfloor$ is torus invariant, and
\item all but $\lfloor 2c \rfloor$ torus invariant prime divisors appear as the support of the Weil orbifold divisors in the orbifold decomposition $\Sigma$.
\end{enumerate}  
In particular, the divisors appearing in the orbifold decomposition $\Sigma$ generate ${\rm Cl}_\qq(X)$.
Moreover, if $c(X, B+M)=0$, then $M\sim_\qq 0$ and $M$ descends on $X$ in the sense of b-divisors.
\end{theorem}

Theorem~\ref{thm:formally toric-projective} holds in dimension $1$: indeed, $K_X+B+M\sim_\qq 0$ implies that $X \simeq \pp^1$, or $X$ is an elliptic curve and $B=0 \sim_\qq M$.
In the latter case, the orbifold complexity (resp. the complexity) of the generalized log pair is $1$ (resp. $2$).
In the former case, properties (1)-(3) follow from Theorem~\ref{thm:proj-MDS-orbifold-complexity}. 
If $c=0$, then $K_{\pp^1}+B\sim_\qq 0$ and $M\sim_\qq 0$, i.e., $M$ descends on $\pp^1$.
We will proceed to prove Theorem~\ref{thm:formally toric-projective} by induction on the dimension of $X$.

We start by extending some of the technical results of~\cite[\S 4]{BMSZ18} on the numerical dimension of log pairs of small complexity to the framework of orbifold complexity.

\begin{lemma}
\label{lemma:gen-pair-low-numerical}
Fix a positive real number $\alpha \in (0,1)$ and a positive integer $d$.
Assume that Theorem~\ref{thm:formally toric-projective} holds in dimension at most $d-1$.
\\
Let $(X,B+M)$ be a $\qq$-factorial projective generalized dlt pair of dimension $d$.
Assume that $K_X+B+M\sim_\qq 0$, and $\orbcomp(X,B+M)=c \geq 0$.
Then, there exists a projective generalized klt pair $(X,B_0+M_0)$ with the following properties: 
\begin{enumerate}
    \item $\alpha B < B_0 <B$;
    \item $ M_0=  M+\overline{A}$ for some ample divisor $A$ on $X$; 
    \item $N_\sigma(K_X+B_0+M_0)$ has no common component with $B_0$; and,
    \item the numerical dimension of $K_X+B_0+M_0$ is $ \leq c$.
\end{enumerate}
\end{lemma}

For the definition of $N_\sigma(D)$ for a $\mathbb Q$-divisor $D$, the reader is referred to~\cite[\S~2.1]{BMSZ18}.

\begin{proof}
Let $\mathcal{G}_\alpha(B,M)$ to be the set of pseudo-effective generalized klt pairs $(X,B'+M')$ satisfying properties (1)-(3) above.
The set $\mathcal{G}_\alpha(B,M)$ is non-empty:
indeed, it contains the generalized pair  $(X,(\alpha+\epsilon)B + M+A)$, where $\epsilon$ is an arbitrarily small positive real number, and $A$ is sufficiently ample so that $K_X+(\alpha+\epsilon)B+M+A$ is ample.
In particular, the numerical dimension of the divisor $K_X+(\alpha+\epsilon)B+M+A$ is equal to the dimension of $X$, and $N_\sigma(K_X+(\alpha+\epsilon)B + M+A)=0$.
\\
Let $(X,B_0+M_0) \in \mathcal{G}_\alpha(B,M)$ be a generalized pair minimizing the numerical dimension of $K_X+B_0+M_0$.
If the minimal numerical dimension is $\leq c$, then we are done. 
Hence, by contradiction, we can assume that the minimal value of the numerical dimension of the generalized log divisors in $\mathcal{G}_\alpha(B,M)$ is $> c$.
As $K_X+B+M\sim_{\mathbb Q}0$, then this minimum is $<\dim X$.
Let $(Y,B_{Y,0}+M_{Y,0})$ be a good minimal model of $K_X+B_0+M_0$ whose existence is granted by~\cite[Lemma 4.4]{BZ16}, since $M_0$ is big. 
Let $B_Y$ be the strict transform of $B$ on $Y$ and let  $M_Y$ be the trace of the nef part of $(X, B+M)$ on $Y$.
The map $X\dashrightarrow Y$ does not contract any component of $B_0$:
indeed, it only contracts the components of $N_\sigma(K_X+B_0+M_0)$.
Let $Y \rightarrow W_0$ be the Iitaka fibration of $K_Y+B_{Y,0}+M_{Y,0}$.
We claim that there exists a component $P$ of $B$ vertical over $W_0$.
Otherwise, let $\Sigma$ be an orbifold decomposition of $B$, such that $\orbcomp(X,B+M;\Sigma) =c$.
Corollary~\ref{complexity.MMP} implies that there exists an orbifold decomposition $\Sigma_Y$ of $(Y, B_Y+M_Y)$ such that 
\[
\orbcomp(Y,B_Y+M_Y;\Sigma_Y) \leq \orbcomp(X,B+M;\Sigma)=c.
\]
Restricting $K_Y+B_Y+M_Y$ to the general fiber $F$ of $Y\rightarrow W$, and considering the induced orbifold structure on $F$, cf. Remark~\ref{rmk:orb.struct.restr}, together with the orbifold decomposition $\Sigma_F:=\Sigma_Y|_F$ of the boundary part $B|_F$ of $K_F+B|_F+M|_F\sim_\qq 0$, then the argument in the proof of~\cite[Lemma 4.3]{BMSZ18} applies verbatim.
Thus, we may assume that there exists a component $P$ of $B$ vertical over $W$.
\\
Let $p$ be the coefficient of $P$ in $B$. 
Fix $\lambda$ to be the minimal non-negative real number for which the following divisor is pseudo-effective:
\begin{align}
\nonumber 
& K_X+B_1+M_1 
= K_X + (\lambda B_0 + (1-\lambda)(B-pP)) + (\lambda M_0 + (1-\lambda) M) \\
\label{eqn:b1} 
= & \lambda(K_X+B_0+M_0)+(1-\lambda)(K_X+B+M-pP) 
\sim_\qq \lambda(K_X+B_0+M_0) - (1-\lambda) p P.
\end{align}
We set 
$
B_1 := \lambda B_0 + (1-\lambda)(B-pP)
$
and
$M_1 := \lambda  M_0 + (1-\lambda)  M$.
The generalized pair $(X,B_1+M_1)$ satisfies condition (2) in the definition of $\mathcal{G}_\alpha(B+M)$; furthermore, $\supp(B_0)=\supp(B_1)$.
We claim that $\lambda \in (0,1)$.
Clearly, $\lambda>0$. 
As $(X, B_0)$ is klt and $M_0$ is big on this model, there exists a klt pair $(X,B_0+\Delta_0)$ with $\Delta_0$ big, such that $K_{X}+B_0+M_0 \sim_\qq K_X+B_0+\Delta_0$;
thus, by~\cite[Lemma 2.2.1]{BMSZ18}, the divisor
\[
t (K_X+B_0+M_0) - (1-t)pP \sim_\qq
t (K_X+B_0+\Delta_0) - (1-t)pP
\]
is pseudo-effective for $0<\delta \ll1$ and for all $t \in [1-\delta, 1]$;
thus, $\lambda \in (0,1)$.
\\
We claim that the numerical dimension of $(X,B_1+M_1)$ is strictly less than the numerical dimension of $(X,B_0+M_0)$.
For $t\in [0,1]$, we set
\[
B_t:=(1-t)B_0+tB_1 
\quad 
\text{ and }
\quad 
M_t:=(1-t)M_0+tM_1.
\]
By construction, $K_X+B_t+M_t$ is pseudo-effective for each $t\in [0,1]$ and, by~\cite[Lemma 4.4]{BZ16}, it admits an ample model $X\dashrightarrow W_t$.
Moreover, by~\eqref{eqn:b1} $K_X+B_t+M_t$ is a convex linear combination of $K_X+B_0+M_0$ and $-P$. 
If $P$ is in the stable base locus of $K_X+B_t+M_t$, then $W_t=W_0$ and $P$ is vertical over $W_t$.
On the other hand,~\cite[Lemma 2.2.1]{BMSZ18} implies that $P$ is vertical over $W_t$ for each $t<1$. 
By~\cite[3.3.2]{HM13}, we may find $\delta>0$ such that $W_t$ is independent of $t\in (1-\delta,1)$, call it $W$, and there is a contraction morphism $W \rightarrow W_1$.
On the other hand~\cite[Lemma 2.2.1]{BMSZ18} implies that $P$ is horizontal for $W_1$, hence, $W\rightarrow W_1$ is not birational: 
the numerical dimension of $K_X+B_1+M_1$ is strictly less than the numerical dimension of $K_X+B_0+M_0$.
\\
The generalized pair $(X,B_1+M_1)$ may not satisfy conditions (1), (3) in the definition of $\mathcal{G}_\alpha(B,M)$.
Writing $K_X+B_1+M_1 \sim_\rr P_\sigma + N_\sigma$, and given $s\in (0,1]$, then
\[
K_X + B_s + M_s = (1-s)(K_X+B+M) + s(K_X+B_1+M_1) \sim_\rr s(P_\sigma + N_\sigma). 
\]
For $0< s \ll 1$
\begin{align*}
B_2 := s B_1 + (1-s) B - s(N_\sigma \wedge B) \geq 0
\quad 
\text{and}
\quad
\alpha B < B_2 < B .
\end{align*}
Setting $M_2:=M_s =(1-s)M + sM_1$, then 
\[
P_\sigma(K_X+B_2+M_2) = sP_\sigma
\quad 
\text{ and } 
\quad 
N_\sigma(K_X+B_2+M_2) = s(N_\sigma - N_\sigma \wedge B).
\]
Thus, the generalized klt pair $(X,B_2+M_2)$ belongs to $\mathcal{G}_\alpha(B+M)$ and its numerical dimension is strictly less than the numerical dimension of $(X,B_0+M_0)$, as the numerical dimension of $(X,B_2+M_2)$ equals the numerical dimension of $(X,B_1+M_1)$.
This gives the desired contradiction.
\end{proof}

\begin{corollary}
\label{cor:num-dim-zero-proj}
Fix a positive real number $\alpha \in (0,1)$ and a positive integer $d$.
Assume that Theorem~\ref{thm:formally toric-projective} holds in dimension $d-1$.
\\
Let $(X,B+M)$ be a $\qq$-factorial projective generalized dlt pair of dimension $d$.
Assume that $K_X+B+M\sim_\qq 0$, and $\orbcomp(X,B+M)=c <1$. 
Then there exists a generalized klt pair $(X,B_0+M_0)$ with the following properties: 
\begin{enumerate}
    \item $\alpha B < B_0 <B$,
    \item $M_0=M+\overline{A}$ for some ample divisor $A$ on $X$,  
    \item $K_X+B_0+M_0\sim_\rr N\geq 0$ has numerical dimension zero, and 
    \item $N$ has no common components with $B_0$.
\end{enumerate}
\end{corollary}

\begin{lemma}
\label{lem:good-dlt-model}
Fix a positive real number $\alpha \in (0,1)$ and a positive integer $d$.
Assume that Theorem~\ref{thm:formally toric-projective} holds in dimension $d-1$.
\newline
Let $(X,B+M)$ be a $\qq$-factorial projective generalized dlt pair of dimension $d$.
Assume that $K_X+B+M\sim_\qq 0$, and $\orbcomp(X,B+M)=c <1$. 
Then there exists a $\qq$-factorial gdlt modification $(Y,B_Y+M_Y)$ of $(X,B+M)$, and a generalized klt pair $(Y,B'_Y+M'_Y)$ with the following properties:
\begin{enumerate}
    \item $\alpha B_Y < B'_Y <B_Y$,
    \item $M_Y'=M_Y+A_Y$ where $A_Y$ is a big and nef divisor,
    \item $K_Y+B_Y'+M_Y'\sim_\rr N_Y'\geq 0$ has numerical dimension zero, 
    \item $N_Y'$ has no common components with $B_Y'$, and
    \item no generalized non-klt center of $(Y,B_Y+M_Y)$ is contained in $N_Y'$.
\end{enumerate}
\end{lemma}

\begin{proof}
Let $(X,B_0+M_0)$ be a generalized klt pair provided by Corollary~\ref{cor:num-dim-zero-proj}.
We know that $B_0$ has no common components with $N$.
Fix $0<t\ll 1$ such that $\lfloor tN \rfloor =0$.
Let $\pi \colon Y \rightarrow X$ be a gdlt modification of $(X,B+M+tN)$.
For $0< t \ll 1$, it is also a gdlt modification for $(X,B+M)$.
Then,
\[
K_Y+B_Y+M_Y = \pi^\ast (K_X+B+M) \sim_\qq 0 
\quad
\text{ and }
\quad
N_Y=\pi^\ast (N)
\]
with $(Y,B_Y+M_Y)$ gdlt, and $(Y,B_Y+N_Y'+M_Y)$ gdlt as well, where $N_Y'$ is the sum of the components of $N_Y$ which are not contained in the support of $B_Y$.
In particular, the effective divisor $N_Y'$ does not contain any generalized non-klt center of the pair $(Y,B_Y+M_Y)$.
Setting 
\[
K_Y+B_{Y,0}+M_{Y,0}=\pi^\ast (K_X+B_0+M_0), 
\]
where $M_{Y,0}$ is the trace on $Y$ of the b-divisor corresponding to $M_0$, $B_{Y, 0}$ is not necessarily effective and 
\[
K_Y+B_{Y,0}+M_{Y,0}\sim_\rr N_Y.
\]
Defining
\begin{align}
\nonumber 
B'_Y := sB_{Y,0}+(1-s)B_Y-s(N_Y\wedge B_Y) 
\quad 
\text{ and } 
\quad 
M'_Y :=(1-s)M_Y + sM_{Y,0},
\end{align}
then, for $0<s\ll1$, $B'_Y$ is effective and $K_Y+B'_Y+M'_Y \sim_\rr sN'_Y \geq 0$.
\end{proof}

\begin{lemma}\label{lemm:X-toric}
Let $(X,B+M)$ be a projective generalized log canonical pair.
Assume that $K_X+B+M\sim_\qq 0$.
Let $n$ be an orbifold structure on $X$ and let $\Sigma$ be an orbifold decomposition of $B$.
If $\orbcomp(X,B+M;\Sigma)<1$, then $X$ is a toric variety.
\end{lemma}

\begin{proof}
By Lemma~\ref{lem:existence-gen-dlt-mod} and~\cite[Lemma 2.3.2]{BMSZ18}, we may assume that $(X, B+M)$ is a $\mathbb{Q}$ factorial gdlt pair with $K_X+B+M \sim_\qq 0$.
By Lemma~\ref{lemma:gen-pair-low-numerical}, for any real number $\alpha \in (0,1)$, there exists a projective generalized klt pair $(X,B_0+M_0)$ with numerical dimension zero, so that $\alpha B < B_0 < B$ and $ M_0= M+\overline{A}$, $A$ ample on $X$.
Moreover, by Lemma~\ref{lem:good-dlt-model}, up to replacing $X$ with a higher $\qq$-factorial dlt modification of $(X,B+M)$, we can assume that $K_X+B_0+M_0\sim_\rr N \geq 0$, where $N$ has no common component with $B_0$ and it contains no generalized lc center of $(X,B+M)$.
In particular, we can choose $0< \delta \ll 1$ and $\alpha$ close enough to $1$ such that, setting
\[
\Sigma_0:=\sum_{P\ {\rm prime}}\left(1-\frac{1}{n_P}\right) P
+\sum_{i=1}^k (1-\delta) b_i B_i,
\]
$\Sigma_0$ is an orbifold decomposition of $B_0$ and $\orbcomp(X,B_0+M_0;\Sigma_0)<1$.
Hence, the generalized klt pair $(X,B_0+M_0)$ has orbifold complexity $<1$ and numerical dimension $0$.
By~\cite[Lemma 4.4]{BZ16}, as $M_0$ is big, $(X,B_0+M_0)$ admits a good minimal model $X_m$ which is a Mori Dream Space.
We shall use the subscript $m$ to denote the strict transform of a divisor or the trace of a birational-divisor on $X_m$.
Then, the klt pair $(X_m,B_{0,m}+M_{0,m})$ is $\qq$-trivial with big boundary.
By Corollary~\ref{complexity.MMP}.\ref{complexity-under-MMP-orbifold}, we conclude that there exists an orbifold decomposition $\Sigma_{0,m}$ of $B_{0,m}$ with
\[
\orbcomp(X_m,B_{0,m}+M_{0,m};\Sigma_{0,m}) <1.
\]
Moreover, as $M_{0, m}$ is the pushforward of a big and nef divisor from a higher model, there exists $0 \leq D \sim_\qq M_{0,m}$ such that $(X_m,B_{0,m}+D)$ is klt and 
$\orbcomp(X_m,B_{0,m}+M_{0,m};\Sigma_{0,m})=\orbcomp(X_m,B_{0,m}+D;\Sigma_{0,m})$.
By Theorem~\ref{thm:proj-MDS-orbifold-complexity}, $X_m$ is a toric variety and all but $\lfloor \orbcomp(X_m,B_{0,m}+M_{0,m};\Sigma_{0,m}) \rfloor$ toric invariant divisors appear among the $B_{i, m}$.
Moreover,  as $X_m$ is a Mori Dream Space and $M_m$ is movable, then there exists $0 \leq D' \sim_\qq M_m$ such $(X_m,B_{0,m}+D)$ is klt and 
$\orbcomp(X_m,B_{m}+D')\leq \orbcomp(X_m,B_{0,m}+M_{0, m};\Sigma_{0,m})< 1$.
In particular, $\lfloor B_m \rfloor$ is a toric divisor, cf.~Theorem~\ref{thm:proj-MDS-orbifold-complexity}.
Note that, as we can choose $D'$ general enough, 
then by Theorem~\ref{thm:proj-MDS-orbifold-complexity} (3), the support of $B_m$ contains all but one toric prime divisors.
\\
To conclude, it suffices to prove that every exceptional divisor contracted by $X\dashrightarrow X_m$ corresponds to a toric valuation of $X_m$ which makes the content of the following lemma.
\end{proof}

\begin{lemma}
With the notation and assumptions of the proof of Lemma~\ref{lemm:X-toric}, the birational morphism $X\dashrightarrow X_m$ is an equivariant birational map of toric varieties.
\end{lemma}

\begin{proof}
We claim that every lc place of $(X_m,B_m)$ is a toric lc place.
In primis, as $(X, B+M)$ is assumed to be dlt, then $(X_m,B_m)$ is dlt outside the indeterminacy locus of $\psi \colon X_m\dashrightarrow X$.
Hence, if the center of a lc place $E$ of $(X_m, B_m)$ is not contained in the intedeterminacy locus of $\psi$, then it is a lc place also for $(X_m, \lfloor B_m \rfloor)$ and the latter is a toric pair.
On the other hand, any lc place $E$ of $(X_m,B_m)$ whose center is contained in the intedeterminacy locus of $\psi$ is also a glc place of $(X_m,B_m+M_m)$.
Since $X\dashrightarrow X_m$ is $(K_X+B+M)$-trivial, then $E$ is also a generalized lc place of $(X,B+M)$ whose center on $X$ is contained in $N$.
This leads to a contradiction since we assumed that $N$ does not contain any lc center of $(X, B+M)$.
Hence, every lc place of $(X_m,B_m)$ is toric.
\\
Assume that $N$ has the minimal number of irreducible components among the generalized pairs $(X,B_0+M_0)$ satisfying conditions (1)-(5) in the statement of Lemma~\ref{lem:good-dlt-model}.
If $N=0$, then $X$ is toric, since it is isomorphic in codimension one to $X_m$.
Hence, we may assume that $N$ is non-trivial.
Let $E$ be a component of $N$ which is not a toric lc place.
In particular, $E$ is not a lc place of $(X_m,B_m)$.
We show that we can apply~\cite[Lemma 5.2]{BMSZ18}.
Indeed, the following conditions are satisfied:
\begin{enumerate}
    \item By construction, $X_m$ is a $\qq$-factorial projective toric variety, 
    \item by Theorem~\ref{thm:proj-MDS-orbifold-complexity}, the support of $B_m$ contains all but one toric prime divisor, and
    \item by the first paragraph $(X_m,B_m)$ is log canonical and every log canonical center is a toric valuation.
\end{enumerate} 
Thus, by~\cite[Lemma 5.2]{BMSZ18}, we conclude that there exists $0\leq \Gamma_m \sim_\rr B_m$ such that $\nu(\Gamma_m)>\nu(B_m)$,
$(X_m,\Gamma_m)$ is log canonical and every lc place of $(X_m,\Gamma_m)$ is a toric valuation.
We denote by $\Gamma$ the strict transform of $\Gamma_m$ on $X$.
\\
We claim that we can choose $\Gamma_m$ so that $(X,\Gamma+M)$ is glc.
Note that 
\[
(X,B+\epsilon N +(1+\epsilon)M)
\] 
is gdlt for $\epsilon$ small enough.
Moreover, 
\[
\pi^\ast (K_{X_m}+B_m+(1+\delta)M_m) = K_X+B+(1+\delta)M_m +E_\delta,
\]
where $E_\delta$ is an effective divisor supported on $N$.
Furthermore, its coefficients are linear functions on $\delta$ and $E_0=0$.
Thus, for $\delta$ small enough, we have $E_\delta \leq \epsilon N$.
So, we conclude that $(X_m,B_m+(1+\delta)M_m)$ is glc.
Note that for $s$ small enough the generalized pair
\[
(X,(1-s)B+(1-s)E_\delta +s \Gamma +(1-s)(1+\delta)M)
\]
is glc.
Indeed, it is a linear combination of $(X,B+(1+\delta)M+E_\delta)$ and $(X,\Gamma)$, where the non-lc centers of the later are contained in the support of $N$.
In particular, we may choose $s_\delta =1-(1+\delta)^{-1}$ small enough, so that
\[
(X, (1-s_\delta )B+(1-s_\delta)E_\delta  + s_\delta \Gamma+ M)
\]
is glc.
Define $\Gamma':= (1-s_\delta) B + s_\delta \Gamma$.
Define $\Gamma'_m$ to be the pushforward of $\Gamma'$ to $X_m$.
Note that $(X,\Gamma'+M)$ is glc.
Observe that $\nu(\Gamma'_m) > \nu(B_m)$ and that every lc place of $(X_m,\Gamma'_m)$ is a toric valuation.
Thus, we may replace $\Gamma$ with $\Gamma'$, and $\Gamma_m$ with $\Gamma'_m$,
to achieve that $(X,\Gamma+M)$ is glc.
Note that no glc center of $(X,\Gamma+M)$ is contained in $N$.
\\
Write
\[
K_X+B =\pi^\ast (K_{X_m}+B_m)+E_B,
\text{ and }
K_X+\Gamma=\pi^\ast (K_{X_m}+\Gamma_m)+E_\Gamma.
\]
For any $\epsilon$, we can choose $s_\delta$ small enough so that
\[
-\epsilon N \leq E_B-E_\Gamma \leq \epsilon N.
\]
Let $E_t = t(E_\Gamma-E_B)+(1-t)N$, and write $E_t=E^{+}_t - E^{-}_t$.
Here, we may assume both
effective divisors $E_t^{\pm}$ have disjoint support, and satisfy $0\leq E^{\pm}_t \leq 2\epsilon N$, for $t$ close enough to one.
In particular, $(X,\Gamma+E^{-}_t + M)$ is glc, whenever $t$ is close enough to one.
We define
\[
B_t:=(1-t)B_0+t\Gamma+E^{-}_t, \text{ and } M_t:=(1-t)M_0+tM.
\]
Note that $(X,B_t+M_t)$ is generalized klt for $t$ close enough to one.
Then, we have that
\begin{align}
 \nonumber   K_X+B_t+M_t & = (1-t)(K_X+B_0+M_0)+t(K_X+\Gamma+M)+E^{-}_t \\
 \nonumber   & \sim_\qq (1-t)(K_X+B_0+M_0)+t(\Gamma - B) + E^{-}_t\\
 \nonumber  & \sim_\rr (1-t)N + t(E_\Gamma-E_B)+ E^{-}_t= E_t +E_t^{-} = E_t^{+}.
\end{align}
Since $\nu(\Gamma_m)>\nu(B_m)$, then the effective divisor $E_t^{+}$
has strictly less components than $N$.
For $t$ close enough to one, $(X,B_t+M_t)$ satisfies all the conditions of Lemma~\ref{lem:good-dlt-model}.
We lead to a contradiction in the minimality of the number of components of $N$.
Thus, $X$ must be a projective toric variety.
\end{proof}

\begin{proof}[Proof of Theorem~\ref{thm:formally toric-projective}]
Let $(X,B+M)$ be a projective generalized log canonical pair, such that $K_X+B+M\sim_\qq 0$.
Let $\Sigma$ be an orbifold decomposition of $B$.
Assume that $\orbcomp(X,B+M;\Sigma)<1$.
By Lemma~\ref{lemm:X-toric}, we know that $X$ is a projective toric variety.
Replacing $X$ with a small toric $\qq$-factorialization, we may assume it is $\qq$-factorial.
Since $M$ is the pushforward of a nef divisor on a higher birational model of $X$,
we conclude that its diminished base locus does not contain divisors.
Hence, we may run a $M$-MMP which terminates with a good minimal model for $M$.
We denote such a minimal model by $X_m$.
Observe that $X\dashrightarrow X_m$ only consists of toric flops.
Since $M_m$ is semiample, we may find $M_m\sim_\qq D_m$ so that $(X_m,B_m+D_m)$ is log canonical and $\qq$-trivial.
By Lemma~\ref{complexity.flops}.\ref{complexity-vs-flops-orbifold}, we know that there exists an orbifold decomposition $\Sigma_m$ of $B_m$ for which
\[
\orbcomp(X_m,B_m+D_m;\Sigma_m)=\orbcomp(X,B+M;\Sigma)=c.
\]
Hence, by~\cite[Theorem 1.2]{BMSZ18} we conclude that $c\geq 0$, $\lfloor B_m \rfloor$ is torus invariant,
hence $\lfloor B \rfloor$ is torus invariant as well.
Moreover, all but $\lfloor 2c \rfloor$ toric invariant prime divisors appear in $\Sigma_m$,
so the same statement holds for $\Sigma$.
Finally, if $c=0$, we claim that $D_m=0$.
Otherwise, we may add any component of $D_m$ as a summand of the decomposition $\Sigma_m$, since $\langle \Sigma \rangle = {\rm Cl}_\mathbb{Q}(X_m)$, thus obtaining an orbifold decomposition $\Sigma'_m$ such that
\[
\orbcomp(X_m,B_m+D_m;\Sigma'_m)<0,
\]
leading to a contradiction.
So, if $c=0$, we have $D_m\sim_\qq 0$, then $M\sim_\qq 0$ as well.
Let $\pi\colon X'\rightarrow X$ be a birational model of $X$, where $M$ descends.
By the negativity lemma we can write $\pi^\ast (M)=M'+E$, where $E$ is effective.
Since $M\sim_\qq 0$, we conclude $M'+E\sim_\qq 0$ and $M'$ is nef.
Assume that the effective divisor $E$ is non-trivial.
Intersecting $M'+E$ by an irreducible curve $C$ which intersect $E$ non-trivially, but is not contained in $E$,
we conclude that $M'\cdot C<0$.
This leads to a contradiction.
Thus, we conclude that $E$ is trivial, analogously we have $\pi^\ast (M)=M'$.
Hence, 
$M \sim_\mathbb{Q} 0$ as a b-divisor.
\end{proof}


\section{Applications of the generalized projective case}
\label{section:applications-projective}

In this subsection, we shall prove some applications of Theorem~\ref{thm:formally toric-projective}.
These statements will be used in the next sections to prove the local case of Theorem~\ref{formally toric}.
First, we will prove that, under certain conditions, a divisorial glc center of a generalized log canonical  pair of orbifold complexity zero must be normal.

\begin{lemma}
\label{lem:equality-sigma-B}
Let $(X,B+M)$ be a projective generalized log canonical pair with $K_X+B+M\sim_\rr 0$.
Let $n$ be an orbifold structure on $X$ and let $\Sigma$
be an orbifold decomposition of $B$ with $\orbcomp(X,B+M;\Sigma)=0$.
Let $B_1,\dots,B_k$ be the orbifold Weil divisors of $\Sigma$.
Then, the following conditions hold:
\begin{enumerate}
\item $X$ is a normal projective toric variety;
\item $\Sigma=B$;
\item $B_i=\frac{P_i}{n_{P_i}}$ for some $P_i\in X^1$; and,
\item the orbifold structure $n$ is supported on $\lfloor B\rfloor$, hence on the toric invariant divisors.
\end{enumerate}
In particular, letting $\Sigma_0$ be the prime decomposition of $B$, then
$\finecomp(X,B+M;\Sigma_0)=0=c(X,B+M;\Sigma_0)$.
\end{lemma}

\begin{proof}
\begin{itemize}
    \item [(1)-(2)]
By Theorem~\ref{thm:formally toric-projective}, we know that $X$ is a normal projective toric variety,  and the components of $\Sigma$ span the $\qq$-Class group of $X$.
If $\Sigma<B$, then increasing the coefficient of some prime component of $B$ in $\Sigma$, we can find an orbifold decomposition $\Sigma < \Sigma' \leq B$, with $\Sigma'> \Sigma$, $|\Sigma'|>|\Sigma|$ and
${\rm dim}_\qq \langle \Sigma \rangle = 
{\rm dim}_\qq \langle \Sigma' \rangle = 
{\rm dim}_\qq {\rm Cl}_\qq(X)$.
Hence, $\orbcomp(X,B+M;\Sigma')<0$, contradicting Theorem~\ref{thm:formally toric-projective}.
Thus, $\Sigma=B$.
    \item[(3)]
Writing
\[
\Sigma=\sum_{P\in X^1}\left(1-\frac{1}{n_P}\right)P 
+ \sum_{i=1}^k b_iB_i,
\]
if we assume that $B_1$ is not prime, then we can decompose  $B_1=B_{1,1}+B_{1,2}$, where $B_{1,1}$ and $B_{1,2}$ are non-trivial orbifold Weil divisors with distinct supports.
Then, the orbifold decomposition
\[
\Sigma' :=
\sum_{P\in X^1}
\left(
1-\frac{1}{n_P}
\right) P+ 
b_1 B_{1,1}+
b_1 B_{1,2}+
\sum_{i=2}^k b_i B_i
\]
has negative orbifold complexity, which leads to a contradiction.
Thus, $B_1=\frac{1}{a}P_1$, for some $P_1 \in X^1$ and for some positive integer $a$ dividing $n_{P_1}$.
If $a< n_{P_1}$, then the orbifold decomposition
\[
\Sigma' =
\sum_{P\in X^1}
\left(
1-\frac{1}{n_P}
\right) 
P+ 
\left(
\frac{n_{P_1}}{a}b_1 
\right) 
\frac{P_1}{n_{P_1}} 
+\sum_{i=2}^k b_i B_i
\]
has negative orbifold complexity since $\frac{n_p}{a}b_1 > b_1$ contradicting Theorem~\ref{thm:formally toric-projective}.
Then, $B_1=\frac{P_1}{n_{P_1}}$.
    \item[(4)]
Let $P \in X^1$ such that $n_P>1$.
\\
Let us assume first that ${\rm coeff}_P(B) > 1-\frac{1}{n_P}$.
As $\Sigma=B$ then ${\rm coeff}_P(\Sigma)={\rm coeff}_P(B)$ and,
up to relabeling and summing the $B_i$, we may assume that $B_1=\frac{1}{n_P} P$ and $\textrm{supp}(B_j) \neq P$, for all $j >1$.
Hence,
\[
{\rm coeff}_P(\Sigma) =
\left(
1-\frac{1}{n_P}
\right) + 
\frac{b_1}{n_P}.
\]
If ${\rm coeff}_P(B)< 1$, then $b_1 < 1$ and 
\begin{align}
\label{eqn:ineq.coeff.<1}
\left(1-\frac{1}{n_P}\right) 
+ \frac{b_1}{n_P} > b_1.
\end{align}
Defining the orbifold structure $n'$ by 
\[
n'(Q):=\begin{cases}
1 & \quad Q=P\\
n_Q & \quad Q \neq P
\end{cases},
\]
for $0< \epsilon \ll 1$ the decomposition 
\[
\Sigma'' =
\sum_{Q\in X^1}
\left(
1-\frac{1}{n'_Q}
\right) Q+ 
(b_1+\epsilon) P+
\sum_{i=s+1}^k b_i B_i
\]
is an orbifold decomposition of $B$ for the orbifold structure $n'$ and $\orbcomp(X, B; \Sigma'')=-\epsilon<0$, which gives the sought contradiction.
\\
If ${\rm coeff}_P(B) = 1-\frac{1}{n_P}$, then $P \neq P_i$ for all $i$, where $P_i \in X^1$ is the support of the orbifold Weil divisor $B_i$.
Defining 
\[
\Sigma^3:= \sum_{i=1}^k b_i P_i,
\] it follows from~\eqref{eqn:ineq.coeff.<1} that $\Sigma^3$ is a decomposition of $B$.
By construction and (1-3),
$\finecomp(X, B+M; \Sigma^3)=\orbcomp(X, B+M; \Sigma)=0$.
Hence, the $P_i$ span ${\rm Cl}_\qq(X)$ and defining
\[
\Sigma^4:= \Sigma^3+\epsilon P,
\quad 0< \epsilon \ll 1,
\]
then also $\Sigma^4$ is a decomposition for $B$, by our initial assumption on $P$, and $\finecomp(X, B+M; \Sigma^4)<\finecomp(X, B+M; \Sigma^3)=0$, which leads to a contradiction.
Hence, ${\rm coeff}_P(B)={\rm coeff}_P(\Sigma)=1$ holds for any $P \in X^1$ in the support of $n$.
By Theorem~\ref{thm:proj-MDS-orbifold-complexity}.2, $\lfloor B \rfloor$ is supported on the toric invariant part.
\\
Finally, as $n$ is supported on $\lfloor B \rfloor$, then $\orbcomp(X,B+M;\Sigma)=\finecomp(X,B+M;\Sigma)=c(X,B+M;\Sigma)$.
Since ${\rm dim}_\qq \langle \Sigma \rangle = {\rm dim}_\qq {\rm Cl}_\qq(X)$, then $c(X,B+M;\Sigma_0) \leq c(X,B+M;\Sigma)$, thus, equality must hold.
That implies that also $\finecomp(X,B+M;\Sigma_0)=0$, by Lemma~\ref{lem:basic.ineq}.

\end{itemize}
\end{proof}

\begin{lemma}
\label{lem:gen-slc-case}
Let $(X/Z,B+M)$ be a $\qq$-factorial generalized log canonical pair and let $z \in Z$ be a closed point.
Assume the following conditions hold:
\begin{enumerate}
    \item 
$K_X+B+M\sim_{\qq,Z}0$;
    \item 
$\orbcomp_z(X/Z,B+M)\leq 0$;
    \item 
$(X/Z,B+M)$ has a divisorial glc center $E$ contained in the fiber over $z$; and,
    \item 
all prime components of $B-E$ intersect $E$ non-trivially.
\end{enumerate}
Then, $E$ is a normal variety.
\end{lemma}

\begin{proof}
Let $\Sigma$ be an orbifold decomposition  of $B$ such that $\orbcomp_z(X/Z,B+M;\Sigma)\leq 0$.
Let $(E,B_E+M_E)$ be the generalized pair obtained by adjunction of $(X,B+M)$ to $E$.
$(E,B_E+M_E)$ is generalized semi-log canonical.
In particular, those codimension one points of $E$ along which $E$ is not normal are nodal points.
Furthermore, the prime components of $B$ intersect $E$ along codimension one points of $E$ contained in its normal locus.
\\
Let us now assume for the sake of contradiction that $E$ is not normal.
Each irreducible component of the conductor of $E$ is a log canonical center of $(X,B)$, thus, also a glc center of $(X,B+M)$.
Let $D_1,\dots,D_k$ be the irreducible components of the conductor of $E$.
By~\cite[Theorem 1]{Mor19}, we can extract a unique divisorial log canonical center of $(X,B)$ centered at $D_1$:
we denote by $X_1\rightarrow X$ be said extraction and by $F_1$ its exceptional divisor.
The divisor $F_1$ is a glc place for $(X,B+M)$, and the cited theorem guarantees that $X_1$ is $\qq$-factorial.
Proceeding inductively, we construct a sequence of birational morphisms 
\[
X_k\rightarrow X_{k-1}\rightarrow \dots \rightarrow X_1\rightarrow X
\]
by successively extracting a glc place centered at each of the $D_i$.
Let $F_i$ be the prime divisor of generalized log discrepancy zero for $(X,B+M)$ extracted at the $i$-th step $X_i \to X_{i-1}$.
Abusing notation, we will denote by $F_i$ its strict transform on $X_j$, for $j>i$.
\\
Let $(X_k,B_{X_k}+M_{X_k})$ be the log pull-back of $(X,B+M)$ to $X_k$ and let $E_k$ be the strict transform of $E$ on $X_k$;
let $(E_k,B_{E_k}+M_{E_k})$ be the generalized log canonical pair obtained by adjunction of $(X_k,B_{X_k}+M_{X_k})$ along $E_k$.
As $(X_k,E_k)$ is lc and $\mathbb Q$-factorial, $X_k$ is klt in a neighborhood of $E_k$ and $E_k$ is $S_2$ by~\cite[Corollary 5.25]{KM98}.
On the other hand, by construction, $E_k$ is smooth in codimension 1 as $E_k$ is seminormal and dominates $E$, and for all $i$ the intersection $F_i\cap E_k$ must have at least two prime components, by~\cite[Proposition 16.4]{Kol92}.
By Lemma~\ref{complexity-dlt} there exists an orbifold decomposition $\Sigma_{X_k}$ of $B_{X_k}$ for which $\orbcomp_z(X_k,B_{X_k}+M_{X_k};\Sigma_{X_k})\leq 0$
and the $F_i$ appear in $\Sigma$ with coefficient $1$.
By Lemma~\ref{lem:complexity-vs-adjunction} and Remark~\ref{rmk:decomposition.div.lc.center}, we can assume that $E_k$ appears in $\Sigma_k$ with coefficient $1$ and that there exists an orbifold decomposition $\Sigma_{E_k}$ for which
\begin{align}
\label{eq:complexity-for-normality}
\orbcomp(E_k,B_{E_k}+M_{E_k};\Sigma_{E_k}) \leq 
\orbcomp(X_k,B_{X_k}+M_{X_k};\Sigma_{X_k})
\leq 0.
\end{align}
By Theorem~\ref{thm:formally toric-projective}, $\orbcomp(E_k,B_{E_k}+M_{E_k};\Sigma_{E_k})= 0$.
By Step 3 in the proof of Lemma~\ref{lem:complexity-vs-adjunction}, the divisors $F_i\vert_{E_k}$ appear as orbifold Weil divisors of the orbifold decomposition
$\Sigma_{E_k}$.
On the other hand, by Lemma~\ref{lem:equality-sigma-B}.3, each orbifold Weil divisors of $\Sigma_{E_k}$ must be prime.
However, since for all $i$ $F_i \cap E_k$ is supported at at least $2$ codimension one points of $E_k$, that leads to the sought contradiction.
\end{proof}

Now, we turn to prove a first partial version of Theorem~\ref{formally toric}: namely, we show that the orbifold complexity is always non-negative for log canonical pairs that are relatively trivial over the base.
This implies that the same holds for all complexities.

\begin{theorem}
\label{thm:inequality-cgeq0}
Let $(X/Z,B+M)$ be a generalized log canonical pair, such that $K_X+B+M\sim_{\rr,Z} 0$.
Let $z\in Z$ be a closed point. 
Then,
\begin{enumerate}
    \item 
$\abscompz(X/Z,B+M)\geq \finecomp_z(X/Z,B+M)\geq \orbcomp_z(X/Z,B+M)\geq 0$; and,
    \item 
\label{thm:implications-c=0}
if $\orbcomp_z(X/Z,B+M)=0$,
there exists a glc center of $(X/Z,B+M)$ contained in the fiber over $z\in Z$.
\end{enumerate}
\end{theorem}

\begin{proof}
To prove $(1)$ it suffices to show that $\orbcomp_z(X/Z,B+M;\Sigma)\geq 0$.
Hence, we shall assume, for the sake of contradiction, that there exists an orbifold structure $n$ on $X$ and an orbifold decomposition $\Sigma$ of $B$ such that $\orbcomp_z(X/Z,B+M;\Sigma)<0$.
By Lemma~\ref{lem:cut-down-lcc}, there exists $B' \geq 0$ such that $(X,B+B'+M)$ has a glc center contained in $\pi^{-1}(z)$ and $K_X+B+B'+M \sim_{\rr, Z}0$, $\orbcomp_z(X/Z,B+B'+M; \Sigma)<0$.
By Lemma~\ref{lem:existence-gen-dlt-mod}, there exists a $\qq$-factorial gdlt modification $(Y,B_Y+M_Y)$ of $(X,B+B'+M)$ over $Z$ with a divisorial glc center $E \subset \pi^{-1}(z)$.
Lemma~\ref{complexity-dlt}.\ref{complexity-vs-dlt-modification-orbifold} implies that there exists an orbifold decomposition $\Sigma_Y$ of $B_Y$ with $\orbcomp_z(Y,B_Y+M_Y;\Sigma_Y)<0$, and $E$ appears in $\Sigma_Y$ with coefficient $1$.
By Lemma~\ref{lem:extraction-gslc-pair} and Remark~\ref{rem:gdlt}, there exists a birational contraction $Y \dashrightarrow Y'$ over $Z$ which is a composition of divisorial contractions and small maps and a generalized log canonical pair $(Y'/Z,B_{Y'}+M_{Y'})$ satisfying the following conditions:
\begin{enumerate}
    \item 
$(Y'/Z,B_{Y'}-\epsilon E'+M_{Y'})$ is $\qq$-factorial gdlt for any $0< \epsilon\ll 1$, where $B_{Y'}$ (resp. $E', M_{Y'}$) is the strict transforms of $B_Y$ (resp. $E, M_Y$) on $Y$;
    \item 
$K_{Y'}+B_{Y'}+M_{Y'} \sim_{\rr, Z} 0$;
    \item 
the fiber of $Y'$ over $z \in Z$ is supported on $E'$; and,
    \item 
there exists a generalized semi-log canonical pair $(E', B_{E'}+M_{E'})$ satisfying the adjunction formula $(K_{Y'}+B_{Y'}+M_{Y'})|_{E'} =K_{E'}+B_{E'}+M_{E'}$.
\end{enumerate} 
By Corollary~\ref{complexity.MMP}.\ref{complexity-under-MMP-orbifold}, there exists an orbifold decomposition $\Sigma_{Y'}$ for which $\orbcomp_z(Y',B_{Y'}+M_{Y'};\Sigma_{Y'})<0$.
By Lemma~\ref{lem:gen-slc-case}, $E'$ is normal, thus, $(E',B_{E'}+M_{E'})$ is a glc log Calabi--Yau pair.
By Lemma~\ref{lem:complexity-vs-adjunction}, there exists an orbifold decomposition $\Sigma_{E'}$ of $B_{E'}$ with $\orbcomp(E',B_{E'}+M_{E'};\Sigma_{E'}) <0$: this contradicts Theorem~\ref{thm:formally toric-projective}.
\\
If $\orbcomp_z(X/Z,B+M; \Sigma)=0$, then 
\begin{align*}
\orbcomp_z(X/Z,B+B'+M)=
\orbcomp_z(X/Z,B+M)=0
\end{align*}
and the last claim of Lemma~\ref{lem:cut-down-lcc} implies that $B'=0$, thus proving $(2)$.
\end{proof}

\begin{corollary}
\label{cor:compl.pt}
Let $(X,B+M)$ be a generalized log canonical pair and $x\in X$ be a closed point.
Then, 
\[
c_x(X, B+M) 
\geq 
\finecomp_x(X,B+M)
\geq 
\orbcomp_x(X,B+M)
\geq 0.
\]
\end{corollary}
\begin{proof}
It suffices to observe that if $n$ is an orbifold structure on $X$ and
\[
\Sigma= \sum_{p \in X^1} \left(
1-\frac{1}{n_P}\right) P
+ \sum_{i=1}^r b_iB_i
\]
is an orbifold decomposition of $B$ for $n$ such that $\orbcomp_x(X, B+M; \Sigma)<0$, then there exists an open set $x \in U \subset X$ such that the dimension of the span of the $B_i$ in ${\rm Cl}_\qq(U)$ is the same as that of the span of the $B_i$ in $\clxx$, by the very definition of $\clxx$, cf. Definition~\ref{def:cl.gr}.
Thus, $\orbcomp_x(X, B+M; \Sigma)=\orbcomp_x(U/U, B\vert_U+M\vert_U; \Sigma) <0$, where for $\orbcomp_x(U/U, B\vert_U+M\vert_U; \Sigma)$ we consider the identity map of $U$ as the structure morphism.
This contradicts Theorem~\ref{thm:inequality-cgeq0}.
\end{proof}

\section{Formally toric plt blow-ups}
\label{sect:plt.blow.ups}
In this section, we will introduce the concept of formally
toric plt blow-ups.
We will prove that a log canonical germ of orbifold complexity zero admits a formally toric $\qq$-factorial plt blow-up (Proposition~\ref{prop:toroida-blow-up}).
We start by recalling the concept of a plt blow-ups
and introducing the concept of formally toric plt blow-ups.
We start by recalling the definition of plt blow-ups (see, e.g.,~\cite[Lemma 1]{MR3187625}).

\begin{definition}
\label{def:plt.blowup}
{\em 
Let $(x\in X,\Gamma)$ be a germ of a klt singularity.
\begin{enumerate}
    \item 
A {\em plt blow-up} for $(X,\Gamma)$ at $x$ is a projective birational morphism $\pi \colon Y \rightarrow X$ such that
\begin{itemize}
    \item 
$\pi$ extracts a unique divisor $E$ and $E=\pi^{-1}(x)$; 
    \item 
$-E$ is ample over the base; and,
    \item 
$(Y,\pi_\ast^{-1}\Gamma+E)$ is plt.
\end{itemize}
    \item 
A {\em $\qq$-factorial plt blow-up} for $(X,\Gamma)$ at $x$ is a projective birational morphism $\pi\colon Y\rightarrow X$ such that
\begin{itemize}
    \item 
$Y$ is $\qq$-factorial;
    \item 
$\pi$ extracts a unique divisor $E$ and $E=\pi^{-1}(x)$;
    \item 
$-E$ is nef over $X$; and,
    \item 
$(Y,\pi_\ast^{-1}\Gamma+E)$ is plt.
\end{itemize}
\end{enumerate}
}
\end{definition}

The existence of $\qq$-factorial plt blow-ups follows from the existence of plt blow-ups,~\cite[Lemma 1]{MR3187625}, and small $\qq$-factorializations for plt pairs,~\cite[Corollary 1.4.3]{BCHM10}.
Moreover, it is a completely local matter, that is, the construction only depends on the germ of the singularity and we are free to shrink to a smaller neighborhood of $x \in X$.
Let us also recall that a $\qq$-factorial plt blow-up is also a relative Mori dream space by~\cite[Corollary 1.3.2]{BCHM10}.

We now define a toroidal version of a $\qq$-factorial plt blow-up:
that will correspond to a special plt extraction in which the exceptional divisor $E$ is toric and the torus invariant divisors on $E$ can be lifted to irreducible divisors on the total space of the plt blow-up, with nice singularities.

\begin{definition}
\label{def:form.tor.plt.blowup}
{\em
Let $(x \in X,\Gamma)$ be a germ of a klt singularity.
Let $\pi \colon Y\rightarrow X$ be a $\qq$-factorial plt blow-up for $(X,\Gamma)$ at $x\in X$ with exceptional divisor $E$.
The morphism $\pi$ is {\em a formally toric plt blow-up} for $(X,\Gamma)$ at $x$ if the following conditions hold:
\begin{enumerate}
    \item 
$(X, \lceil \Gamma \rceil)$ is a log canonical pair at $x$;
    \item 
$E$ is an lc place of $(X, \lceil \Gamma \rceil)$;
    \item 
$E$ is a projective toric variety and the divisor $\Gamma_E$ defined by the adjunction
\[
(K_{Y}+E+ \pi_\ast^{-1}\lceil \Gamma \rceil)\vert_E =
K_E+\Gamma_E
\]
is the reduced sum of all torus invariant divisors; and,

    \item 
each prime component of $\pi_\ast^{-1}\lceil \Gamma \rceil$ restricts on $E$ to a prime component of $\Gamma_E$.
\end{enumerate}
}
\end{definition}

\begin{remark}
\label{rem:gamma_i} 
{\em 
We adopt the notation of Definition~\ref{def:form.tor.plt.blowup}.
\begin{enumerate}
    \item 
Condition $(2)$ implies that 
$K_Y+ \pi_\ast^{-1}\lceil\Gamma\rceil + E= \pi^\ast (K_X+\lceil\Gamma\rceil)$ and $(Y, E+ \pi_\ast^{-1}\lceil \Gamma \rceil)$ is lc.
Hence, $K_E+\Gamma_E \sim 0$ and $E \setminus \supp(\Gamma_E)$ is a torus acting equivariantly on $E$.
    \item 
Let $E_i$ be a component of $\pi_\ast^{-1}\lceil \Gamma \rceil$ and let $T_i=\supp(E_i\vert_E)$.
By condition $(3)$, $T_i$ is a torus invariant prime component of $\Gamma_E$ and
\[
E_i\vert_E = \frac{1}{m_i} T_{i},
\]
where $m_i$ is the Cartier index of $E$ at $T_i$.
\end{enumerate}
}
\end{remark}

In the next three sections, we will use the following technical result for $\qq$-factorial plt blowups.

\begin{lemma}
\label{lem:linear-equivalence.gen}
Let $\pi \colon Y \to X$ be a $\qq$-factorial plt blow-up of exceptional divisor $E$.
Let $D$ and $D'$ be two Weil divisors on $Y$ whose support do not contain $E$.
If $D\vert_E \sim_\qq D'\vert_E$, then $[D-D']s \in \clxx_{\rm tor}$.
\end{lemma}

\begin{proof}
Since $Y$ is $\qq$-factorial, then
$D-D'$ is a $\qq$-Cartier divisor.
Thus, $D-D'$ is a Weil $\qq$-Cartier divisor on $Y$ which intersect 
all projective curves on $E$ trivially.
As $\pi$ is a relative Mori dream space, we can run a $(D-D')$-MMP over the base. 
Since $Y\rightarrow X$ is birational, this minimal model program
must terminate with a good minimal model for $D-D'$ over $X$.
Since this divisor is trivial on the fiber over $x$, we assume this MMP is an isomorphism over a neighborhood of such point.
Hence, $D-D'$ is semiample over a neighborhood of $x$.
By taking the ample model over $X$, we conclude that
$m(D-D')\sim 0$ for some $m \in \mathbb{N}_{>0}$ over a neighborhood of $x$:
indeed, the ample model must be an isomorphism over a neighborhood of $x$ since it must contract $E$, by construction, which proves our claim. 
\end{proof}

This simple observation immediately yields the following useful corollary.

\begin{corollary}
\label{cor:generation}
Let $\pi \colon Y \to X$ be a formally plt blow-up of a klt germ $(x\in X, \Gamma)$. 
Then, the components of $\Gamma$ span $\clxxq$.
\end{corollary}

\begin{proof}
Let $W$ be a Weil divisor on $X$ and let $\widetilde W$ be its strict transform.
Then since the components $E_1, \dots, E_r$ of $\pi_\ast^{-1} \lceil \Gamma \rceil$ restrict to the prime components of the torus invariant divisor $\Gamma_E$ on $E$, then there exists rational numbers $a_1, \dots, a_r$ such that $\widetilde W\vert_E \sim_\qq \sum_{i=1}^r a_i E_i\vert_E$.
Hence, by Lemma~\ref{lem:linear-equivalence.gen}, 
$\widetilde W - (\sum_{i=1}^r a_i E_i) \sim_{\qq, X} 0$.
Pushing forward to $X$, then $W$ is $\qq$ linear equivalent, over a suitable neighborhood of $x \in X$ to a sum of components of $\Gamma$.
\end{proof}

From now on, we turn to prove that a germ of a generalized log canonical pair of complexity zero admits
a formally toric plt blow-up.

\begin{proposition}
\label{prop:toroida-blow-up}
Let $(Y/X,B_Y+M_Y)$ be a generalized log canonical pair.
Let $x\in X$ be a closed point and let $B:= \pi_\ast B_Y$.
Assume that:
\begin{itemize}
	\item 
$\pi \colon Y\rightarrow X$ is birational; 
	\item
$K_Y+B_Y+M_Y \sim_{\qq, X} 0$; and,
	\item
$\orbcomp_x(Y/X,B_Y+M_Y)=0$.
\end{itemize}
Then,
\begin{enumerate}
    \item 
the b-divisor $M$ descends to a neighborhood of $x\in X$. 
In particular, $M\sim_{\qq, X} 0$ in the sense of b-divisors;
    \item 
there exists a klt pair $(X,\Gamma)$ and $\Gamma < \lceil B \rceil$, $\lfloor B\rfloor\leq \lceil \Gamma\rceil$;
and,
    \item 
the germ $(x\in X,\Gamma)$ admits a formally toric plt blow-up.
\end{enumerate}
\end{proposition}

\begin{proof}
For the reader's convenience, we divide the proof into several steps.\\

\noindent
{\it Step 1}. In this step, we make some basic reductions.\\

\noindent
By Lemma~\ref{complexity-dlt}.\ref{complexity-vs-dlt-modification-orbifold}, we may assume that $(Y,B_Y+M_Y)$ is $\qq$-factorial gdlt.
By Theorem~\ref{thm:inequality-cgeq0}.\ref{thm:implications-c=0}, we may assume that $x$ is a glc center of $(X,B+M)$.
Replacing $(Y,B_Y+M_Y)$ with the birational model constructed in the proof of Theorem~\ref{thm:inequality-cgeq0}, we may assume that $(Y,B_{Y}+M_{Y})$ is $\qq$-factorial generalized log canonical and that the fiber over $x$ consists of a unique divisorial glc center $E$ which is normal, see also Lemma~\ref{lem:gen-slc-case}.
By Corollary~\ref{complexity.MMP}.\ref{complexity-under-MMP-orbifold},  $\orbcomp_x(Y,B_Y+M_Y)$ can only decrease in performing such substitution, thus such substitution does not alter the assumptions of the proposition.
\\
Let $(E,B_E+M_E)$ be the generalized log canonical pair obtained by adjunction to $E$.
By Proposition~\ref{prop:complexity-minimum} and Lemma~\ref{lem:complexity-vs-adjunction}, there exists an orbifold structure $n$ (resp. $m$, cf.~\eqref{def:orb.struct.m.adj}) on $Y$ (resp. $E$) together with an orbifold decomposition $\Sigma_Y\leq B_Y$ (resp. $\Sigma_E\leq B_E$) for $n$ (resp. $m$) such that
\begin{align}
\label{eqn:ineq.adj.1}
0 \geq
\orbcomp_x(Y/X,B_Y+M_Y;\Sigma_Y) \geq
\orbcomp(E,B_E+M_E;\Sigma_E).
\end{align} 
By Theorem~\ref{thm:formally toric-projective}, $\orbcomp(E,B_E+M_E; \Sigma_E)=0$ and equality must hold in~\eqref{eqn:ineq.adj.1}.
Hence, $E$ is a projective normal toric variety
and $\lfloor B_E\rfloor$ is contained in the torus invariant boundary;
$M_E\sim_\qq 0$ and the nef part of $(E, B_E +M_E)$ descends onto $E$, in the sense of b-divisors; $\langle \Sigma_E \rangle = {\rm Cl}_\mathbb{Q}(E)$ by Proposition~\ref{prop:gen-class-group}.
\\
By Remark~\ref{rmk:decomposition.div.lc.center}, we can assume that $n_E=1$ and 
\[
\Sigma = \sum_{P \in X^1} \left( 1 - \frac{1}{n_P}\right) P + \sum_{i=1}^k b_i B_i
\]
with $B_1=E$, $b_1=1$, and $E \not \subseteq \supp(B_i)$, for $i \geq 2$.
By construction, we can assume that the orbifold structure $m$ on $E$ is the one defined in~\eqref{def:orb.struct.m.adj}, and that the orbifold decomposition $\Sigma_E$ of $B_E$ is of the form
\[
\Sigma_E :=\sum_{Q\in E^1}
\left(1-\frac{1}{m_Q}\right)Q
+\sum_{i=2}^kb_i B_i\vert_E.
\]

\noindent
{\it Step 2.} In this step, we prove that for each torus invariant prime divisor $T \subset E$, there exists $i \geq 2$ such that $\supp(B_i \cap E)=T$. 
We also prove some additional properties of $B_i$.
\\ 

\noindent
By Lemma~\ref{lem:equality-sigma-B}, $\Sigma_E$ must actually be the prime decomposition of $B_E$ and the set $\mathcal{S} \subset E^1$ defined in~\eqref{eqn:def.S} must be empty, by Remark~\ref{rem:equality-adjunction}, since, as already observed above, $\langle \Sigma_E \rangle = {\rm Cl}_\mathbb{Q}(E)$ and the orbifold complexity is $0$.
For all $i$ the support of $B_i\vert_E$ is irreducible, again by Lemma~\ref{lem:equality-sigma-B}.
Given $T \in E^1$ a torus invariant divisor, by Theorem~\ref{thm:formally toric-projective}, there exists $i\in \{2,\dots,k\}$ such that $\supp(B_i\vert_E)=T$.
\\ 

\noindent
{\it Step 3.}
In this step, we prove that for each torus invariant prime divisor $T$ on $E$,  there exists $E_T \in Y^1$ such that $\supp(E_T \cap E)=T$ and $E_T\vert_E=\frac{1}{i_{T}}T$, where $i_{T}$ is the Cartier index of $E$ at the generic point of $T$.
\\ 
Let $T\subset E^1$ be a torus invariant divisor.
By the previous step, there exists a natural number $i \geq 2$ such that $\supp(B_i\vert_E)=T$.
Applying Lemma~\ref{lem:equality-sigma-B}, then, since $B_i\vert_E$ is an orbifold Weil divisor, 
\begin{equation}
\label{eq:pull-back-of-bi} 
B_i\vert_E =
\frac{1}{m_T}T.
\end{equation} 
Let us assume that $m(T)=i_T$; 
this case only happens if the support of $n$ does not contains $T$, 
cf.~\eqref{def:orb.struct.m.adj}, 
in which case $B_i$ is a Weil divisor on $Y$.
Furthermore,~\eqref{eq:pull-back-of-bi} holds only if $B_i$ has a unique prime component through $T$, since $m(T)=i_T$.
Hence, it suffices to take $E_T=B_i$.
\\
If, instead, the support of $n$ contains $T$, then there exists a unique $P_T \in Y^1$ with $n_{P_T}>1$ containing $T$;
furthermore, $m(T)=i_T n_{P_T}$, see Step 2 of the proof of Lemma~\ref{lem:complexity-vs-adjunction}. 
Writing
\[
B_i =\frac{{\rm num}_{P_T}(B_i)}{n_{P_T}}P_T+W_i, 
\quad 
{\rm num}_{P_T}(B_i) \in \mathbb{Z}_{>0},
\quad 
P_T \not \subset \supp W_i,
\] 
we can assume that $W_i$ is a Weil divisor around $T$, so that
\[
B_i\vert_E= 
\left(
\frac{{\rm num}_{P_T}(B_i) \gamma_{_{P_T}, T}}{n_{P_T} i_T} +\frac{n_{P_T}\gamma_{W, T}}{n_{P_T} i_T}
\right)T,
\qquad 
\gamma_{P, T},\
\gamma_{W, T} \in \zz_{\geq 0},
\]
which, by~\eqref{eq:pull-back-of-bi}, implies that 
\[
{\rm num}_{P_T}(B_i)=1=\gamma_{P_T, T}
\quad 
\text{ and }
\quad 
\gamma_{W, T}=0,
\]
Thus, $B_i=\frac{1}{n_{P_T}}P_T$.
Hence, $P_T\vert_E=\frac{1}{i_T}T$ around $T$ and it suffices to define $E_T:=P_T$.
\\

\noindent
{\it Step 4}. 
In this step, we prove that the nef part of $(Y, B_Y+M_Y)$ descends onto $Y$ over a neighborhood of $E$.
Moreover, $M_Y\vert_E\equiv 0$.\\

\noindent
Let us fix a sufficiently high model $\pi' \colon Y' \to Y$ which is a log resolution of $(Y, B_Y)$ and on which the nef part of $(X, B+M)$ descends to the divisor $M_{Y'}$ nef over $X$.
We can also assume that the nef part of $(E, B_E+M_E)$ descends on $E'$, the transform of $E$ on $Y'$, and its trace $M_{E'}$ is given by the restriction of $M_{Y'}\vert_{E'}$.
As we know that the nef part of $(E, B_E+M_E)$ on $E$ is $\qq$-trivial as a b-divisor by Theorem~\ref{thm:formally toric-projective}, then $M_{E'}:=M_{Y'}\vert_{E'} \sim_{\qq} 0$.
On the other hand, $M_{Y}\vert_E \equiv 0$, as well, since otherwise by Remark~\ref{rem:res-MY} we could find an orbifold decomposition $\Sigma'_E > \Sigma_E$ of $B_E$ and $\orbcomp(E, B_E+M
_E; \Sigma'_E)<0$ which would give a contradiction.
We shall denote by $Y'_x$ the fiber of $Y'$ over $x$;
$Y'_x$ contains $E'$.
By the negativity lemma, as $M_{Y'}$ is nef over $X$, hence also over $Y$,
\begin{align}
\label{eqn:neg.lemma}    
{\pi'}^\ast (M_Y)=M_{Y'}+F,
\quad 
F \geq 0
\end{align}
where $F$ is $\pi'$-exceptional.
That implies, in particular, that the nef part of $(X, B+M)$ descends to $M_Y$ on $Y \setminus \pi'(\supp F)$.
Thus, if $F\cap Y'_x=\emptyset$, then the nef part descends over a neighborhood of $E$.
On the other hand, if $F\cap Y'_x\neq \emptyset$, as $F \geq 0$ and $\supp(F)$ does not contain the whole fiber $Y'_x$, then there exists a curve $C\subset Y'_x$ such that $C\cdot F>0$.
By the projection formula, since on $Y$ the fiber over $x$ is completely supported on $E$ and $M_Y\vert_E \equiv 0$, then $C\cdot {\pi'}^\ast (M_Y)=0$, so that $M_{Y'}\cdot C<0$, contradicting the nefness of $M_{Y'}$ over $X$.
Thus, $F\cap Y'_x=\emptyset$ and $M$ descends over a neighborhood of $E$.
\\

\noindent
{\it Step 5}. 
Let $T_1, \dots, T_r$ be the toric invariant divisors of $E$ and let $E_i := E_{T_i}$, $i=1, \dots, r$ be the divisors defined in Step 3.
In this step, we show that the log pair $(Y, E+ \sum_{i=1}^r E_i)$ is log canonical around $E$
and 
\[
(K_Y+ E+ \sum_{i=1}^r E_i)\vert_E= K_E+\sum_{i=1}^r T_i.
\]
\\
By adjunction, since $\mathcal S =\emptyset$ and hence $m_Q=i_Q$ for all $Q$ in the support of the orbifold structure defined in Step 1, 
\begin{equation}\label{adjunction-of-cq}
(K_Y+E)|_E = 
K_E +\sum_{Q \in E^1}\left( 1-\frac{1}{i_Q}\right) Q \leq K_E+B_E.
\end{equation} 
If $i_Q>1$, then $Q$ is a torus invariant divisor by Lemma~\ref{lem:equality-sigma-B}. 
Let $T_1,\dots,T_r$ be the torus invariant divisors of $E$.
We will denote by $E_i$ the divisor $E_{T_i} \in Y^1$ constructed in Step 3.
Then, by Step 3
\[
(
K_Y+E+\sum_{i=1}^rE_i
)|_E = K_E +\sum_{i=1}^s\left(
1-\frac{1}{i_{T_i}}\right)T_i + \sum_{i=1}^s\frac{1}{i_{T_i}}T_i
+ \sum_{i=s+1}^r T_i =
K_E +\sum_{i=1}^r T_i
\sim 0
\]
is a log canonical pair.\\

\noindent
{\it Step 6.} In this step, we prove that there exists an effective divisor $D_Y$ so that $(Y,D_Y+E)$ is plt around $E$ and $(K_Y+E+D_Y)\vert_E\equiv 0$.
\\

\noindent 
Let $H$ be an ample Cartier divisor on $Y$.
Then, $H_E:=H|_E$ is an ample Cartier divisor on $E$ and there exists positive integers $a_i$, $i=1, \dots, r$ such that
\[
H_E \sim_\qq \sum_{i=1}^r a_i T_i, 
\]
Let us fix a positive rational $0<\epsilon\ll 1$
and $0 \leq D_\epsilon\sim_{\qq, X} \epsilon H$ a general element of the relative $\qq$-linear system of $\epsilon H$. 
Then, the log pair $(E, \sum_{i=1}^r(1-\epsilon a_i)T_i+D_\epsilon|_E)$ is klt.
Defining the divisor
\[
D_Y:=\sum_{i=1}^r (1- \epsilon a_i m_i)E_i + D_\epsilon, 
\]
then by adjunction and the above observations
\[
(K_Y+
E+
(1-\epsilon a_i m_i)E_i+D_\epsilon)|_E =
K_E+ 
\sum_{i=1}^r(1-\epsilon a_i)T_i+
D_\epsilon|_E
\sim_\qq 0.
\]
By inversion of adjunction, we conclude that 
$(Y,D_Y+E)$ is plt around $E$.
\\

\noindent
{\it Step 7.} 
In this step, we prove that $M$ descends to a neighborhood of $x$ in $X$.\\

\noindent 
We run a $(K_Y+E+D_Y)$-MMP with scaling of an ample divisor relatively over $X$.
We show that this run of the MMP terminates with a good minimal model:
indeed, no step of such MMP contracts curves contained in $E$, as $(K_Y+E+D_Y)\vert_E\sim_\qq 0$.
Hence, each step of this MMP must preserve the pre-image of a neighborhood of $x \in X$.
On the other hand, up to shrinking, over $U:= X\setminus \{x\}$ the pair $(Y, E+D_Y)$ is klt, hence the $(K_Y+E+D_Y)$-MMP must terminate over $U$. 
Hence, after finitely many steps of the above MMP, all disjoint from $E$, we obtain a model $Y'' \to X$ on which $K_{Y''}+D_{Y''}+E''$ is dlt big and nef over $X$, $D_{Y''}$ (resp. $E''$) being the strict transform of $D_Y$ (resp. $E$) on $Y''$.
By construction $D_\epsilon\vert_E$ is ample.
As $E$ is the fiber over $x$, it follows that up to shrinking $X$ around $x$, we can assume that $D_\epsilon$ is ample over $X$.
Hence, for $0<\eta \ll 1$, $\Theta:=D_\epsilon -\eta E$ is ample over $X$ and choosing $0 \leq G\sim_{\qq, X} \Theta$ a general element of the $\qq$-linear system of $\Theta$ over $X$, then $(Y, (1-\eta)E+ \sum_{i=1}^r (1-\epsilon a_im_i) E_i + G)$ is klt and 
\begin{align}
\label{eqn:klt.trick.1}
K_Y+ 
(1-\eta)E+
\sum_{i=1}^r (1-\epsilon a_im_i) E_i + 
G \sim_{\qq, X}
K_Y+E+D_Y.
\end{align}
In particular, the relative log canonical ring of $K_Y+E+D_Y$ over $X$ is finitely generated, see~\cite[Theorem 1.2]{BCHM10}, and $K_{Y''}+D_{Y''}+E''$ is semiample over $X$.
Furthermore, since $(K_{Y''}+D_{Y''}+E'')\vert_{E''} \sim_\qq 0$, then the contraction $Y'' \to Y'''$ to the ample model of $K_{Y''}+D_{Y''}+E''$ must contract $E$;
thus, $Y''$ is isomorphic to $X$ in a neighborhood of $x$.
A fortiori, the same holds on $Y$, up to shrinking around $x\in X$, that is, 
\begin{align}
\label{eqn:klt.trick.2}
K_Y+E+D_Y
\sim_{\qq, X} 
0
\sim_{\qq, X}
K_Y+ 
(1-\eta)E+
\sum_{i=1}^r (1-\epsilon a_im_i) E_i + 
G
\end{align}
since, by construction, $Y \dashrightarrow Y''$ is an isomorphism over a suitable neighborhood of $x$.
\\
For $0< \lambda \ll 1$, 
 we define the generalized log pair $(Y,
(1-\eta)E+
\sum_{i=1}^r (1-\epsilon a_im_i) E_i + G+
\lambda M_Y)$
which is a generalized klt pair whose nef part is that of $(Y, B_Y+M_Y)$.
By~\eqref{eqn:klt.trick.2},
\begin{align}
\label{eqn:klt.trick.3}
K_Y+
(1-\eta)E+
\sum_{i=1}^r (1-\epsilon a_im_i) E_i + 
G+
\lambda M_Y
\sim_{\qq, X}
\lambda M_Y.
\end{align}
By~\cite[Lemma 4.4]{BZ16}, the relative $(K_Y+
(1-\eta)E+
\sum_{i=1}^r (1-\epsilon a_im_i) E_i + 
G+
\lambda M_Y)$-MMP over $X$ with scaling of an ample divisor terminates.
Since, by Step 4, $M_Y\vert_E \sim_\qq 0$, then the same argument as in the previous paragraph works to show that $M_Y \sim_{\qq, X} 0$.
By Step 4 then the nef part of $
(X, B+M)$ descends to a neighborhood of $X$, i.e., $M \sim_{\qq, X}0$ as a b-divisor.
In particular,~\eqref{eqn:klt.trick.2}-\eqref{eqn:klt.trick.3} imply that $(Y, (1-\eta) E+\sum_{i=1}^r (1-\epsilon a_im_i) E_i+G)$ descends to a log pair $(X, \sum_{i=1}^r (1-\epsilon a_im_i) \tilde E_i+ \tilde G)$ which is klt around $x$, where $\tilde E_i$ (resp. $\tilde G$) is the strict transform of $E_i$ (resp. $G$) on $X$.
Furthermore, $Y \to X$ is a relative Mori dream space, by 
\\

\noindent
{\it Step 8}. 
In this step, we prove that there exists a divisor $\Gamma_Y$ on $Y$ such that $(Y,\Gamma_Y)$ is klt and $K_Y+\Gamma_Y \sim_{\qq, X} 0$.
Moreover, $(X, \Gamma)$ satisfies properties (2) and (3) of the statement of the proposition, where $\Gamma$ is the strict transform of $\Gamma_Y$ on $X$.
\\

\noindent 
As $Y$ is a relative Mori dream space over $X$, we can run the relative $(-E)$-MMP on $Y$ which terminates with a good minimal model $Y \dashrightarrow Y'''$.
As the fiber of $Y$ over $x$ is completely supported on $E$ it follows that the map $Y \dashrightarrow Y'''$ is a composition of a  finite sequence of $(-E)$-flips.
Thus, by Lemma~\ref{complexity.flops}, 
$\orbcomp_x(Y'''/X, B_{Y'''}+M_{Y'''})=0$, where $B_{Y'''}$ (resp. $M_{Y'''}$) is the strict transform of $B_Y$ (resp. $M_Y$) on $Y'''$.
As $Y'''$ is a relative good minimal model, then $-E'''|_{E'''}$ is nef and big.
Since the effective cone of divisors of a toric variety is generated by the prime toric divisors, there exists positive rational numbers $b_i$, $i=1, \dots, r$ such that
\begin{align*}
-E|_E 
&\sim_\qq 
\sum_{i=1}^s b_iT_i, \\
-E'''|_{E'''} 
&\sim_\qq 
\sum_{i=1}^s b_iT'''_i,
\end{align*}
where the second line follows from the first one since $Y \dashrightarrow Y'''$ is an isomorphism in codimension one and $T'''_i$ is the strict transform of $T_i$ on $E'''$.
By inversion of adjunction, we conclude that for $0< \epsilon' \ll 1$, $(Y''', (1-\epsilon')E'''+\sum_{i=1}^r (1-\epsilon' b_i m_i)E'''_i)$ is klt around $E'''$ and
$
K_{Y'''}+
(1-\epsilon')E'''+
\sum_{i=1}^r (1-\epsilon' b_i m_i)E'''_i
\vert_{E'''} 
\sim_\qq 
0$.
The same argument as in Step 7 implies that 
\[
K_{Y'''}+(1-\epsilon')E'''+\sum_{i=1}^r (1-\epsilon' b_i m_i)E'''_i \sim_{\qq, X} 0.
\]
Setting 
\[
\Gamma:= \sum_{i=1}^r (1-\epsilon' b_im_i)F_i,
\]
where $F_i$ is the strict transform of $E'''_i$ on $X$, thus, $(X,\Gamma)$ is a klt pair and moreover, by construction, it satisfies property $(2)$ of the statement.
To construct a formally toric plt blow-up of $(X, \Gamma)$, it suffices to consider $Y'''' \to X$ to be a $\qq$-factorialization of the ample model of $-E'''$ over $X$. 
To show that such model satisfies all conditions in the definition of a formally toric plt blow-up, it suffices to notice that $\orbcomp_x(Y''''/X, B_{Y''''}+M_{Y''''})=0$, by Corollary~\ref{complexity.MMP}, and repeat the arguments from the previous steps of the proof.
\end{proof} 

\begin{corollary}
\label{cor:locally-M-descends}
Let $(X,B+M)$ be a generalized log canonical pair and $x\in X$ be a closed point.
If $\orbcomp_x(X,B+M)=0$, 
then
\begin{enumerate}
    \item 
$M$ descends over a neighborhood of $x\in X$;
     \item 
there exists a klt pair $(X,\Gamma)$ and $\Gamma < \lceil B \rceil$, $\lfloor B\rfloor\leq \lceil \Gamma\rceil$;
and,
    \item 
the germ $(x\in X,\Gamma)$ admits a formally toric plt blow-up.
\end{enumerate}
\end{corollary}

\begin{proof}
It suffices to observe that if $n$ is an orbifold structure on $X$ supported at $x$ and 
\[
\Sigma= \sum_{p \in X^1} \left(
1-\frac{1}{n_P}\right) P
+ \sum_{i=1}^r b_iB_i
\]
is an orbifold decomposition of $B$ for $n$ at $x$ such that $\orbcomp_x(X, B+M; \Sigma)=0$, then there exists an open set $x \in U \subset X$ such that the dimension of the span of the $B_i$ in ${\rm Cl}_\qq(U)$ is the same as that of the span of the $B_i$ in $\clxx$, by the very definition of $\clxx$, cf. Definition~\ref{def:cl.gr}.
Hence, $\orbcomp_x(X, B+M; \Sigma)=0=\orbcomp_x(U/U, B\vert_U+M\vert_U; \Sigma)$ where for $\orbcomp_x(U/U, B\vert_U+M\vert_U; \Sigma)$ we consider the identity map of $U$ as the structure morphism.
Thus, the conclusion follows from Proposition~\ref{prop:toroida-blow-up}.
\end{proof}

\section{Toric formality}
\label{section:proof-local}

In this section, we introduce the concept of toric formality. 
We show that a formally toric plt blow-up, as defined in \S~\ref{sect:plt.blow.ups}, is a special type of birational extraction that is formally isomorphic to a special blow-up of a toric singularity.
We will prove that the existence of a formally toric plt blow-up implies that the germ is formally toric, cf.~Proposition~\ref{prop:toric-plt-implies-toric}. 
These results will constitute the main ingredient in the proof of the local version of Theorem~\ref{formally toric}, cf.~Theorem~\ref{thm:local-case}.

\subsection{The relative Cox ring}
\label{subsect:rel.cox}
In order to prove Proposition~\ref{prop:toric-plt-implies-toric}, we will introduce and prove a few preliminary results related to Cox rings.
We will study the relative Cox ring of a formally toric blow-up $\pi\colon Y\rightarrow X$.
In particular, we will focus on describing the ideal generated by sections whose associated divisor on $Y$ intersects $E$ non-trivially.
We start by providing some basic definitions and introducing some notation.

\begin{notation}
\label{notation:class.gr}
{\em 
Let $Y\rightarrow X$ be a projective birational morphism between normal quasi-projective varieties, $X={\rm Spec}(R)$, and let $x\in X$ be a distinguished (closed) point.
Let $Y_x$ be the base change of $Y$ induced by the natural morphism $X_x\rightarrow X$ and let ${\rm Cl}(Y_x)$ be its class group, that is, ${\rm Cl}(Y_x):={\rm WDiv}(Y_x)/{\rm Prin}(Y_x)$
}
\end{notation}

As $X_x$ is birational to $Y_x$, ${\rm Cl}(Y_x)={\rm WDiv}(Y_x)/\principal(X_x)$.
Moreover, as $X_x$ is essentially of finite type, ${\rm Cl}(Y_x/X_x)\simeq {\rm Cl}(Y_x)$.

\begin{assumption}
\label{assumption.cox.ring}
{\em
We adopt Notation~\ref{notation:class.gr}.
We assume that 
\begin{enumerate}
    \item 
    ${\rm Cl}(Y_x)$ is finitely generated, and
    \item
    \label{assumption:tors.free}
    ${\rm Cl}(Y_x)$ is torsion free.
\end{enumerate}
}
\end{assumption}

\begin{definition}
\label{def:cox.ring}
{\em 
We adopt Notation~\ref{notation:class.gr} and Assumption~\ref{assumption.cox.ring}.
\\
Let $K\subset {\rm WDiv}(Y)$ be a finitely generated subgroup whose elements map isomorphically onto ${\rm Cl}(Y_x)$.
The {\rm Cox ring} of $Y$ relatively to $X$ around $x$ is defined as
\begin{align*}
\coxyx
:=
\bigoplus_{D \in K} 
H^0(Y,\mathcal{O}_Y(D)) 
\end{align*}
The multiplication in $\coxyx$ is defined in the function field of $Y$, as per usual.
}
\end{definition}

The ring $\coxyx$ is naturally graded by the free group $\clyx$, by identifying $K$ with its image in $\clyx$.

\begin{remark}
\label{rem:well-definedness}
{\em 
We adopt Notation~\ref{notation:class.gr} and Assumption~\ref{assumption.cox.ring}.
\\
The Cox ring $\coxyx$ parametrizes all effective Weil divisors on $Y$ locally above $x$.
However, since $H^0(X,\mathcal{O}_X)^\ast =R^\ast $, and $R^\ast$ may be strictly larger than $\mathbb K^\ast$, then $\coxyx$ may have non-trivial units.
In particular, for a Weil divisor $W$ on $Y$, there exists a ${\rm Cl}(Y)$-invariant element $w\in \coxyx$ and such choice of $w$ is determined up to multplication by an element in $R^\ast $.
}
\end{remark}

\subsection{Toric formality and plt blow-ups}
\label{subsect:torsion.free}
In this subsection, we shall prove the following proposition relating the existence of a formally toric plt blow-up to the formal toricness of a fixed germ of klt singularity.

\begin{proposition}
\label{prop:toric-plt-implies-toric}
Let $(x \in X,\Gamma)$ be a germ of a klt pair. 
Assume that $(x \in X,\Gamma)$ admits a formally toric plt blow-up $ \pi \colon Y \to X$ at $x \in X$ and let $E$ be the exceptional divisor on $\pi$.
Then, $\pi$ is formally toric over $x$ with respect to the pair $(Y, \pi^{-1}_\ast \lceil \Gamma \rceil +E)$ and $(x \in X,\lceil \Gamma \rceil)$ is a formally toric singularity.
\end{proposition}

\subsubsection{Torsion free case}
We first proceed to prove Proposition~\ref{prop:toric-plt-implies-toric} under the assumption that ${\rm Cl}(X_x)$ is free. The main tool is the relative Cox ring introduced in the previous subsection.

We will be working in the framework
of the following construction arising from a formally toric plt blow-up of a germ of a klt singularity, cf.~Definition~\ref{def:form.tor.plt.blowup}.

\begin{construction}
\label{const:plt-toric}
{\em 
We adopt Notation~\ref{notation:class.gr}.
\\
Let $(x \in X,\Gamma)$ be the germ of a klt singularity.
We assume that $X = {\rm Spec}(R)$.
We denote by $\mathfrak{m}_x \subset R$ the maximal ideal associated to the closed point $x$.
Let $\pi\colon Y\rightarrow X$ be a formally toric plt blow-up for $(X, \Gamma)$ with exceptional divisor $E$.
\\
By~\cite[Theorem 3.27]{BM21}, ${\rm Cl}(Y_x)$ is finitely generated;
thus, we can adopt Assumption~\ref{assumption.cox.ring}.1.
As already mentioned above, we will also adopt~\ref{assumption.cox.ring}.\ref{assumption:tors.free}.
Thus we can defined $\coxyx$ as in Definition~\ref{def:cox.ring}.
\\
We proceed to the following construction:
\begin{enumerate}
   
    \item 
We fix a homogeneous element $e \in \coxyx$ such that ${\rm div}(e)=E$.
By Remark~\ref{rem:well-definedness}, $e$ is defined up to multiplication by a unit of $R$.

	\item 
We denote by $T_1,\dots, T_r \subset E$ the torus invariant divisors of $E$ supported at the intersection between components $E_i$ of $\pi_\ast^{-1}\lceil \Gamma \rceil$ and $E$.
We fix homogeneous elements $x_1,\dots, x_r \in \coxyx$ such that ${\rm div}(x_i)=E_i$.
By Remark~\ref{rem:well-definedness}, all these elements are uniquely defined up to multiplication by a unit of $R$.
Moreover, 
\[
E_i\vert_E= \frac{T_i}{m_i},
\]
where $m_i$ is the Cartier index of $E$ at the generic point of $T_i$, cf.~Remark~\ref{rem:gamma_i}.
We also set $F_i:=\pi_\ast E_i$.

	\item 
We denote by $\mathfrak{m}_{\rm Cox}$ the ideal of $\coxyx$ generated by $\mathfrak m_x$ and by sections of $H^0(Y,\mathcal{O}_Y(D))$ whose divisor of zeroes on $Y$ intersects $E$ non-trivially, for $D \in K$ the finitely generated subgroup of ${\rm WDiv}(X)$ defined in Definition~\ref{def:cox.ring}.

	\item 
If we pass to an affine open subset $X' \subset X$ containing $x$, denoting $Y':=\pi^{-1}(X') \subseteq Y$,
then we can define, using the same procedure of Definition~\ref{def:cox.ring}.
${\rm Cox}(Y'/X')$ is a $\clyx$-graded ring.
To define ${\rm Cox}(Y'/X')$, it suffices to restrict to $Y'$ the divisor in the finitely generated group $K$.

	\item 
In ${\rm Cox}(Y'/X')$, we define the ideal $\mathfrak{m}_{\rm Cox,X'}$, as the ideal generated by $\mathfrak m_{x, X'}$ and by sections of $H^0(Y',\mathcal{O}_{Y'}(D))$, $D \in K$ whose divisor of zeroes on $Y'$ intersect $E$ non-trivially.
In particular, $\mathfrak{m}_{{\rm Cox},X'} \subset {\rm Cox}(Y'/X')$ is the ideal generated by the image of $\mathfrak{m}_{{\rm Cox},X}$ under the natural restriction map (of sections/cycle) from $Y$ to $Y'$.
\end{enumerate}
}
\end{construction} 

The following lemma relates the freeness of ${\rm Cl}(X_x)$ and ${\rm Cl}(Y_x)$ in the context of Construction~\ref{const:plt-toric}.

\begin{lemma}
\label{lem:no-torsion-cl}
With the notation and assumptions of Construction~\ref{const:plt-toric}, if ${\rm Cl}(X_x)$ is free, then ${\rm Cl}(Y_x)$ is also free.
\end{lemma}

\begin{proof}
Let $W$ be a Weil divisor on $Y$ such that $mW\sim 0$ for some $m \in \mathbb{N}_{>0}$.
Setting $W_X:=\pi_\ast W$, $W_X$ is a Weil divisor on $X$ and $mW_X\sim 0$.
Since ${\rm Cl}(X_x)$ has no torsion, then $W_X\sim 0$ on a neighborhood $X'$ of $x$.
Then, $\pi^\ast (W_X)=W+\alpha E \sim 0$ over the counter-image $Y'$ of $X'$ on $Y$.
Since $W\sim_\mathbb{Q}0$, then also $aE\sim_\mathbb{Q}0$.
However, since $E\vert_E \not \sim_\mathbb{Q}0$ as $-E \vert_E$ is big, then $a=0$ and $W\sim 0$ on $Y'$.
\end{proof}

In view of Lemma~\ref{lem:no-torsion-cl}, we can substitute Assumption~\ref{assumption.cox.ring}.\ref{assumption:tors.free} with the assumption that ${\rm Cl}(X_x)$ is free.
We will do so for the remainder of this subsection.

We now proceed to prove a few technical results that will be used in the proof of Proposition~\ref{prop:no-torsion-toric}.

\begin{lemma}
\label{mcox-maximal}
With the notation and assumptions of Construction~\ref{const:plt-toric},
the ideal
$\mathfrak{m}_{{\rm Cox}} \subset \coxyx$ is a maximal ideal.
\end{lemma}

\begin{proof}
Let $D \neq 0$ be a Weil divisor of $K$.
Then $0 \neq [D] \in {\rm Cl}(Y_x)$, since, by assumption, $K$ maps isomorphically onto ${\rm Cl}(Y_x)$.
For any section $0 \neq \sigma \in H^0(Y, \mathcal O_{Y}(D))$, ${\rm div}(\sigma) \cap E \neq \emptyset$: 
indeed, if that were not the case, then considering the open set $X \setminus{ \pi(\supp ({\rm div}(\sigma)))}$, it immediately follows that $[D] = 0 \in {\rm Cl}(Y_x)$, leading to a contradiction.
Hence, 
\begin{align*}
\bigoplus_{0 \neq D \in K} H^0(Y, \mathcal O_Y(D)) \subset \mathfrak{m}_{\rm Cox}
\quad 
\text{and}
\quad 
\coxyx/\mathfrak{m}_{\rm Cox}\simeq 
H^0(X, \mathcal{O}_{X})/\mathfrak{m}_{x} \simeq \kk.
\end{align*}
\end{proof}

\begin{lemma}
\label{lem:linear-equivalence}
With the notation and assumptions of Construction~\ref{const:plt-toric},
let $D$ and $D'$ be two Weil divisors on $Y$ whose support do not contain $E$.
If $D\vert_E \sim_\qq D'\vert_E$, then $[D]= [D'] \in \clyx$.
\end{lemma}

\begin{proof}
By Lemma~\ref{lem:linear-equivalence.gen}, $D-D'\sim_\qq 0$ over a sufficiently small neighborhood of $x \in X$. 
The conclusion then follows from Lemma~\ref{lem:no-torsion-cl}.
\end{proof}

\begin{lemma}\label{E-no-torsion}
With the notation and assumptions of Construction~\ref{const:plt-toric},
the group ${\rm Cl}(E)$ is torsion free.
\end{lemma}

\begin{proof}
Let $W_E$ be a torsion Weil divisor on $E$.
Then $W_E \sim \sum_{i=1}^r a_iT_i$, $a_i \in \mathbb Z$.
Setting $W_Y:= \sum_{i=1}^r a_i m_i E_i$, then $W_Y|_E=W_E$.
By Lemma~\ref{lem:linear-equivalence}, then $[W_Y]=0 \in \clyx$.
Thus, $W_E\sim 0$ on $E$.
\end{proof}

\begin{lemma}
\label{-E-is-a-gen}
With the notation and assumptions of Construction~\ref{const:plt-toric},
let $R_E$ be the ray generated by $-E$ in ${\rm Cl}(Y_x) \otimes \mathbb Q$.
Then $R_E\cap {\rm Cl}(Y_x)$ is generated by $-E$.
\end{lemma}

\begin{proof}
Let $W$ be a Weil divisor on $Y$ such that $mW\sim -E$.
Write $W=W'+\alpha E$ and $E  \not \subset \supp W'$.
Then, $W'':=mW'+(m\alpha+1)E\sim 0$ and $\pi_\ast W''\sim 0$.
Setting $W'_X:=\pi_\ast W'$, then $mW'_X\sim 0$; 
thus, $W'_X\sim 0$ by the freeness of ${\rm Cl}(X_x)$.
Writing $\pi^\ast (W'_X)=W'+\beta E$, $\beta \in \mathbb Z$, then, $mW'\sim -m\beta E$ and 
\[
(m(\alpha-\beta)+1)E \sim 0.
\]
Since the intersection of a general curve on $E$ with $E$ is negative, we conclude that $m=1$ as claimed.
\end{proof}

\begin{lemma}
\label{lem:q-linear}
With the notation and assumptions of Construction~\ref{const:plt-toric},
the following equalities holds:
\begin{align}
\label{eqn:equal.cl.gr}
\rank ({\rm Cl}(Y_x))
=
\rank ({\rm Cl}(X_x))+1
=
\rank({\rm Cl}(E)).
\end{align}
In particular, we can choose $K = \zz E \oplus K' \subset {\rm WDiv}(X)$ in the construction of $\coxyx$ in such a way that the support of any non-zero element of $K'$ is contained in $E_1 \cup \dots \cup E_r$.
\end{lemma}

\begin{proof}
The first equality in~\eqref{eqn:equal.cl.gr} is simply a consequence of~\cite[Proposition 6.5]{Har77}.
\\
To prove the second equality, let $W_1,\dots,W_k$ be Weil divisors supported on the torus invariant divisors of $E$ and such that the classes $[W_1],\dots, [W_k],[-E|_E]$ give a basis
of ${\rm Cl}_\qq(E)$.
For all $i$, we can write $W_i = \sum_{j=1}^r \alpha_{i, j} T_j$, $i=1, \dots, k$, $\alpha_{i, j} \in \qq$.
Setting $W_{Y,i}:= \sum_{j=1}^r \alpha_{i, j} m_j E_j$, $i=1, \dots, k$, then $W_{Y, i} \vert_E = W_i$, cf. Remark~\ref{rem:gamma_i} and we claim that $[W_{Y,1}],\dots,[W_{Y,k}],[-E]$ is a basis of ${\rm Cl}_\qq (Y_x)$.
First, we prove that the $W_{Y, i}$ generate ${\rm Cl}_\qq (Y_x)$.
Indeed, let $D$ be a Weil divisor on $Y$ whose support does not contain $E$.
There exist $\alpha_1,\dots,\alpha_k \in \qq$ such that $\sum_{i=1}^k \alpha_i W_i \sim_\qq D|_E$.
Lemma~\ref{lem:linear-equivalence} implies that $\sum_{i=1}^k \alpha_i [W_{Y,i}]=[D]$ in ${\rm Cl}_\qq (Y_x)$.
On the other hand, the classes $\{ [W_{Y, 1}], \dots, [W_{Y, k}], [E]\}$ are linearly independent in ${\rm Cl}_\qq (Y_x)$ since they restrict to linearly independent classes in ${\rm Cl}_\qq (E)$.
\\
To show the second part of the statement, it suffices to notice that the classes $[E_i \vert_E]$ span the image of the restriction map
\[
{\rm Cl}(Y_x) \to {\rm Cl}_\qq(E), 
\quad 
[D] \mapsto [D\vert_E].
\]
\end{proof}

\begin{lemma}\label{dimension-of-cox-ring}
With the notation and assumptions of Construction~\ref{const:plt-toric},
the dimension of the Cox ring $\coxyx$ is at most $\dim Y + \rho(E)$.
\end{lemma}

By Construction~\ref{const:plt-toric}, $\dim Y + \rho(E)=r+1$.

\begin{proof}
The variety $X$ is a GIT quotient of the spectrum of $\coxyx$ by the Picard torus ${\rm Spec}(\mathbb{K}[{\rm Cl}(Y_x)])$, see~\cite[Remark~1.2.3.2]{ADHL15}.
Hence, 
$\dim\coxyx
\leq 
\dim X_x + \rank {\rm Cl}(E)$.
By Lemma~\ref{lem:q-linear}, the Picard torus has dimension $\rank{\rm Cl}(E)=\rho(E)$.
Hence, $\dim \coxyx\leq \dim Y+\rho(E)\leq \dim E+\rho(E)+1=r+1$.
\end{proof}

Now, we turn to prove that the spectrum of the $\clyx$-graded ring $\coxyx$ is regular at the maximal ideal $\mathfrak{m}_{\rm Cox}$.
As we are working under the assumption that ${\rm Cl}(X_x)$ has no torsion, we shall use Lemma~\ref{lem:linear-equivalence} to measure linear equivalence on $Y$ by restricting to $E$.

\begin{lemma}
\label{lem:regular-Cox}
With the notation and assumptions of Construction~\ref{const:plt-toric},
the spectrum of the ring $\coxyx$ is regular at the maximal ideal $\mathfrak{m}_{\rm Cox}$.
\end{lemma}

\begin{proof}
By Lemma~\ref{dimension-of-cox-ring},
$\dim \coxyx\leq r+1$.
Hence, it suffices to prove that $\mathfrak{m}_{\rm Cox}$ is generated by $r+1$ elements.
We shall show that the elements $x_1,\dots,x_r, e \in \coxyx$ provide a set of generators of $\mcox$.
\\
Let $0 \neq w \in \mcox$ be a homogeneous element $w \in H^0(Y, \mathcal O_Y(D))$, for some $D \in K$.
Let $W:= \ddivv(w)$ on $Y$.
We may assume that $W$ doesn't contain $E$ in its support, as otherwise $w\in \langle e\rangle$.
In particular, since
\[
H^0(Y, \mathcal O_Y(-E)) \cdot e = \mathfrak{m}_x \subset H^0(Y, \mathcal O_Y) = H^0(X, \mathcal O_X),
\]
we can assume that the degree of $w \in \coxyx$ is not zero.
\\

\noindent {\it Claim.}  
If there exists $w'\in \langle x_1, \dots, x_r\rangle$ such that $W' := \ddivv(w')$ satisfies $W'|_E=W|_E$, then $w \in \langle x_1, \dots, x_r, e\rangle$.

\begin{proof}[Proof of Claim]
By Lemma~\ref{lem:linear-equivalence}, $[W']= [W] \in \clyx$.
Hence, $w$ and $w'$ are homogeneous element of the same degree in $\coxyx$.
Let $m$ be divisible enough so that $mW$ and $mW'$ are Cartier divisors on $Y$.
Then, $mW'|_E =mW|_E$ and $(w')^m\vert_E = \lambda w^m\vert_E$, for some $\lambda \in \mathbb K^\ast$, as $E$ is projective. 
Hence, by the restriction exact sequence, we conclude that $w^m-\lambda(w')^m \in \langle  e\rangle$.
By~\cite[Proposition 1.3.3.4]{ADHL15}, the ideal $\langle e\rangle$ is prime since $e$ is irreducible in $\coxyx$ by Lemma~\ref{-E-is-a-gen}.
Hence, $w-\mu_m w' \in \langle  e\rangle$ for an appropriate $m$-th root $\mu_m$ of $\lambda$.
As we assumed that $w' \in \langle x_1,\dots, x_r\rangle$, then $w\in \langle x_1,\dots,x_r,e\rangle$ as claimed.
\end{proof}
\noindent
We are now going to construct an element $w' \in \coxyx$ satisfying the conditions of the Claim.
To this end, we may assume that $W$ is a reduced effective Weil divisor.
Let 
$T_i \in E^1$ be a torus invariant divisor.
Let us recall that ${\rm coeff}_{T_i}W|_E=\frac{q}{m_i}$, $q \in \mathbb N$ and $q = {\rm coeff}_P(W|_E)$.
On the other hand, if $P$ is a prime component of $W|_E$ which is not torus invariant, then $W$ is a Cartier divisor at the generic point of $P$ and ${\rm coeff}_P(W|_E)P$ is an integral multiple of $P$.
As ${\rm Cox}(E)$ is a free polynomial ring generated by the sections $t_j$, corresponding to the torus invariants divisors $T_j$, $j=1, \dots, r$,~\cite[Theorem 2.1.3.2]{ADHL15}, then $P$ is just the vanishing locus of $p(t_1,\dots,t_r)$ for some polynomial $p \in \mathbb K[t_1, \dots, t_r]$ which is homogeneous with respect to the natural ${\rm Cl}(E)$-grading on $K[t_1, \dots, t_r]$.
All monomials of $p(t_1, \dots, t_r)$ correspond to effective Weil divisors supported on the toric invariant divisors and 
they belong to the same linear equivalence class as $P$.
Considering the polynomial $p(x_1^{m_1},\dots,x_r^{m_r}) \in \clyx$, by Lemma~\ref{lem:linear-equivalence}, all monomials of $p(x_1^{m_1},\dots, x_m^{m_r})$ in the $x_i$ correspond to effective Weil divisors in the same linear equivalence class on $Y$ which are supported on $E_1\cup\dots\cup E_r$.
Thus, $p(x_1^{m_1},\dots,x_r^{m_r})\in \langle x_1,\dots,x_r\rangle$ is ${\rm Cl}(Y_x)$-invariant.
We conclude that the associated Weil divisor $W'$ on $Y$ satisfies $W'|_E=P$.
This concludes the construction of $w'$.
\end{proof}

\begin{proof}[Alternative proof of Lemma~\ref{lem:regular-Cox}]
We consider the elements  
$x_1,\dots,x_r,e \in \coxyx$, cf. Construction~\ref{const:plt-toric}.
We define $\scoxyx:={\rm Spec}\coxyx$;
for an element $s \in \coxyx$, ${\rm div}(s)$ will denote the corresponding principal divisor on $\scoxyx$.
Let $K_{\scoxyx}+ {\rm div}(x_1)+\dots +{\rm div}(x_r)+{\rm div}(e)$ be the log pull-back of 
$K_Y+E+E_1+\dots+E_r$ to $\scoxyx$,
via the natural rational map $\scoxyx\dashrightarrow Y$, cf.~\cite[Construction 1.6.3.1]{ADHL15}.
The divisor $K_{\scoxyx}$ is Cartier by the freeness of the Class group and it is actually the pull-back of $K_X$, cf.~\cite[Theorem 1.1]{Arz09}.
\\
We claim that $(\scoxyx,{\rm div}(x_1)+\dots +{\rm div}(x_r)+{\rm div}(e))$ is a log canonical pair.
This argument is similar to the one in~\cite{KO15}.
Let $W\rightarrow \scoxyx$ be a torus equivariant log resolution of the above log pair.
Assume that $W$ admits a good torus quotient $W\rightarrow W_Y$ so that $W_Y\rightarrow Y$ is birational.
By the proof of~\cite[Theorem 4.7]{LS13}, there exists a $\qq$-trivial log sub-pair $(W_Y,B_{W_Y})$ so that the total discrepancy of $(\coxyx,{\rm div}(x_1)+\dots+{\rm div}(x_r)+{\rm div}(e))$ is larger than or equal to the total discrepancy of $(W_Y,B_{W_Y})$.
Furthermore, by equation (7) in the proof of~\cite[Theorem 4.7]{LS13}, we know that $B_{W_Y}$ has coefficient one at the strict transform of each $E_1,\dots,E_r$ and $E$.
By the negativity lemma, we conclude that
$K_{W_Y}+B_{W_Y}$ is the log pull-back of
$K_Y+E_1+\dots+E_r+E$ to $W_Y$.
In particular, $(W_Y,B_{W_Y})$ is a log canonical sub-pair.
Therefore, the log pull-back of 
$(\scoxyx,{\rm div}(x_1)+\dots+{\rm div}(x_r)+{\rm div}(e))$ to $W$ is a log canonical sub-pair.
In particular, the above pair is log canonical.
Now, each principal divisor ${\rm div}(x_i)$ and ${\rm div}(e)$ passes through the closed point defined by $\mathfrak{m}_{\rm Cox}$.
We have exactly $r+1$ Cartier divisors passing through that closed point.
Recall by Lemma~\ref{dimension-of-cox-ring}, that $\scoxyx$ has dimension $r+1$. 
By~\cite[Lemma 2.4.3]{BMSZ18}, $(\scoxyx,{\rm div}(x_1)+\dots +{\rm div}(x_r)+{\rm div}(e))$ is log smooth at the closed point corresponding to $\mathfrak{m}_{\rm Cox} \subset \coxyx$.
Hence, the elements $x_1,\dots,x_r,e$ generate $\mathfrak{m}_{\rm Cox}$.
\end{proof}

Our first step toward a proof of Proposition~\ref{prop:toric-plt-implies-toric} is to show that the formally toric singularity $x \in X$ is formally isomorphic to the singularity at the vertex of the cone over the exceptional divisor $E$ of $\pi$ with respect to the semiample $\qq$-divisor $-E|_E$.
To this end, we denote by ${\rm Cone}(E,-E|_E)$ the orbifold cone corresponding to the semiample and big $\qq$-divisor $-E|_E$ and by $\mathfrak{m}_v$ the maximal ideal of the vertex of the cone which we denote by $v$.
Let $R(E, -E\vert_E)$ be the $\mathbb N$-graded ring
\[
\ringcone:=\bigoplus_{m\in \nn} H^0(E,m(-E|_E))t^m \subset \mathbb K(E)[t],
\]
where $t$ is simply an auxiliary indeterminate.
We set $\mathfrak{m}_{E} \subset \ringcone$ to be the maximal ideal
\[
\mathfrak{m}_{E} := \bigoplus_{m\in \nn_{>0}} H^0(E,m(-E|_E))t^m.
\]
By definition, 
\[
\conee := {\rm Spec}(\ringcone)
\]
and $\mathfrak m_v:=\mathfrak{m}_{E}$.
Alternatively, taking $c \in \nn_{>0}$ the least natural number such that $-cE \vert_E$ is Cartier and
denoting by $R^{(c)}(E, -E\vert_E)$ the ring
\[
\ringconec:=\bigoplus_{m\in \nn} H^0(E,m(-cE|_E))t^{cm}\subset \mathbb K(E)[t^c],
\]
then $\conee$ can be defined as the normalization of ${\rm Spec}(\ringconec)$ inside $\mathbb K(E)(t)$.

\begin{proposition}
\label{prop:no-torsion-toric}

With the notation and assumptions of Construction~\ref{const:plt-toric},
\[
\widehat{X_x} \simeq \conev
\]
and under this isomorphism the completions of the $F_i$ correspond to the completions of the toric invariant divisors of $\conee$.
In particular, $(x \in X, \sum_{i=1}^r F_i)$ is a formally toric singularity.
\end{proposition}

\begin{proof}
By Lemma~\ref{lem:regular-Cox}, the spectrum of ${\rm Cl}(Y/X)$ is regular at the maximal ${\rm Cl}(Y_x)$-invariant ideal $\mathfrak{m}_{\rm Cox}$.
Let $\kk[ \xi_1,\dots,\xi_r,\eta ]$ be a free graded $\mathbb K$-algebra for which we define the following ${\rm Cl}(Y_x)$-grading:
\begin{align}
\label{eqn:grad.defn}
\deg \xi_i= [E_i], \quad \deg \eta= [E].
\end{align}
Then, we can define the following graded morphism of $\mathbb K$-algebras
\begin{align*}
\phi \colon \kk[ \xi_1,\dots,\xi_r,\eta ] &\to \coxyx\\
\xi_i \mapsto x_i\quad & \eta \mapsto e
\end{align*}
which, by definition, is equivariant for the $\clyx$-grading on the two rings.
Denoting by $\compcox$ the completion of $\coxyx$ with the respect to the $\mcox$-adic topology, by Lemma~\ref{lem:regular-Cox}, $\phi$ descends to an isomorphism of the completions
\[
\hat \phi \colon 
\kk[\![ \xi_1,\dots,\xi_r, \eta ]\!] \to \compcox.
\] 
Any monomial in the $\xi_1, \dots, \xi_r, \eta$ has a $\clyx$-grading that descends from the $\clyx$-grading in~\eqref{eqn:grad.defn}.
Such grading descends also to the images of the monomials via $\hat \phi$, as $\phi$ is $\clyx$-equivariant by construction.
By abusing notation, we will indicate the images of $x_i$ and $e$ via the natural map $\coxyx \to \widehat{\coxyx_{\mcox}}$ with the same notation.
Hence, any monomial in the $x_i, e$, $i=1, \dots, r$ in $\compcox$ inherit the same $\clyx$-degree as their original one in $\coxyx$.
\\

\noindent
{\it Claim.} 
Let $S$ be the ring of power series on degree $0$ monomials on the variables $\xi_1,\dots,\xi_r,\eta$.
Then, 
\[
S=\hat \phi(\compR),
\]
where $\compR$ is the completion of $R_{\mathfrak m_x}$ with respect to the $\mathfrak m_x$-adic topology.

\begin{proof}[Proof of Claim]
The following diagram is commutative and all the arrows correspond to injective morphisms
\[
\xymatrix{
R\ar[r]\ar[d] & \coxyx\ar[d]\\
\compR \ar[r] & \widehat{\coxyx}_{\mathfrak{m}_{\rm Cox}}
} 
\]
The inclusion $\compR \subset \hat{\phi}^{-1}(S)$ follows from the fact that $R_{\mathfrak{m}_x}$ consists of fractions of homogenous elements of $\coxyx$ of degree $0$.
On the other hand, any element of $\hat{\phi}^{-1}(S)$ can be written as a formal sum of degree $0$ monomials in the $x_i, e$, with respect to the ${\rm Cl}(Y_x)$-grading.
As all such monomials are contained in $\compR$, by degree considerations, then the inclusion $\compR \subset \hat{\phi}^{-1}(S)$ follows from~\cite[3.9.8]{EGA1} since the $\mcox$-adic topology on $\coxyx$ induces the $\mathfrak{m}_x$-adic topology on $R_{\mathfrak{m}_x}$.
\end{proof}
\noindent
Let $V \subset \mathbb N^{r+1}_{\geq 0}$ be the submonoid defined by
\[
V:= \left \{(a_1, \dots, a_r, b) \in  \mathbb N^{r+1}
\ \Big \vert \
\left[ \sum_{i=1}^r a_i E_i + b E \right] = 0 \in \clyx \right \}.
\]
Then the above claim shows that there exists an isomorphism
\begin{align}
\label{eqn:isom.compl.1}
\compR \simeq 
\left \{ 
\sum_{(a_1, \dots, a_r, b) \in V}^\infty c_{(a_1, \dots, a_r. b)} 
\xi_1^{a_1} \dots \xi_r^{a_r} \eta^b, 
\ \ 
c_{(a_1, \dots, a_r, b)} \in \mathbb K 
\right \},
\end{align}
Given a degree $0$ monomial 
$\xi_1^{a_1}\dots \xi_r^{a_r} \eta^b$ the exponent $b$ of $\eta$ is automatically defined by the $\clyx$-grading and the exponents $a_i$ of the $\xi_i$, since 
\[
\left[\sum_{i=1}^r a_i E_i \right] = [-bE] \in \clyx.
\]
Defining a morphism $\tau$ of monoids by
\begin{align*}
\tau  \colon 
\mathbb Z_{\geq 0}^r &\to \clyx\\
 (a_1, \dots, a_r) &\mapsto \left[\sum_{i=1}^r a_iE_i \right],
\end{align*}
and defining $V':=\tau^{-1}(\mathbb{N}[-E])$, then monomials of degree 0 in the $\xi_i, \eta$ are all of the following form
\begin{align}
\label{eqn:monomials.def}
\xi_1^{a_1}\dots \xi_r^{a_r} \eta^{\tilde{\tau}(a_1, \dots, a_r)}  \text{ for} \ (a_1, \dots, a_r) \in V', 
\end{align}
where $\tilde{\tau}(a_1, \dots, a_r)$ is the unique non-negative integer such that $[\sum a_i E_i + \tilde{\tau}(a_1, \dots, a_r) E ]= 0 \in \clyx$.
Thus, $\tilde{\tau}$ induces a natural $\mathbb N$-grading $\tilde{\tau} \colon V' \to \mathbb N$ on the monomials in~\eqref{eqn:monomials.def}
defined by simply defining the degree of $\xi_1^{a_1}\dots \xi_r^{a_r}$, for $(a_1, \dots, a_r) \in V'$, to be $\tilde{\tau}(a_1, \dots, a_r)$.
This observation and the Claim imply that there exists an isomorphism
\begin{align}
\label{eqn:isom.compl}
\compR \simeq 
S':=
\left \{ 
\sum^\infty_{(a_1, \dots, a_r) \in V'} c_{(a_1, \dots, a_r)} y_1^{a_1} \dots y_r^{a_r}, 
\ c_{(a_1, \dots, a_r)} \in \mathbb K 
\right \},
\end{align}
where the multiplication for the ring on the right is simply determined by the multiplication of the monomials in the $y_i$.
In particular, the ring $S'$ in~\eqref{eqn:isom.compl} is isomorphic to $S$, by~\eqref{eqn:isom.compl.1}; 
thus, $S'$ is normal, being isomorphic to the completion of a normal excellent local ring.
Moreover, as $V'$ is the set of integral points of a saturated rational polyhedral cone by construction, cf. Lemma~\ref{-E-is-a-gen}, then $S'$ is the completion of the monomial ring $\mathbb K[V']$ whose spectrum is a normal toric variety.
\\
To conclude the proof, we just need to show that 
\begin{align}
\label{eqn:isom.spec.compls}
\widehat{X_x} \simeq \conev.
\end{align}
As $-E|_E$ has fractional coefficients only along the torus invariant divisors of $E$, ${\rm Cone}(E,-E|_E)$  is endowed with a natural structure  of an affine toric variety, cf. also the paragraph before the statement of the proposition.
By Lemma~\ref{E-no-torsion}, ${\rm Cl}(E)$ is torsion-free;
let $c$ be the minimal positive natural number such that $cE$ is Cartier along $E$.
By relative Kawamata-Viehweg  vanishing applied to the short exact sequence,
\[
\xymatrix{
0 \ar[r] &
\mathcal O_Y(-(ct+1)E) \ar[r] &
\mathcal O_Y(-ctE) \ar[r] &
\mathcal O_E(-ctE\vert_E) \ar[r] &
0, &
t \in \mathbb Z_{\geq 0},
}
\]
it follows immediately that 
\[
H^0(Y, \mathcal O_E(-ctE\vert_E)) \simeq \pi_\ast \mathcal O_Y(-ctE) / \pi_\ast \mathcal O_Y(-(ct+1)E).
\]
By construction, a basis of $H^0(Y, \mathcal O_E(-ctE\vert_E))$ is given by restricting to $E$ monomials of the form	
\[
x_1^{a_1} \dots x_r^{a_r} 
\quad \text{ for } 
(a_1, \dots, a_r) \in V' 
\text{ and } 
\tilde{\tau}(a_1, \dots, a_r)=ct.
\]
\noindent
Hence, defining $V'_c$ to be the monoid $V'_c:= V' \cap \tilde{\tau}^{-1}(\mathbb N c)$, then $\ringconec$ is isomorphic as a graded algebra to the $\mathbb K$-algebra $\mathbb K[V'_c]$ generated by monomials $x_1^{a_1} \dots x_r^{a_r}$ with $(a_1, \dots, a_r) \in V'_c$.
As, the normalization of $\mathbb K[V'_c]$ is $\mathbb K[V']$
then $\ringcone \simeq \mathbb K[V']$ and passing to the completions we obtain that
\begin{align}
\label{eqn:isom.compl.2}
\widehat{\ringcone_{\mathfrak{m}_{E}}} \simeq 
\left \{ \sum^\infty_{(a_1, \dots, a_r) \in V'} c_{(a_1, \dots, a_r)} y_1^{a_1} \dots y_r^{a_r}, \ c_{(a_1, \dots, a_r)} \in \mathbb K \right \}
\end{align}
and the isomorphism in~\eqref{eqn:isom.spec.compls} follows from~\eqref{eqn:isom.compl}.
As in~\eqref{eqn:isom.compl} the support of any monomial in the $y_i$ is contained in $\sum_{i=1}^r F_i$, while in~\eqref{eqn:isom.compl.2} such monomials correspond to toric invariant elements, then it immediately follows that the $\sum_{i=1}^r F_i$ is a formally toric divisor, which concludes the proof.
\end{proof}

\subsubsection{The general case} 
We proceed to prove Proposition~\ref{prop:toric-plt-implies-toric} in full generality, i.e., without assuming that the $\clyx$ is torsion free.
In order to do so, we will reduce the general case, inductively, to the one in which ${\rm Cl}(X_x)$ is torsion free case by means of appropriate Galois coverings.
To this end, let us introduce the following notion of toric formality for finite group actions.
\begin{definition}
\label{def:finite-toric-act}
{\em 
Let $G$ be a finite abelian group.
Let $x \in X$ be a formally toric singularity and let $\xymatrix{
G \hspace{-.2cm} & \ar@(ul,dl)[]^{\rho} X
}$ be a faithful action fixing $x$.
We say that the action $\rho$ is a {\em formally toric action} if there exists a formal isomorphism $\phi \colon \widehat{X_x} \to \widehat{Z_{z}}$ to the completion of an affine toric variety $Z$ at a torus invariant point $z$ and a faithful action 
$\xymatrix{
Z \ar@(ur,dr)[]_{\rho'}& \hspace{-.5cm}G
}$
fixing $z$ such that the actions induced by $\rho$ and $\rho'$ at the level of completions at $x$ and $z$, respectively, are conjugate via $\phi$.
}
\end{definition}

\begin{proposition}\label{prop:local-toricity}
Let $(x \in X,\Gamma)$ be a germ of a klt pair. 
Assume that $(x \in X,\Gamma)$ admits a formally toric plt blow-up $ \pi \colon Y \to X$ of exceptional divisor $E$.
Then, $\widehat{X_x}\simeq 
\widehat{{\rm Cone}(E,-E|_E)_v}$ and the divisor $\sum_{i=1}^r F_i$ is mapped to the toric divisor of $\conev$.
In particular, the pair $(X, \sum_{i=1}^r F_i)$ is formally toric at $x$.

\end{proposition}

\begin{proof}
We proceed by induction order of the torsion subgroup of ${\rm Cl}(X_x)$.
If the torsion subgroup of ${\rm Cl}(X_x)$ is trivial, then Proposition~\ref{prop:no-torsion-toric} concludes the proof.
\\
Let $W$ be a torsion Weil divisor in $\clxx$ of order $l>0$.
We have the following commutative diagram 
\begin{align}
\label{eqn:diag.plt.blow.tors}    
\xymatrix{
\widetilde{Y} \ar[d]_{\widetilde{\pi}} \ar[r]^{\psi_{\widetilde{Y}}} &
Y \ar[d]^{\pi}
\\
\widetilde{X} \ar[r]^{\psi_{\widetilde{X}}} & X
}
\end{align}
where $\psi_{\widetilde{X}} \colon \widetilde{X} \to X$ is the index one cover of $W$ and $\widetilde{Y}$ is the normalization of the main component of $Y\times_X \widetilde{X}$.
The morphism $\psi_{\tilde X}$ is a $\zz_l$-Galois cover which is unramified in codimension one;
moreover, the fibre $\psi_{\widetilde{X}}^{-1}(x)$ over $x$ consists of a unique point $\tilde{x} \in \widetilde{X}$.
Then, $\widetilde{X}={\rm Spec}(\widetilde{R})$, $\widetilde{R}:=\oplus_{j=0}^{l-1}H^0(X, \mathcal O_X(-jW))$ and we shall denote by $\mathfrak{m}_{\tilde{x}} \subset \widetilde{R}$ be the maximal ideal corresponding to $\tilde{x}$.
Moreover, the cyclic Galois $\zz_l$ action on $\widetilde{X}$ induces by construction a $\zz_l$-action on $\widetilde{E}$ such that $\widetilde{E}/\zz_l = E$ and a $\zz_l$-action on $\reetilde$ such that $\reetilde^{\zz_l} = \ree$.
\\
Let $(\widetilde{Y},\widetilde{E})$ be the log pull-back of $(Y,E)$ on $\widetilde Y$.
As $(Y,E)$ is plt, the same holds for $(\widetilde{Y},\widetilde{E})$;
in particular, $\widetilde E$ is normal and irreducible, and 
\[
\psi_{\widetilde Y}^\ast (K_Y+E)= K_{\widetilde Y}+\widetilde E.
\]
Let $(\widetilde{X},\Gamma_{\widetilde{X}})$ (resp. $(\widetilde{X},\lceil \Gamma_{\widetilde{X}} \rceil)$) be the log pull-back of $(X,\Gamma)$
(resp. $(X,\gammaxred)$) on $\widetilde{X}$.
Let us recall that $\gammaxred=\sum_{i=1}^r F_i$ and $E_i=\pi_\ast^{-1} F_i$.
We shall denote by $\widetilde{F}_i$ the reduced support of $\psi^{-1}_{\widetilde{X}}(F_i)$, and $\widetilde{E}_i:= (\widetilde{\pi})^{-1}_\ast \widetilde{F}_i$.
Then,
\begin{align}
\label{eqn:log.pback.Ytilde}
\psi_{\widetilde Y}^\ast (K_Y+E+\sum_{i=1}^r E_i)= K_{\widetilde Y}+\widetilde E +\sum_{i=1}^r \widetilde E_i.
\end{align}
Defining $\Gamma_E$ (resp. $\Gamma_{\widetilde E}$) on E (resp. $\widetilde E$) by the adjunction formula
\[
(K_Y+E+\sum_{i=1}^r E_i)\vert_E= 
K_E+\Gamma_E, 
\quad 
(\text{resp. }
(K_{\widetilde Y}+\widetilde E +\sum_{i=1}^r \widetilde E_i) \vert_{\widetilde{E}}=
K_{\widetilde{E}}+
\Gamma_{\widetilde{E}}
\]
then~\eqref{eqn:log.pback.Ytilde} implies that
\begin{align}
\label{eqn:log.pback.Etilde}
 K_{\widetilde{E}}+
 \Gamma_{\widetilde{E}}=
 \psi_{\widetilde{X}}\vert_{\widetilde{E}}^\ast
 (K_E+\Gamma_E)
\end{align}
\noindent
{\it Claim 1.} 
The following properties hold for the varieties in~\eqref{eqn:diag.plt.blow.tors}:
\begin{enumerate}
    \item 
$\widetilde{E}$ is a projective toric variety;

    \item 
${\rm Cl}_\qq(\widetilde{E})\simeq {\rm Cl}_\qq(E)$;

    \item   
$\widetilde{Y}$ is $\qq$-factorial around $\widetilde{E}$ and $\widetilde{\pi}$ is a $\qq$-factorial plt blow for $(\widetilde X, \Gamma_{\widetilde{X}})$;

    \item 
$\widetilde{\pi}$ is a formally toric plt blow-up for $(\widetilde{X},\Gamma_{\widetilde{X}})$;

    \item 
$\rank\clxx 
= \rank {\rm Cl}(\widetilde{X}_{\tilde{x}})$;

   \item
$c_{\tilde{x}}(\widetilde{X}, \gammaxredp)=0$; 
and,

    \item 
$|{\rm Cl}(\widetilde{X}_{\tilde{x}})_{\rm tor}|
<
|{\rm Cl}(X_{x})_{\rm tor}|$.
\end{enumerate}

\begin{proof}[Proof of the Claim]
\begin{enumerate}
    \item 
By~\eqref{eqn:log.pback.Ytilde}, the finite morphism $\widetilde{E}\rightarrow E$ is unramified over the torus of $E$, which in turn implies the toricness of $\widetilde{E}$.
    \item 
The morphism $\widetilde{\pi}\vert_{\widetilde{E}} \colon \widetilde{E}\rightarrow E$ is a Galois quotient by a $\zz_l$ subgroup of the torus of $\widetilde{E}$.
Hence, every torus invariant divisor on $\widetilde{E}$ is invariant for such action and ${\rm Cl}_\qq(\widetilde{E})\simeq {\rm Cl}_\qq(E)$.

    \item 
By $(2)$, the classes $\psi_{\widetilde{Y}}^\ast E_i\vert_{\widetilde{E}}$ generate ${\rm Cl}_\qq(\widetilde{E})$ as a $\qq$-vector space, since each $E_i$ is supported at a unique torus invariant prime divisor $T_i \subset E$, all of the $T_i$ appear in the support of some $E_i$, and $\widetilde{E} \to E$ only ramifies along the divisors $T_i$. 
As the pair $(\widetilde{Y},\widetilde{E})$ is plt and $(\widetilde{X},\tilde{x})$ supports a klt singularity, then the $\mathbb Q$-factoriality of $\widetilde{Y}$ follows from~\cite[Proposition~12.1.4]{KM92} as in the proof of Lemma~\ref{lem:linear-equivalence}.
As 
$\widetilde{E} = \supp \psi^{\ast}_{\widetilde{Y}} E$,
then $-\widetilde{E}$ is nef over $\widetilde{X}$ and $\widetilde{\pi}$ is a $\qq$-factorial plt blowup

  \item 
By (1) and (3), it 
suffices to prove that each $\widetilde{E}_i$ is prime.
For the sake of contradiction, up to re-ordering the indices, let us assume that $\widetilde{E}_1$ is not prime.
We know that $\widetilde{E}_1\cap \widetilde E$ is supported at a prime toric divisor of $\widetilde{E}$ by (1).
As the log pair $(\widetilde{Y},\widetilde{E}_1+\widetilde E)$ is log canonical and $\qq$-factorial at the generic point of $\widetilde{E}_1\cap \widetilde E$, by~\cite[Theorem 18.22]{Kol92}, we conclude that $\widetilde{E}_1$ must be prime. 
The analogous argument applies to the other $\widetilde E_i$.
    
    \item 
The same proof as in Lemma~\ref{lem:q-linear}  applies here, since, by (4) $\widetilde{\pi}$ is a formally toric plt blow-up for $(\widetilde{X}, \Gamma_{\widetilde{X}})$.

   \item 
Indeed,
$|\gammaxredp|\geq |\gammaxred|$ and $\dim(X)=\dim(\widetilde{X})$.
By (5) then $c_{\tilde{x}}(\widetilde{X},\gammaxredp)=0$.
    
    \item 
The homomorphism
$\psi^\ast \colon {\rm Cl}(X,x)_{\rm tor}\rightarrow {\rm Cl}(\widetilde{X},\tilde{x})_{\rm tor}$ is surjective by Lemma~\ref{lem:p1-cyclic} since ${\rm Cl}(X,x)_{\rm tor}$ is isomorphic to the abelianization of $\pi_1^{\rm loc}(X, x)$, see~\cite[Corollary 4.15]{BM21}.
Thus, $|{\rm Cl}(\widetilde{X}_{\tilde{x}})_{\rm tor}|
\leq 
|{\rm Cl}(X_{x})_{\rm tor}|$ and the inequality is strict since $[W]$ lies in the kernel of $\psi_{\widetilde{X}}^\ast$.
\end{enumerate}
\end{proof}
\noindent
By properties (6) and (7) of the claim, and the inductive hypothesis, $(\tilde{x}\in \widetilde{X}, \sum_{i=1}^r \widetilde{F}_i)$ is a locally toric singularity and there exists an isomorphism
\begin{align}
\label{eqn:isom.tilde}
\widehat{\widetilde{X}_{\tilde{x}}} \simeq \conevtilde,
\end{align}
where $\tilde{v}$ is the vertex of $\coneetilde$.
Under the isomorphism in~\eqref{eqn:isom.tilde} the divisor $\sum_{i=1}^r \widetilde{F}_i$ is mapped, upon passing to the completion, onto the sum of the toric invariant divisor of $\coneetilde$.
Thus, the action of a formal torus exists in the completion $\widehat{\widetilde{X}_{\tilde{x}}}$ since $\widetilde x \in \widetilde X$ is a formally toric singularity.
Moreover, the (formal) torus is uniquely determined since $\sum_{i}\widetilde F_i$ is mapped to the full sum of the torus invariant divisors under the isomorphism in~\eqref{eqn:isom.tilde}.
\\ \\
{\it Claim 2.} The action of $\zz_l$ on $\widehat{\widetilde{X}_{\tilde{x}}}$ factors through the torus action.
\begin{proof}[Proof of Claim 2]
Indeed, $\widehat{\widetilde{X}_{\tilde{x}}}$ is isomorphic to
the spectrum of $\kk[[\underline{x}^{\underline{m_1}},\dots,\underline{x}^{\underline{m_k}}]]$ for certain monomials $\underline{x}^{\underline{m_i}}$ in the variables $x_1,\dots,x_r$.
Each monomial $\underline{x}^{\underline{m_1}}$ defines a Cartier divisor on the germ supported on the $\widetilde{F}_i$, hence its divisor of zeroes is fixed under the $\zz_l$-action on $\widetilde X$.
In particular, given $g$ a generator of $\zz_l$, $g\cdot \underline{x}^{m_1}=u\underline{x}^{m_1}$ for a certain unit $u \in \kk[[\underline{x}^{\underline{m_1}},\dots,\underline{x}^{\underline{m_k}}]]^\ast$.
Since $g$ has order $l$, we can conclude that $u^l=1$ and $u$ is a root of unity.
Hence, $\zz_l$ is a subgroup of the maximal torus of 
$\widehat{\widetilde{X}_{\widetilde{x}}}$.
\end{proof}
\noindent
By Claim 2, the $\zz_l$-action on $\widehat{\widetilde{X}_{\tilde{x}}}$ is formally toric in the sense of Definition~\ref{def:finite-toric-act}.
In particular, it is compatible with the $\zz_l$ action on the cone $\conevtilde$.
Given that $\widetilde{X}/\zz_l =X$ and under this identification $\tilde{x}$ is identified with $x$, we can conclude that
\begin{align}
\label{eqn:isom.compl3}
\widehat{X_x} 
= 
\widehat{\widetilde{X}_{\tilde{x}}}/\zz_l 
\simeq 
\conevtilde/\zz_l 
=
\conev.
\end{align}
\end{proof}

\begin{proof}[Proof of Proposition~\ref{prop:toric-plt-implies-toric}]
This immediately follows from Proposition~\ref{prop:local-toricity}.
\end{proof}

In the course of the proof of Proposition~\ref{prop:local-toricity}, we have used the following technical result.

\begin{lemma}\label{lem:p1-cyclic}
Let $Y\rightarrow X$ be a formally toric plt blow-up at $x\in X$.
Then, the local fundamental group
$\pi_1^{\rm reg}(X,x)$ is a finite abelian group. 
\end{lemma}

\begin{proof} 
The argument is essentially contained in~\cite[Proof of Theorem~4.11]{BFMS20}.
Indeed, as the normal sheaf $E \vert_E$ of $E$ in $Y$ is Cartier and trivializes over the torus $\mathbb{G}_m^{\dim E}$ of $E$, the commutative diagram contained in~\cite[(4.1)]{BFMS20} yields the split short exact sequence
\[
1\rightarrow 
\zz \rightarrow G \rightarrow 
\pi_1((\kk^*)^{\dim E}) \rightarrow 1.
\]
By~\cite[Proposition 4.9]{BFMS20}, 
the abelian group $G$ surjects onto $\pi_1^{\rm reg}(X,x)$.
Since $\pi_1^{\rm reg}(X,x)$ is finite, we conclude that it is a finite abelian group of rank at most $\dim(X)$.
\end{proof} 

In the next section, we will also use the following technical result showing that every log canonical place of the round-up of the boundary in a formally toric blow-up yields a formally toric log canonical place, see Definition~\ref{def:formally-toric-lc-place} for the relevant notion.

\begin{proposition}
\label{prop:toric.lc.places}
Let $(x \in X,\Gamma)$ be a germ of a klt pair. 
Assume that $(x \in X,\Gamma)$ admits a formally toric plt blow-up $ \pi \colon Y \to X$ at $x \in X$ and let $E$ be the exceptional divisor on $\pi$.
Let $F$ be an lc place of $(X,\lceil \Gamma \rceil)$, other than $E$, whose center contains $x$.
Then $F$ is a formally toric lc place for $(x \in X,\lceil \Gamma \rceil)$.
\end{proposition}

\begin{proof}
Let $\Gamma_Y$ be the strict transform of $\Gamma$ on $Y$;
we will denote its components by $E_i$, $i=1, \dots, r$.
\\
By Proposition~\ref{prop:toric-plt-implies-toric}, $(x \in X, \Gamma)$ is a formally toric singularity and $\pi$ is formally toric at $x$ with respect to the pair $(Y, \lceil \Gamma_Y \rceil + E)$.
In particular, there exists a diagram, cf.~\eqref{eqn:formally.toric.diag},
\begin{align}
\label{eqn:formally.toric.diag2}
\xymatrix{
\widehat{Y(\Xi)}_{z_0} := 
Y(\Xi) \times_{Z} \widehat{Z(\sigma)}_{z_0} 
 \ar[d]^{\tilde{f}_t} & & 
\widehat{Y}_x:= 
Y \times_X \widehat{X}_x \ar[d]^{\tilde{f}}
\ar[ll]_\psi
\\
\widehat{Z(\sigma)}_{z_0} & &
\widehat{X}_x
\ar[ll]^\phi
}
\end{align}
satisfying the conditions of Definition~\ref{def:formally-toric}.
In particular, the base change of the prime divisors $E_i$ and $E$ to $\widehat{Y_x}$ are mapped via $\psi$ to the base change of the torus invariant divisors $\Xi_1, \dots, \Xi_r, \Xi_E$ of $Y(\Xi)$ to $\widehat{Z(\sigma)_{z_0}}$.
\\
By definition of a formally toric plt blow-up, $(Y,\lceil \Gamma_Y \rceil + E)$ is qdlt along $E$, cf.~Remark~\ref{rem:gamma_i} and Lemma~\ref{lem:toric-complexity=0};
the same conclusion must holds also in a neighborhood of $E$ in $Y$.
By this observation and by condition (4) of Definition~\ref{def:form.tor.plt.blowup}, it follows that any lc center $C$ of $(Y,\lceil \Gamma_Y \rceil + E)$ is uniquely determined by those components among the $E_i$ and $E$ containing its generic point $\eta_C$; 
in particular, $C \cap E$ is irreducible.
Hence, on $Y$, any log canonical place $F$ of $(Y,\lceil \Gamma_Y \rceil + E)$ intersecting $E$ non-trivially can be extracted by the normalized blow-up of an ideal sheaf of the form
$\mathcal{I}_F:=
\mathcal{I}_{E_1}^{m_1} \cap \dots \cap \mathcal{I}_{E_r}^{m_r} \cap \mathcal I_{E}^{m_E}$, 
for certain non-negative integers $m_1, \dots, m_r, m_E$, see, for example,~\cite[\S 11.3]{CLS11}.
Hence, via~\eqref{eqn:formally.toric.diag2}, the ideal $\mathcal I_F$ is mapped to the ideal 
$\mathcal I_{\Xi_F} :=\mathcal{I}_{\Xi_1}^{m_1} \cap \dots \cap \mathcal{I}_{\Xi_r}^{m_r} \cap \mathcal I_{\Xi_E}^{m_E}$, upon passing to the appropriate completions.
Blowing up $Y(\Xi)$ along the toric ideal $I_{\Xi_F}$ extracts a toric valuation $\Xi_F$ which is an lc place of the torus invariant divisor $\Xi_1 + \dots + \Xi_r + \Xi_E$ on $Y(\Xi)$.
The existence of such valuation shows that $F$ is a formally toric lc place for the pair $(X, \lceil \Gamma \rceil)$.
\end{proof}

\section{Proof of the main results}

The aim of this section is to provide a proof of the following more general version of Theorem~\ref{formally toric}.

\begin{theorem}
\label{formally.toric}
Let $(X/Z,B+M)$ be a generalized log canonical pair and let $z\in Z$ be a closed point.
Assume that
\[
K_X+B+M\sim_{\qq, Z} 0.
\]
Then,
\[
\abscompz(X/Z,B+M)\geq \finecomp_z(X/Z,B+M)\geq
\orbcomp_z(X/Z,B+M) \geq 0, 
\]
and if the equality
\[
c_z(X/Z,B+M) =0 
\]
holds, then the following conditions are satisfied: 
\begin{enumerate}
    \item 
$M$ descends to a torsion divisor over a neighborhood of $z$;

    \item 
$X \rightarrow Z$ is formally toric at $z$ for $(X, \lfloor B\rfloor)$; 
and,

    \item 
if $\Sigma$ is an orbifold decomposition of $B$ such that $\orbcomp_z(X/Z, B; \Sigma)=0$, then all prime divisors 
appear in the support of $\Sigma$.
\end{enumerate}
\end{theorem}

\subsection{Local case}
We prove the following local version of Theorem~\ref{formally.toric}.

\begin{theorem}
\label{thm:local-case}
Let $(X,B)$ be a log canonical pair and let $x\in X$ be a closed point.
Assume that $\orbcomp_x(X,B)=0$. 
Then $(x \in X, \lfloor B\rfloor)$ is a formally toric singularity.
Moreover, the components of $B$ span $\clxxq$.
\end{theorem}

\begin{proof}
Let $(X,B)$ be a log canonical pair such that $\orbcomp_x(X,B)=0$.
By Corollary~\ref{cor:locally-M-descends}, 
there exists a boundary $\Gamma$ on $X$ such that 
\begin{itemize}
    \item 
$0\leq \Gamma \leq \lceil B \rceil$, $\lfloor B \rfloor \leq \lceil \Gamma \rceil$;
    \item 
$(X,\Gamma)$ is klt; 
    \item 
$(X, \Gamma)$ admits a formally toric plt blow-up $\pi \colon Y \rightarrow X$; and,
    \item 
$\pi$ extracts a log canonical place $E$ of $(X,\lceil \Gamma \rceil)$ over $x\in X$.
\end{itemize}
The conclusion then follows by Proposition~\ref{prop:toric-plt-implies-toric}.
To prove the last part of the statement, it suffices to notice that the components of $\Gamma$ span $\clxxq$, by Corollary~\ref{cor:generation}.
\end{proof}

As an immediate corollary, we have the following result.

\begin{corollary}\label{cor:lcc-are-formally toric}
Let $(X,B+M)$ be a generalized log canonical pair and let $x\in X$ be a closed point.
Assume that $\orbcomp_x(X,B+M)=0$.
Then, any log canonical place of $(X,B+M)$ is formally toric for the formally toric structure provided by Theorem~\ref{thm:local-case}.
\end{corollary}

\begin{proof}
By Corollary~\ref{cor:locally-M-descends}, $M$ descends on a neighborhood of $x \in X$ and we may drop $M$ to simply consider the log canonical pair $(X,B)$.
Moreover, the log canonical places of $(X,B+M)$ are exactly the log canonical places of $(X,B)$, and $\orbcomp_x(X,B)=0$.
By Theorem~\ref{thm:local-case}, $(x\in X, \lfloor B \rfloor)$ is a formally toric singularity.
\\
Let $\Gamma$ be the boundary whose existence is guaranteed by Corollary~\ref{cor:locally-M-descends}.
$\Gamma$ satisfies the following properties:
\begin{itemize}
    \item 
$0\leq \Gamma_X \leq \lceil B \rceil$, $\lfloor B \rfloor \leq \lceil \Gamma_X \rceil$;
    \item 
$(X,\Gamma_X)$ is klt; 
    \item 
$(X, \Gamma_X)$ admits a formally toric plt blow-up $\pi \colon Y \rightarrow X$; and,
    \item 
$\pi$ extracts a log canonical place $E$ of $(X,\lceil \Gamma_X \rceil)$ over $x\in X$.
\end{itemize}
By the proof of Proposition~\ref{prop:toroida-blow-up}, cf.~Step 0 there, $E$ is a log canonical center of $(X,B)$.
By Proposition~\ref{prop:local-toricity}, the completion $\widehat{X_x}$ of $X$ at $x$ is isomorphic to the completion of a toric cone $C:={\rm Spec}(\oplus_{l=0}^\infty H^0(E, lE\vert_E))$ at the vertex $v$ corresponding to the maximal ideal $\oplus_{l>0}^\infty H^0(E, lE\vert_E)$ of the section ring.
\\
Let $F$ be an lc place of $(X,B)$ whose center passes through $x\in X$: there exists a projective birational morphism $t \colon Y'\rightarrow Y$ such that $Y'$ is $\qq$-factorial and $t$ only extracts $F$, cf.~\cite[Theorem 1]{Mor19}.
Let $E'$ be the strict transform of $E$ on $Y'$.
Thus,
\[
K_{Y'}+ B_{Y'}=(\pi \circ t)^\ast(K_X+B), \quad 
B_{Y'}:= (\pi \circ t)^{-1}_\ast B + E' + F,
\]
and $\orbcomp_x(Y'/X, B_{Y'})=0$, by Lemma~\ref{complexity-dlt}.
Furthermore, we can assume that there exists a decomposition $\Sigma$ of $B_{Y'}$ such that $\orbcomp_x(Y'/X, B_{Y'}; \Sigma)=0$ and $E', F$ are orbifold Weil divisors of coefficient $1$ in the decomposition $\Sigma$.
By the connectedness of log canonical centers,~\cite[Theorem 1.1]{FS20}, $F \cap E' \neq \emptyset$.
Thus, by adjunction along $E'$, there exists a log canonical pair $(E', B'_{E'})$ such that
\[
(K_{Y'}+B_{Y'})|_{E'} =
K_{E'}+ B_{E'}.
\]
{\it Claim}. 
The prime divisor $E'$ is a normal toric variety and the support of $F\vert_{E'}$ is torus invariant on $E'$.
\begin{proof}[Proof of the Claim]
To prove normality of $E'$, it suffices to show that $E'$ is smooth in codimension 1, since $E'$ is a divisorial lc center of a log canonical pair.
Away from $F\cap E'$, certainly $E'$ is normal, since $r \vert_{X\setminus F}$ is an isomorphism.
On the other hand, along the codimension $2$ points of $E'\cap F \subset Y'$, $E'$ is normal by adjunction, since $(Y', B_{Y'})$ is lc and $B_{Y'} \geq E'+F$.
Hence, $t':=t\vert_{E'} \colon E' \to E$ is a birational morphism of normal varieties, and 
\[
(t')^\ast(K_E+B_E)= K_{E'}+B_{E'}, 
\]
where the divisor $B_E$ is  defined by $K_E+B_E=(K_Y+E+\pi_\ast^{-1}B)\vert_E$.
By construction, $\orbcomp(E, B_E)=0$ and since $t'$ only extracts prime divisors $G_1, \dots, G_h$ contained in $F \cap E'$, so that $B_{E'} \geq \sum_{i=1}^h G_i$, then Lemma~\ref{complexity-dlt} implies also that $\orbcomp(E', B_{E'})=0$. 
By Theorem~\ref{thm:formally toric-projective}, $E'$ must be toric and the $G_i$ are toric prime divisors.
To show that $F \cap E'$ is torus invariant, it suffices to notice that since $B_{Y'} \geq E'+F$, then any component of $F \cap E'$ appears with coefficient $1$ in $B_{E'}$.
\end{proof}
\noindent
Writing the log pullback formula
\begin{align}
\label{eqn:adj.r'.1}
K_{Y'}+ 
(\pi \circ t)^{-1}_\ast \lceil \Gamma_X \rceil+
E'+(1-a)F=
(\pi \circ t)^\ast(K_X+\lceil \Gamma_X \rceil),
\end{align}
then $a \geq 0$ and if $a=0$, $F$ is a lc center of $(X, \lceil \Gamma_X \rceil)$.
By construction, on $E'$
\begin{align}
\label{eqn:adj.r'.2}
(t')^\ast((K_Y+E+\pi_\ast^{-1}\lceil \Gamma_X \rceil)\vert_E) &= 
(t^\ast
(K_Y+E+\pi_\ast^{-1}\lceil \Gamma_X \rceil))\vert_{E'},
\quad \text{and}
\\
(K_Y+E+\pi_\ast^{-1}\lceil \Gamma_X \rceil)\vert_{E} & = 
K_E+\sum_{i=1}^r T_i,
\end{align}
where the $T_i$ are all torus invariant divisors on $E$.
By the Claim above, 
\begin{align}
\label{eqn:adj.r'.3}
(t')^\ast (K_E+\sum_{i=1}^r T_i)=
K_{E'}+(t')_\ast^{-1}\sum_{i=1}^r T_i+G,
\end{align}
where $G$ is the torus invariant divisor on $E'$ supporting $F \cap E'$.
As $\Gamma_X$ is supported on the components of $B$ and $K_{Y'}+B_{Y'}$ is log canonical, then, writing,
\[
(K_{Y'}+ 
(\pi \circ r)^{-1}_\ast \lceil \Gamma_X \rceil+
E'+(1-a)F)\vert_{E'}
=
K_{E'}+\Delta_{E'}
\]
the coefficient of each component of $G$ in $\Delta_{E'}$ is $1$ if and only if $a=0$.
Combining this observation with \eqref{eqn:adj.r'.1}-\eqref{eqn:adj.r'.3}, then $a=0$ and, thus,
$F$ is an lc place of $(X,\lceil \Gamma_X \rceil)$.
Hence, $F$ is a formally toric lc place of $(X,\lceil \Gamma_X \rceil)$ by Proposition~\ref{prop:toric.lc.places}.
\end{proof}

\subsection{Birational case}
\label{subsection:bir-case}

We prove the following birational version of Theorem~\ref{formally.toric}.

\begin{theorem}\label{thm:formally toric-birational}
Let $X\rightarrow Z$ be a birational contraction and $z\in Z$ a closed point.
Let $(X,B+M)$ be a generalized log canonical pair.
Assume that $K_X+B+M\sim_{\mathbb Q, Z} 0$ and $c_z(X/Z,B+M)=0$.
Then $X\rightarrow Z$ is formally toric over $z\in Z$
with respect to the pair  $(X, \lfloor B \rfloor)$.
\end{theorem}

\begin{proof}
Let $B_Z$ be the pushforward of $B$ on $Z$.
By Proposition~\ref{prop:toroida-blow-up}, 
there exists an effective divisor $\Gamma_Z$ on $Z$ satisfying the following properties:
\begin{itemize}
    \item 
$0\leq \Gamma_Z \leq \lceil B_Z \rceil$, $\lfloor B_Z \rfloor \leq \lceil \Gamma_Z \rceil$;
    \item 
$(Z,\Gamma_Z)$ is klt; 
    \item 
$(Z, \Gamma_Z)$ admits a formally toric plt blow-up $\pi \colon Y \rightarrow Z$; and,
    \item 
$\pi$ extracts a log canonical place $E$ of $(Z,\lceil \Gamma_Z \rceil)$ over $z$.
\end{itemize}
Moreover, $M$ descends over a neighborhood of $z\in Z$, so that $(X,B)$ is log canonical on $X$ and $K_X+B\sim_{\qq, Z} 0$.
\\
By Lemma~\ref{lem:cut-down-lcc}, there exists an lc center of $(X, B)$ contained in the fiber over 
$z$.
Possibly passing to an appropriate dlt modification of $(X, B)$, we may assume that the fiber on $X$ over $z$ contains a a divisorial lc place $E$ of $(X,B)$.
By Lemma~\ref{complexity-dlt}.1, the latter operation will not affect $c_z(X/Z, B)$.
Then, by~\cite[Theorem 1.1]{Bir12}, we may run a minimal model program for $-E$ over $Z$ with scaling of an ample divisor and that will terminate with a good minimal model $Y$ over $Z$.
\\
By the proof of Proposition~\ref{prop:toroida-blow-up}, the birational morphism $Y\rightarrow Z$ is of Fano type: 
such model will be a $\qq$-factorial plt blow-up of $(Z, \Gamma_Z)$.
Running the $E_Y$-MMP on $Y$ over $Z$, where $E_Y$ is the strict transform of $E$ on $Y$, that will terminate with a small $\qq$-factorialization $Y' \to Z$ of $Z$.
By Corollary~\ref{complexity.MMP}.\ref{complexity-under-MMP}, since $X \dashrightarrow Y'$ over $Z$ is a composition of divisorial contractions of lc centers of $(Z, B_Z)$ and small maps, then $\abscompz(Y'/Z, B_{Y'})=0$, where $B_{Y'}$ is the strict transform of $B$ on $Y'$. 
Hence, $(Y',B_{Y'})$ is a $\qq$-factorial log canonical pair together with a small morphism $Y'\rightarrow Z$ and $\orbcomp_z(Y'/Z,B_{Y'})=0$.
This implies also that $\orbcomp_z(Z,B_Z)=0$.
\newline
By Theorem~\ref{thm:local-case}, we conclude that $(z \in Z, \lfloor B_Z \rfloor)$ is a formally toric germ.
Hence, by Lemma~\ref{formally toricity-vs-MMP}, $Y'\rightarrow Z$ is a formally toric morphism, being a small $\qq$-factorialization of a formally toric germ.
Moreover, $\lfloor B_{Y'}\rfloor$ is formally toric for the morphism $Y' \to Z$.
By Corollary~\ref{cor:lcc-are-formally toric}, every lc place of $(Z,B_{Z})$ is formally toric; 
thus, $X\rightarrow Z$ only extract formally toric lc places of $(Z,B_{Z})$.
Again, by Lemma~\ref{formally toricity-vs-MMP}, then $X\rightarrow Z$ is formally toric and $\lfloor B \rfloor$ is formally toric as well.
\end{proof}

\subsection{Fibration case}
\label{subsection:fib-case}

In this subsection, we shall prove Theorem~\ref{formally.toric} when $X \to Z$ is a fibration. More precisely, following statement.

\begin{theorem}
\label{thm:formally toric-fibration}
Let $\pi \colon X\rightarrow Z$ be a fibration of normal varieties.
Let $z\in Z$ be a closed point.
Let $(X,B+M)$ be a generalized log canonical pair.
Assume that $K_X+B+M\sim_{\mathbb Q, Z} 0$ and $c_z(X/Z,B)=0$.
Then $\pi$ is formally toric over $z\in Z$ with respect to the pair $(X, \lfloor B \rfloor)$ and $\mathbf{M}$ is trivial over a neighborhood of $z\in Z$.
\end{theorem}

\begin{proof}
Let $A$ be a divisor on $X$ which is relatively very ample over $Z$.
Let $C:=C_A(X/Z)$ be the relative cone of $X$ over $Z$, with respect
to the relative polarization given by $A$ -- that is, $C$ is the relative spectrum over $Z$ of the $\mathcal O_Z$-algebra $\bigoplus_{m \geq 0} \pi_\ast \mathcal O_X(mA)$.
By construction, $C$ has a natural $\cc^\ast$-action given by the grading on $\bigoplus_{m \geq 0} \pi_\ast \mathcal O_X(mA)$.
Let $c \colon C \to X$ be the associated morphism:
$c$ admits a section $\sigma \colon Z \to C$ which is $\cc^\ast$-invariant.
Let $C_\infty$ the relative spectrum over $X$ of the $\mathcal O_X$-algebra $\oplus_{m\geq 0}  \mathcal{O}_X(mA)$.
There exists a natural morphism $\gamma \colon C_\infty \to C$ which contracts the prime divisor $X_\infty \subset C_\infty$ isomorphic to $X$ down to $\sigma(Z)$.
We denote $z_C:=\sigma(z) \in C$ and $c_\infty \colon C_\infty \to X$ the structure morphism.
\\
We wish to define a generalized pair on $C$.
We proceed as follows:
letting $B= \sum_{i=1}^r b_i B_i$ be the decomposition of $B$ into its prime components, then we define the boundary $B_C:= \sum_{i=1}^r b_i B_{i, C}$ where $B_{i, C}:= c^{-1}(B_i)$.
For the nef part:
let $X'\rightarrow X$ be a model where $M$ descends and let $M'$ be the trace of the b-divisor on $X'$ so that $M'$ is nef over $Z$;
let $\pi'\colon X'\rightarrow Z$ the contraction obtained composing $\pi$ with $X' \to X$.
Let $A'$ be the pull-back of $A$ to $X'$ and let $C'_\infty$ be the relative spectrum over $X'$ of the $\mathcal O_{X'}$-algebra $\bigoplus_{m\geq 0} \mathcal{O}_{X'}(mA')$. 
Let $M_{C'_\infty}$ be the pull-back of $M'$ to $C'_\infty$ and let $M_{C_\infty}$ the pushforward of $M_{C'_\infty}$ to $C_\infty$.
\\
By construction, defining $B_{C_\infty}:=\sum_{i=1}^r b_i B_{i, C_\infty}$, $B_{i, C_\infty}:=c_\infty^{-1}(B_i)$, then $(C_\infty, X_\infty+ B_{C_\infty} + M_{C_\infty})$ is a generalized log canonical pair and
\[
K_{C_\infty}+X_\infty+ B_{C_\infty} + M_{C_\infty}\vert_{X_\infty} = 
c_\infty\vert_{X_\infty}^\ast(K_X+B+M)
\]
since $c_\infty\vert_{X_\infty} \colon X_\infty \to X$ is an isomorphism by construction.
Hence, 
\[
K_{C_\infty}+X_\infty+ B_{C_\infty} + M_{C_\infty}\vert_{X_\infty} 
\sim_{\qq, Z} 0,
\]
which in turn implies that $(C, B_C+M_C)$ is a generalized log canonical pair for which
the subvariety $\sigma(Z)$ is a log canonical center since the birational morphism $C_\infty \to C$ contracts the divisor $X_\infty$ to $\sigma(Z)$ and $K_{C_\infty}+X_\infty+ B_{C_\infty} + M_{C_\infty}\vert_{X_\infty}$ is torsion along all fibers of such contraction.
Moreover,
\begin{align}
&|B_C|=|B|, \qquad \dim C = \dim X +1, \quad \text{ and } 
\\
\label{eqn:relpic.c.infty}
&\dim_\qq\relpicxz -1 \geq \dim_\qq{\rm Cl}_\qq(C_{z_C}).
\end{align}
To prove that~\eqref{eqn:relpic.c.infty} holds, first let us notice that, since $C_\infty\rightarrow X$ is a $\mathbb{A}^1$-bundle, then
$\dim_\qq {\rm Cl}_\qq (C_\infty)= \dim_\qq {\rm Cl}_\qq(X)$;
moreover, the morphism $\gamma$ only contracts the prime divisor $X_\infty$, while the pullback of any Cartier divisor on $Z$ to $C$ is trivial in a neighborhood of $z_C$.
Thus, $\orbcomp_{z_C}(C,B_C+M_C)\leq 0$, and so by Corollary~\ref{cor:compl.pt} equality must hold.
By Theorem~\ref{thm:local-case}, we conclude that 
\begin{itemize}
    \item[(a)] 
$(C, \lfloor B_C\rfloor)$ is formally toric pair at $z_C$,  and 
    \item[(b)]   
$M_C$ descends over a neighborhood of $z_C$.
\end{itemize}
By Corollary~\ref{cor:lcc-are-formally toric}, $X_\infty$ is a formally toric lc place at $z_C$, thus, Lemma~\ref{formally toricity-vs-MMP} implies that $\gamma \colon C_\infty \to C$ is formally toric around $z_C$ with respect to the pair $(C_\infty, X_\infty + \lfloor B_{C_{\infty}} \rfloor)$.
Then, since we have the following commutative diagram
\begin{align}
\label{diag:comm.x.infty}
\xymatrix{
X_\infty \ar[r]^{c_\infty\vert_{X_\infty}} \ar[d]_{\gamma\vert_{X_\infty}}
&
X \ar[d]^\pi
\\
z_C \in \sigma(Z) \ar[r]_{\sigma^{-1}}
&
Z \ni z
}
\end{align}
and the morphism $\gamma\vert_{X_\infty}$ in the LHS column of~\eqref{diag:comm.x.infty} is formally toric by the above discusssion, then $\pi \colon X \to Z$ is formally toric around $z$;
furthermore, as the restriction of $\lfloor B_{C_{\infty}} \rfloor$ to $X_\infty$ is formally toric and it is mapped to $\lfloor B \rfloor$ via $c_\infty\vert_{X_\infty}$, we can conclude that $\pi$ is formally toric with respect to the pair $(X, \lfloor B \rfloor)$.
Moreover, conclusion (b) above and the commutative diagram in~\eqref{diag:comm.x.infty} imply that $M$ must descend over a neighborhood of $z$, since, by construction, 
\[
c_\infty\vert_{X_{\infty}}^\ast M = M_{C_{\infty}} \vert X_\infty.
\]
\end{proof}

\subsection{Conclusion}

\begin{proof}[Proof of Theorem~\ref{formally.toric}]
Let $(X/Z,B+M)$ be a generalized log canonical pair, such that 
$K_X+B+M\sim_{\qq,Z}0$.
Let $z\in Z$ be a closed point and $\Sigma$ be an orbifold decomposition of $B$.
By Theorem~\ref{thm:inequality-cgeq0}, 
$\orbcomp_z(X/Z,B+M;\Sigma)\geq 0$ 
and Lemma~\ref{lem:basic.ineq} implies that 
$\abscompz(X/Z,B+M)\geq \finecomp_z(X/Z,B+M)
\geq 
\orbcomp_z(X/Z,B+M;\Sigma)$.
\\
Now, we turn to prove the characterization of morphism
with complexity zero.
If $X\rightarrow Z$ is formally toric over $z\in Z$, 
then the toric boundary has complexity zero.
Hence, it suffices to prove that absolute complexity zero implies both, that $X\rightarrow Z$ is formally toric over $z\in Z$ and $M$ descends over a neighborhood of $z$.
Thus, we can assume that $c_z(X/Z,B+M)=0$.
If $X\rightarrow Z$ is the identity, 
then the statement of the theorem follows from Theorem~\ref{thm:local-case}.
If $X\rightarrow Z$ is a birational morphism,
the statement follows from Theorem~\ref{thm:formally toric-birational}.
Finally, if $X\rightarrow Z$ is a fibration, then  Theorem~\ref{thm:formally toric-fibration} concludes the proof.
\end{proof}

\begin{proof}[Proof of Theorem~\ref{weak thm}]
This follows at once from Theorem~\ref{formally.toric} simply by taking $M=-(K_X+B)$.
\end{proof}

\begin{proof}[Proof of Theorem~\ref{thm:intro.local.case}]
It suffices to notice that by Definition~\ref{def:orb.compl.loc}
$\dim X +
\rank {\rm Cl}(X_x) -
\sum_{i=1}^n b_i= c_{x}(X, B)$.
Thus, Corollary~\ref{cor:compl.pt} implies that $c_x(X, B) \geq 0$ and if equality holds, then the formal toricness of $(X, \lfloor B \rfloor)$ at $x$ follows from Theorem~\ref{thm:local-case}.
\end{proof}

\bibliographystyle{HH}
\bibliography{bib}

\begin{thebibliography}{10}
\expandafter\ifx\csname url\endcsname\relax
  \def\url#1{\texttt{#1}}\fi
\expandafter\ifx\csname doi\endcsname\relax
  \def\doi#1{\burlalt{doi:#1}{http://dx.doi.org/#1}}\fi
\expandafter\ifx\csname urlprefix\endcsname\relax\def\urlprefix{URL }\fi
\expandafter\ifx\csname href\endcsname\relax
  \def\href#1#2{#2}\fi
\expandafter\ifx\csname burlalt\endcsname\relax
  \def\burlalt#1#2{\href{#2}{#1}}\fi

\bibitem{Amb16}
F.~Ambro.
\newblock Cyclic covers and toroidal embeddings.
\newblock {\em Eur. J. Math.}, 2(1):9--44, 2016.
\newblock \doi{10.1007/s40879-015-0084-y}.

\bibitem{Art69}
M.~Artin.
\newblock Algebraic approximation of structures over complete local rings.
\newblock {\em Inst. Hautes \'{E}tudes Sci. Publ. Math.}, 1(36):23--58, 1969.
\newblock \urlprefix\url{http://www.numdam.org/item?id=PMIHES_1969__36__23_0}.

\bibitem{ADHL15}
I.~Arzhantsev, U.~Derenthal, J.~Hausen, and A.~Laface.
\newblock {\em Cox rings}, volume 144 of {\em Cambridge Studies in Advanced
  Mathematics}.
\newblock Cambridge University Press, Cambridge, 2015.

\bibitem{Arz09}
I.~V. Arzhantsev.
\newblock On the factoriality of {C}ox rings.
\newblock {\em Mat. Zametki}, 85(5):643--651, 2009.
\newblock \doi{10.1134/S0001434609050022}.

\bibitem{Bir12}
C.~Birkar.
\newblock Existence of log canonical flips and a special {LMMP}.
\newblock {\em Publ. Math. Inst. Hautes \'{E}tudes Sci.}, 115:325--368, 2012.
\newblock \doi{10.1007/s10240-012-0039-5}.

\bibitem{BCHM10}
C.~Birkar, P.~Cascini, C.~D. Hacon, and J.~McKernan.
\newblock Existence of minimal models for varieties of log general type.
\newblock {\em J. Amer. Math. Soc.}, 23(2):405--468, 2010.
\newblock \doi{10.1090/S0894-0347-09-00649-3}.

\bibitem{BZ16}
C.~Birkar and D.-Q. Zhang.
\newblock Effectivity of {I}itaka fibrations and pluricanonical systems of
  polarized pairs.
\newblock {\em Publ. Math. Inst. Hautes \'{E}tudes Sci.}, 123:283--331, 2016.
\newblock \doi{10.1007/s10240-016-0080-x}.

\bibitem{BFMS20}
L.~Braun, S.~Filipazzi, J.~Moraga, and R.~Svaldi.
\newblock The jordan property for local fundamental groups, 2020,
  \burlalt{2006.01253}{http://arxiv.org/abs/2006.01253}.

\bibitem{BM21}
L.~Braun and J.~Moraga.
\newblock Iteration of cox rings of klt singularities, 2021,
  \burlalt{2103.13524}{http://arxiv.org/abs/2103.13524}.

\bibitem{BMSZ18}
M.~V. Brown, J.~McKernan, R.~Svaldi, and H.~R. Zong.
\newblock A geometric characterization of toric varieties.
\newblock {\em Duke Math. J.}, 167(5):923--968, 2018.
\newblock \doi{10.1215/00127094-2017-0047}.

\bibitem{Corti}
A.~Corti, editor.
\newblock {\em Flips for 3-folds and 4-folds}, volume~35 of {\em Oxford Lecture
  Series in Mathematics and its Applications}.
\newblock Oxford University Press, Oxford, 2007.
\newblock \doi{10.1093/acprof:oso/9780198570615.001.0001}.

\bibitem{CLS11}
D.~A. Cox, J.~B. Little, and H.~K. Schenck.
\newblock {\em Toric varieties}, volume 124 of {\em Graduate Studies in
  Mathematics}.
\newblock American Mathematical Society, Providence, RI, 2011.
\newblock \doi{10.1090/gsm/124}.

\bibitem{dFKX}
T.~de~Fernex, J.~Koll\'{a}r, and C.~Xu.
\newblock The dual complex of singularities.
\newblock In {\em Higher dimensional algebraic geometry---in honour of
  {P}rofessor {Y}ujiro {K}awamata's sixtieth birthday}, volume~74 of {\em Adv.
  Stud. Pure Math.}, pages 103--129. Math. Soc. Japan, Tokyo, 2017.
\newblock \doi{10.2969/aspm/07410103}.

\bibitem{Fil18}
S.~Filipazzi.
\newblock On a generalized canonical bundle formula and generalized adjunction.
\newblock {\em Ann. Sc. Norm. Super. Pisa Cl. Sci.}, 21(5):1187--1221, 2020.

\bibitem{FS20}
S.~Filipazzi and R.~Svaldi.
\newblock On the connectedness principle and dual complexes for generalized
  pairs, 2020, \burlalt{2010.08018}{http://arxiv.org/abs/2010.08018}.

\bibitem{FG14}
O.~Fujino and Y.~Gongyo.
\newblock Log pluricanonical representations and the abundance conjecture.
\newblock {\em Compos. Math.}, 150(4):593--620, 2014.
\newblock \doi{10.1112/S0010437X13007495}.

\bibitem{GHS16}
M.~Gross, P.~Hacking, and B.~Siebert.
\newblock Theta functions on varieties with effective anti-canonical class,
  2019, \burlalt{1601.07081}{http://arxiv.org/abs/1601.07081}.

\bibitem{EGA1}
A.~Grothendieck.
\newblock \'{E}l\'{e}ments de g\'{e}om\'{e}trie alg\'{e}brique. {II}. \'{E}tude
  globale \'{e}l\'{e}mentaire de quelques classes de morphismes.
\newblock {\em Inst. Hautes \'{E}tudes Sci. Publ. Math.}, 1(8):222, 1961.
\newblock \urlprefix\url{http://www.numdam.org/item?id=PMIHES_1961__8__222_0}.

\bibitem{HM13}
C.~D. Hacon and J.~McKernan.
\newblock The {S}arkisov program.
\newblock {\em J. Algebraic Geom.}, 22(2):389--405, 2013.
\newblock \doi{10.1090/S1056-3911-2012-00599-2}.

\bibitem{Har77}
R.~Hartshorne.
\newblock {\em Algebraic geometry}.
\newblock Graduate Texts in Mathematics, No. 52. Springer-Verlag, New
  York-Heidelberg, 1977.

\bibitem{Kol92}
K.~J. et~al., editors.
\newblock {\em Flips and abundance for algebraic threefolds}.
\newblock Soci\'{e}t\'{e} Math\'{e}matique de France, Paris, 1992.
\newblock Papers from the Second Summer Seminar on Algebraic Geometry held at
  the University of Utah, Salt Lake City, Utah, August 1991, Ast\'{e}risque No.
  211 (1992) (1992).

\bibitem{KO15}
Y.~Kawamata and S.~Okawa.
\newblock Mori dream spaces of {C}alabi-{Y}au type and log canonicity of {C}ox
  rings.
\newblock {\em J. Reine Angew. Math.}, 701:195--203, 2015.
\newblock \doi{10.1515/crelle-2013-0029}.

\bibitem{KM99}
S.~Keel and J.~McKernan.
\newblock Rational curves on quasi-projective surfaces.
\newblock {\em Mem. Amer. Math. Soc.}, 140(669):viii+153, 1999.
\newblock \doi{10.1090/memo/0669}.

\bibitem{K13}
J.~Koll\'{a}r.
\newblock {\em Singularities of the minimal model program}, volume 200 of {\em
  Cambridge Tracts in Mathematics}.
\newblock Cambridge University Press, Cambridge, 2013.
\newblock \doi{10.1017/CBO9781139547895}.
\newblock With a collaboration of S\'{a}ndor Kov\'{a}cs.

\bibitem{KM92}
J.~Koll\'{a}r and S.~Mori.
\newblock Classification of three-dimensional flips.
\newblock {\em J. Amer. Math. Soc.}, 5(3):533--703, 1992.
\newblock \doi{10.2307/2152704}.

\bibitem{KM98}
J.~Koll\'{a}r and S.~Mori.
\newblock {\em Birational geometry of algebraic varieties}, volume 134 of {\em
  Cambridge Tracts in Mathematics}.
\newblock Cambridge University Press, Cambridge, 1998.
\newblock \doi{10.1017/CBO9780511662560}.
\newblock With the collaboration of C. H. Clemens and A. Corti, Translated from
  the 1998 Japanese original.

\bibitem{Laz04a}
R.~Lazarsfeld.
\newblock {\em Positivity in algebraic geometry. {I}}, volume~48 of {\em
  Ergebnisse der Mathematik und ihrer Grenzgebiete. 3. Folge. A Series of
  Modern Surveys in Mathematics [Results in Mathematics and Related Areas. 3rd
  Series. A Series of Modern Surveys in Mathematics]}.
\newblock Springer-Verlag, Berlin, 2004.
\newblock \doi{10.1007/978-3-642-18808-4}.
\newblock Classical setting: line bundles and linear series.

\bibitem{LS13}
A.~Liendo and H.~S\"{u}ss.
\newblock Normal singularities with torus actions.
\newblock {\em Tohoku Math. J. (2)}, 65(1):105--130, 2013.
\newblock \doi{10.2748/tmj/1365452628}.

\bibitem{Mor19}
J.~Moraga.
\newblock Extracting non-canonical places.
\newblock {\em Adv. Math.}, 375:107415, 12, 2020.
\newblock \doi{10.1016/j.aim.2020.107415}.

\bibitem{Mor20b}
J.~Moraga.
\newblock Fano type surfaces with large cyclic automorphisms, 2020,
  \burlalt{2001.03797}{http://arxiv.org/abs/2001.03797}.

\bibitem{Mor20c}
J.~Moraga.
\newblock Kawamata log terminal singularities of full rank, 2021,
  \burlalt{2007.10322}{http://arxiv.org/abs/2007.10322}.

\bibitem{Pro01}
Y.~G. Prokhorov.
\newblock On a conjecture of {S}hokurov: characterization of toric varieties.
\newblock {\em Tohoku Math. J. (2)}, 53(4):581--592, 2001.
\newblock \doi{10.2748/tmj/1113247802}.

\bibitem{Sho00}
V.~V. Shokurov.
\newblock Complements on surfaces.
\newblock In {\em Complements on surfaces}, volume 102, pages 3876--3932. J.
  Math. Sci. (New York), 2000.
\newblock \doi{10.1007/BF02984106}.
\newblock Algebraic geometry, 10.

\bibitem{MR3187625}
C.~Xu.
\newblock Finiteness of algebraic fundamental groups.
\newblock {\em Compos. Math.}, 150(3):409--414, 2014.
\newblock \doi{10.1112/S0010437X13007562}.

\bibitem{Yao13}
Y.~Yao.
\newblock A criterion for toric varieties, 2013.

\end{thebibliography}

\end{document}